\documentclass[aop,preprint]{imsart}
\setattribute{journal}{name}{}

\RequirePackage[OT1]{fontenc}
\RequirePackage{amsthm,amsmath,amsfonts,amssymb,mathrsfs}
\RequirePackage[numbers]{natbib}
\RequirePackage[colorlinks,citecolor=blue,urlcolor=blue]{hyperref}
\usepackage{tikz}
\usetikzlibrary{arrows,calc}
\usepackage{bbm}
\usepackage{enumerate}

\usepackage{mathtools}
\mathtoolsset{showonlyrefs}

\usepackage{environ}
\NewEnviron{eq}{%
\begin{equation}\begin{split}
  \BODY
\end{split}\end{equation}
}

\startlocaldefs

\def\sss{\scriptscriptstyle}

\newcommand{\var}[1]{\ensuremath{\mathrm{Var}\left(#1\right)}}

\newcommand{\ind}[1]{\ensuremath{\mathbbm{1}{\left\{#1\right\}}}}

\newcommand{\dto}{\ensuremath{\xrightarrow{d}}}

\newcommand{\PR}{\ensuremath{\mathbbm{P}}}
\newcommand{\E}{\ensuremath{\mathbbm{E}}}
\newcommand{\R}{\ensuremath{\mathbbm{R}}}
\newcommand{\N}{\ensuremath{\mathbbm{N}}}
\newcommand{\e}{\ensuremath{\mathrm{e}}}
\newcommand{\1}{\ensuremath{\mathbbm{1}}}

\newcommand{\CM}{\ensuremath{\mathrm{CM}_n(\boldsymbol{d})}}
\newcommand{\BCM}{\ensuremath{\mathrm{BipCM}_n(\boldsymbol{d})}}

\newcommand{\dif}{\mathrm{d}}
\newcommand{\bld}[1]{\boldsymbol{#1}}

\newcommand{\cE}{\mathcal{E}}
\newcommand{\cG}{\mathcal{G}}
\newcommand{\cL}{\mathcal{L}}
\newcommand{\cM}{\mathcal{M}}
\newcommand{\cN}{\mathcal{N}}
\newcommand{\cP}{\mathcal{P}}
\newcommand{\cS}{\mathcal{S}}
\newcommand{\sF}{\mathscr{F}}

\newcommand{\PRCM}{\PR_{\sss \mathrm{CM}}}
\newcommand{\ECM}{\mathrm{ECM}_n(\bld{d})}
\newcommand{\PA}{\mathrm{PA}_n(\bld{\delta},m_n)}
\newcommand{\elld}{\ell_{n,\delta}}
\newcommand{\QV}{\mathrm{QV}}

\newcommand{\GP}{\mathrm{GP}}
\newcommand{\CRM}{\mathrm{CRM}}
\newcommand{\GRG}{\mathrm{GRG}_n(\bld{w})}

\newcommand{\dTV}{d_{\sss \mathrm{TV}}}

\newcommand{\smpl}{\mathrm{Smpl}}
\newcommand{\cA}{\mathcal{A}}
\newcommand{\rE}{\mathrm{E}}
\newcommand{\rV}{\mathrm{V}}
\newcommand{\cW}{\mathcal{W}}

\newcommand{\ri}{\mathrm{i}}
\newcommand{\eqd}{\stackrel{\sss d}{=}}
\newcommand{\poi}{\mathrm{Poi}}
\newcommand{\sB}{\mathscr{B}}

\newcommand{\lbl}{\mathrm{Lbl}}
\newcommand{\Err}{\mathrm{Err}}
\newcommand{\Var}{\mathrm{Var}}
\newcommand{\fG}{\mathfrak{G}}

\newcommand{\xiCM}{\xi_{\sss \mathrm{CM}}}
\newcommand{\xiGRG}{\xi_{\sss \mathrm{GRG}}}
\newcommand{\eps}{\varepsilon}

\newtheorem{theorem}{Theorem}
\newtheorem*{claim*}{Claim}

\newtheorem{lemma}[theorem]{Lemma}
\newtheorem{proposition}[theorem]{Proposition}
\newtheorem{corollary}[theorem]{Corollary}
\newtheorem{assumption}{Assumption}
\newtheorem{remark}{Remark}
\newtheorem{fact}{Fact}
\numberwithin{fact}{section}
\newtheorem{defn}{Definition}

\let\plainqed\qedsymbol

\numberwithin{equation}{section}
\numberwithin{theorem}{section}

\endlocaldefs

\begin{document}

\begin{frontmatter}
\title{Limits of Sparse Configuration Models and Beyond: Graphexes and Multi-Graphexes}
\runtitle{Graphex limits of random graphs}

\begin{aug}
\author{\fnms{Christian} \snm{Borgs}\thanksref{m1}\ead[label=e1]{Christian.Borgs@microsoft.com}},
\author{\fnms{Jennifer} \snm{T. Chayes}\thanksref{m1}\ead[label=e2]{jchayes@microsoft.com}}
\author{\fnms{Souvik} \snm{Dhara}\thanksref{m1,m2,t3}\ead[label=e3]{sdhara@mit.edu}}
\and
\author{\fnms{Subhabrata} \snm{Sen}\thanksref{m3}
\ead[label=e4]{subhabratasen@fas.harvard.edu}
}

\thankstext{t3}{SD was supported by an internship at Microsoft Research Lab -- New England, and by the Netherlands Organisation for Scientific Research (NWO) through Gravitation Networks grant 024.002.003.}
\runauthor{Borgs, Chayes, Dhara, Sen}

\affiliation{Microsoft Research\thanksmark{m1}, Massachusetts Institute of Technology\thanksmark{m2} and Harvard University \thanksmark{m3}}

\address{Microsoft Research,\\
One Memorial Drive,\\
Cambridge MA 02142.\\
\printead{e1}\\
\phantom{E-mail:\ }\printead*{e2}
}

\address{Department of Mathematics\\
Massachusetts Institute of Technology, \\
77, Massachusetts Avenue, Building 2, \\
Cambridge MA 02139. \\
\printead{e3}
}

\address{Department of Statistics\\
Harvard University, \\
One Oxford Street, Suite 400\\
Cambridge, MA 02138. \\
\printead{e4}
}

\end{aug}

\begin{abstract}
We investigate structural properties of large, sparse random graphs through the lens of \emph{sampling convergence} (\citet{BCCV17}). Sampling convergence generalizes left convergence to sparse graphs, and describes the limit in terms of a \emph{graphex}. We introduce a notion of sampling convergence for sequences of multigraphs, and establish the graphex limit for the configuration model, a preferential attachment model, the generalized random graph, and a bipartite variant of the configuration model.
The results for the configuration model, preferential attachment model and  bipartite configuration model provide necessary and sufficient conditions for these random graph models to converge.
The limit for the configuration model and the preferential attachment model is an augmented version of an exchangeable random graph model introduced by \citet{CF17}.
\end{abstract}

\begin{keyword}[class=MSC]
\kwd[Primary ]{60C05}
\kwd{05C80}
\kwd[; secondary ]{60G09}
\kwd{60G55}
\kwd{60G57}
\end{keyword}

\begin{keyword}
\kwd{graphexes, sparse random graphs, Caron-Fox model, configuration model, preferential attachment, generalized random graph}
\end{keyword}

\end{frontmatter}

\section{Introduction}
\subsection{Aims and informal overview}
The study of large networks, arising from applications in the social, physical and life sciences, has witnessed meteoric growth over the past two decades. 
It is widely believed that a thorough understanding of the typical structural properties of these large networks can often provide deep insights into the workings of many social, economic and biological systems of practical interest. Random graph models have been extensively used to study properties of these networks, with many recent models aimed at capturing specific properties of real world networks (we refer the interested reader to \cite{RGCN1} and the references therein for an overview).

In this light, it is desirable to study the asymptotic structural properties of random graphs.
A natural question here is to identify a deterministic structure that captures the typical behavior of these random graph models.
This question is analogous to deriving strong law of large numbers, but now on the space of graphs.
The first challenge is to figure out the topology needed  on the space of graphs for such convergence results.
In case of dense graphs, where the number of edges in the graphs scale quadratically with the number of vertices, the theory of graph limits \cite{LS06,LS07,BCLSV08,BCLSV12,L2012,BCLSV06,DJ07} provides the relevant framework to obtain this asymptotic description, and an extensive line of work \cite{BCLSV11,CDS11,RS12} is aimed at describing the asymptotic behavior of dense random graphs. However, this framework fails to provide non-trivial information about sparse graph sequences, and thus
motivates a recent line of work to extend the theory of graph limits to the sparse setting \cite{BCCZ14,BCCZ14b,BCCL18,BCCH16,BCCV17,BR09,J2017}.

In this paper, we derive limits of fundamental sparse random graphs
with respect to the notion of \emph{sampling convergence} introduced recently by \citet{BCCV17}.
 Empirical evidence suggests that typical real-world networks are composed of high degree vertices, called ``hubs'', which form a skeleton
of the network, and low degree vertices that constitute the body of the network \cite{Bar16,New10}.
The limiting object under sampling convergence, called graphex, makes this distinction explicit,
see Section~\ref{sec:discussion} for a more detailed discussion.

Our principal contributions in this article are as follows.

\paragraph*{Convergence of multigraphs}
We introduce a notion of sampling convergence for multigraphs, generalizing the notion of sampling convergence introduced in  \cite{BCCV17}, and identify the resulting limit object, which we call a \emph{multigraphex}.
We also formulate an equivalent notion of convergence in terms of certain two dimensional point processes; it is this representation we use when
establishing the limits of the various random models considered in this paper.

\paragraph*{Limits of random graphs} We deduce the (multi)graphex limit of  fundamental random graph models, under the sparse setting -- the configuration model (Theorem~\ref{thm:main-CM}), a preferential attachment model (Theorem~\ref{thm:PAM-limit}), the generalized random graph (Theorem~\ref{thm:GRG-limit}), and a bipartite variant of the configuration model (Theorem~\ref{thm:main-BCM}). The proof techniques here are completely disjoint from the previous results in the dense settings \cite{BCLSV11,CDS11,RS12}.
In the dense case, graph convergence is equivalent to convergence of subgraph densities, which are real-valued random variables.
Such equivalence breaks down in the sparse setting.
To this end, we make use of an idea, put forth in \cite{BCCV17}, that  sampling convergence is equivalent to weak convergence of certain two-dimensional point processes (Proposition \ref{prop:samp-embed-lbl-equiv}). The relevant point processes for the configuration model, the preferential attachment model and the bipartite configuration model all have a specific ``rank-one'' structure (see Remark~\ref{rem:rank1} below), which in turn allows us to conclude that weak convergence is equivalent to the weak convergence of a one-dimensional L\'{e}vy process. This facilitates a precise characterization of the necessary and sufficient conditions for sampling convergence in these random graph models.
To illustrate the ``non rank-one case'', we analyze the generalized random graph, and derive sufficient conditions for sampling convergence.
Our analysis in this case provides a fairly general template, and may prove to be useful for establishing  sampling convergence for other graph sequences of practical interest.

\paragraph*{New interpretation of the Caron-Fox model}
Finally, our results provide a novel, alternative perspective on the \citet{CF17} model, which has induced immense recent in theoretical statistics.
Specifically, a corollary of our result (Theorem~\ref{thm:main-CM}) establishes that Caron-Fox graphs can be interpreted as the limit of samples from a configuration model or preferential attachment random graphs.

\vspace{.2cm}

\subsection{Notations and terminologies}
Before we progress further,
 we introduce some notation used throughout our subsequent discussion.
We use the standard notation of $\xrightarrow{\sss\PR}$, and $\dto$ to denote convergence in probability and in distribution, respectively.
We use the Bachmann–Landau notation $O(\cdot)$, $o(\cdot)$, $\Theta(\cdot)$, $\Omega(\cdot)$ for asymptotics of real numbers.
$\R_+=[0,\infty)$, ${\overline\R}_+=\R_+\cup\{\infty\}$,
$\N_0 = \N\cup \{0\}$,
and $\otimes$ := product of measures.

Given a multigraph $G$, we use the generic notation $\rV(G)$, $\rE(G)$ to denote the set of vertices and edges respectively, and set $v(G) = |\rV(G)|$.
Further, we denote the number of non-loop edges as $e(G)$.
Let $\fG_f$ denote the set of all multigraphs with finite number of vertices and edges.
Thus $\fG_f$ is countable and we equip this space with discrete topology.

For any topological space $X$,  $\sB(X)$ will denote the Borel sigma-algebra of $X$. We define a measure on a metric space to be locally finite if it assigns finite measure to all bounded Borel sets.
Let $\cM(\R_+)$ and $\cM(\R^2_+)$
denote the space of locally finite Borel measures on $\R_+$ and $\R^2_+$, respectively,
equipped with the vague topology.
$\cN(\R_+) \subset \cM(\R_+)$ and
$\cN(\R^2_+) \subset \cM(\R^2_+)$ will denote the subspaces of counting measures, equipped with the vague topology.
For a Polish space $\cS$, let $\cP(\cS)$ denote the space of all probability measures on $(\cS,\sB(\cS))$, equipped with the topology for weak convergence of probability measures.
For an $\cS$-valued random variable $X$, let $\cL(X)$ denote the law of $X$, which is an element of $\cP(\cS)$.

We will also need the notion of completely random measures, defined as  random measures $\mu\in \cM(\R_+)$ that obey the condition that for all finite families of bounded disjoint sets $(A_i)_{i\leq k}$ in $\sB(\R_+)$, $(\mu(A_i))_{i\leq k}$ is an independent collection of random variables.
We will in particular be interested in completely random measures that are
stationary, i.e., completely random measures $\mu$ such that the distribution of $\mu([t,t+s])$ depends only on $s$ for any $t,s\in \R_+$.  These can be represented in the form
\begin{equation}\label{eq:CRM-form}
\mu = a \lambda+\sum_{i\geq 1} w_i \delta_{\theta_i},
\end{equation}
where $\{(w_i,\theta_i)\}_{i\geq 1}$ is a Poisson point process on $(0,\infty)\times\R_+$ with rate measure $\rho(\dif w) \lambda(\dif \theta)$, with $\rho$ obeying
certain integrability conditions, see Appendix~\ref{sec:crm} for details.
We  use the notation $\CRM(a\lambda,\rho\times\lambda)$ for the law
of a completely random measure of the form \eqref{eq:CRM-form}.

Finally, given a topological space $X$ and an interval $I\subseteq\R$, we use
$\mathbb{D}(I,X)$ to denote the set of c\`adl\`ag functions $f:I\to X$, i.e., the set of functions $f$ that are right-continuous, and have limits from the left.

\subsection{Sampling convergence and multigraphexes}

\paragraph*{Sampling convergence}

The following two definitions are straight forward generalizations of the notion of sampling and sampling convergence from the simple graph setting of \cite{BCCV17} to our multigraph setting.

\begin{defn}[$p$-sampling] \normalfont
For a multigraph $G$, the $p$-sampled multigraph (denoted by $\smpl(G,p)$) is an unlabeled random graph obtained by keeping each vertex independently with probability~$p$, taking the induced edges on the kept vertices and deleting the isolated vertices.
\end{defn}
\begin{defn}[Sampling convergence] \label{defn:samp-conv} \normalfont Let $(G_n)_{n\geq 1}$ be a sequence
of (multi)graphs. $(G_n)_{n\geq 1}$ is said to be sampling convergent if for all $t>0$, $\smpl(G_n,t/\sqrt{2e(G_n)})$ converges in distribution in $\fG_f$.
\end{defn}
The distribution of the sampled graph is characterized by the subgraph frequencies and thus relates to left convergence.
In order for the sampled subgraph to be informative, $p$ is to be chosen such that the sampled subgraph is non-empty. The choice of $p\sim 1/\sqrt{e(G_n)}$ ensures that the sampled subgraph has $\Theta(1)$ number of edges in expectation, and hence also an expected number of vertices that is $\Theta(1)$.
Note that in the sparse setting, it this only holds since we removed the isolated vertices - without this, the expected number of sampled vertices, $tv(G_n)/\sqrt{2e(G_n)}$, would diverge.


\paragraph*{Graphexes and adjacency measures}
Next, we formally
introduce the limiting objects for sampling convergence.

\begin{defn}[Random adjacency measure] \label{defn:adj_measure}
\normalfont
An adjacency measure
is a
measure  $\xi\in\cN(\R_+^2)$
such that $\xi(A\times B) = \xi(B\times A)$ for all $A,B\in \sB(\R_+)$.
A random adjacency measure is a $\cN(\R_+^2)$ valued random variable that is almost surely an adjacency measure. It is called \emph{exchangeable} if $\xi(\phi^{-1}(A)\times \phi^{-1}(B))\eqd\xi(A\times B)$ for every measure preserving map $\phi:\R_+\to\R_+$.
\end{defn}

Expressing a random adjacency measure $\xi$
 as
$$\xi = \sum_{ij} \beta_{ij}  \delta_{(\alpha_i, \alpha_j)}$$
where $\beta_{ij} \in \N_0$,
one can naturally associate an unlabeled graph $\cG(\xi)$ as follows:
Consider a countable vertex set where vertex $i$ is labeled by~$\alpha_i$.
$\cG(\xi)$ is obtained by having $\beta_{ij}$ many edges between the vertices labeled $\alpha_i$ and $\alpha_j$, deleting the isolated vertices, and finally erasing the labels of the vertices.

The limits of sampling convergence will be related to
 a $\fG_f$ valued stochastic process
obtained from specific random adjacency measures in a way we make more precise below.
To define the random adjacency measures under consideration, we first define a multigraphex. We denote the sequence space $\ell_1 = \{ (x_i)_{i\geq 1}: x_i \in \R_+\ \forall i, \sum_{i=1}^{\infty} x_i < \infty\}$.

\begin{defn}[Multigraphex]\label{def:multigraphex}
\normalfont
A multigraphex is a triple $\cW = (W,S,I)$ such that $I\in \ell_1$, $S:\R_+  \mapsto \ell_1$ is a measurable function, and $W:\R_+^2\times\N_0\mapsto \R_+$ is a measurable function satisfying  $W(x,y,k) = W(y,x,k),$  $\sum_{k=0}^{\infty} W(x,y,k) = 1,$ for any $x,y\in\R_+$ and $k\in\N_0$.
We will assume throughout that, $ \min\{ \sum_{k \geq 1}S(\cdot, k), 1\}$ is integrable. Further, setting $\mu_W(\cdot) = \int (1- W( \cdot ,y,0) )\dif y$, we assume that
\begin{enumerate}[(a)]
\item $\Lambda(\{x: \mu_W(x) = \infty \})=0$ and $\Lambda(\{x: \mu_W(x) >1 \})<\infty$,
\item $\int (1- W(x,y,0))\ind{\mu_W(x) \leq 1 }\ind{\mu_W(y) \leq 1 } \dif y \dif x <\infty$,
\item $\int  (1- W(x,x,0))\dif x <\infty$.
\end{enumerate}
\end{defn}
$W$ is called a multigraphon, $S$ is called a multi-star function and $I$ is called an isolated edge sequence.
A graphon $W$ is a multigraphon with $W(x,y,k) = 0$ for $k \geq 2$. In this case, we describe the graphon as a function $W: \R_+^2 \to [0,1]$ and set $W(x,y) = W(x,y,1)$. Similarly, a simple star function $S:\R_+ \mapsto \ell_1$ satisfies $S(x, k) = 0$ for $k \geq 2$ and all $x \in \R_+$. In this case, we describe the star function as $S: \R_+ \mapsto \R_+$, and set $S(x) = S(x,1)$. Finally, a simple isolated edge constant $I$ corresponds to the case where $I(k) = 0$ for $k \geq 2$. In this case, we represent $I = I(1)$ as a constant.
A graphex is obtained by replacing the multigraphon, multi-star function, and multi-edge sequence in Definition~\ref{def:multigraphex} by their simple analogues,
with the isolated edge constant  sometimes referred to as the dust component of the graphex \cite{J2017,BCCL18}.

In this paper, the case where $W$ is a multigraphon, but $I$ and $S$ are simple plays an important role;
henceforth, whenever we specify a star function
$S: \R_+ \mapsto \R_+$
or an isolated edge constant $I$, we assume that these
describe a star function or  an edge sequence with
$S(x,k)$ or $I(k) = 0$ for $k \geq 2$
without
explicitly mentioning it in every case.

\begin{defn}[Adjacency measure of a multigraphex] \label{defn:adj_multigraphex}
\normalfont
Given any multigraphex $\cW = (W,S,I)$,
define $\xi_{\sss \cW}$, the \emph{random adjacency measure
generated by $\cW$}  as follows:
\begin{eq}\label{eq:def-graphex}
\xi_{\sss \cW} &= \sum_{i\neq j} \zeta_{ij} \delta_{(\theta_i,\theta_j)} + \sum_{i} \zeta_{ii} \delta_{(\theta_i,\theta_i)}+ \sum_{j,k} g(\theta_j, \chi_{jk})\big(\delta_{(\theta_j,\sigma_{jk})}+\delta_{(\sigma_{jk},\theta_j)}\big)
\\
&\hspace{1cm}+ \sum_k  h(\eta_k'') \big(\delta_{(\eta_k,\eta_k')}+\delta_{(\eta_k',\eta_k)}\big),\\
\zeta_{ij} &= r, \quad \text{if } \sum_{l=0}^{r-1}W(v_i,v_j,l) \leq U_{\{i,j\}} \leq  \sum_{l=0}^{r}W(v_i,v_j, l),\\
g(\theta_j, \chi_{jk}) &= r, \quad \text{if } \sum_{l=0}^{r-1} S(v_j, l) \leq \chi_{jk} \leq \sum_{l=0}^{r} S(v_j, l), \\
h(\eta_k '') &= r, \quad\text{if } \sum_{l=0}^{r-1} I(l) \leq \eta_k'' \leq \sum_{l=0}^{r} I(l).
\end{eq}
where $(U_{\{i,j\}})_{i,j\geq 1}$ is a collection of independent uniform[0,1] random variables, $\{(\theta_j,v_j)\}_{j\geq 1}$, $\{(\chi_{jk},\sigma_{jk})\}_{k\geq 1}$ for all $j\geq 1$ are unit rate Poisson point processes on $\R_+^2$, and $(\eta_k,\eta_k',\eta_k'')_{k\geq 1}$ is a unit rate Poisson point processes on $\R_+^3$, where all the above Poisson point processes are independent of each other and $(U_{\{i,j\}})_{i,j\geq 1}$.
\end{defn}

\begin{remark}\label{rem:adjacency-proof}
\normalfont
It is not too hard to check that the measure $\xi_{\sss \cW}$ introduced above is
a.s. locally finite and defines an exchangeable
random adjacency measure.
This raises the question whether every exchangeable random adjacency measure necessarily corresponds to some (possibly random) multigraphex.
For graphexes, this result was established in \cite[Theorem 4.7]{VR15}. Their proof crucially uses a local finiteness criterion from Kallenberg \cite[Prop 9.25]{K2005}, which, however turns out to not be quite correct.  Indeed, the conditions for local finiteness in this proposition need to be supplemented by an
extra condition, which was implicitly assumed by Kallenberg in his proof (as well as in the proof of  \cite[Theorem 4.7]{VR15}). We discuss this issue in \cite{BCDS18suppl}, where we state and prove the corrected proposition,
and then apply it prove the analogue of \cite[Theorem 4.7]{VR15} for multigraphexes, stating that for any exchangeable random adjacency measure $\xi$, there exists a random multigraphex $\cW$ such that $\xi\eqd\xi_{\sss \cW}$. 
%
%
%
%
\end{remark}

The adjacency measure $\xi_{\sss \cW}$ associated with the multigraphex $\cW = (W,S,I)$ naturally defines
 a $\fG_f$ valued stochastic process, as we define next.
 For a point process $\xi$, let us denote by $\xi\vert_A$ the measure $\xi$ restricted to $A$.

\begin{defn}[Multigraphex process]\normalfont
For any given multigraphex $\cW$
 we define the
\emph{multigraphex process generated by $\cW$}
as the $\fG_f$-valued stochastic process $(\GP_t(\cW))_{t\geq 0}$
where $\GP_t(\cW) = \cG(\xi_{\sss \cW}\vert _{[0,t]^2})$.
\end{defn}

\begin{remark}\label{rem:Multigrphex-distr}\normalfont
There is an equivalent, albeit operationally slightly simpler description of the distribution of $\GP_t(\cW)$.
Indeed, it can be obtained by considering a single Poisson process
$\{v_j\}_{j\geq 1}$ of rate $t$ on $\R_+$, and then adding edges according to the following procedure:
\begin{itemize}
\item[$\rhd$]
for $i\neq j$, connect $v_i$ and $v_j$ with $n_{ij}$ edges, where $\PR(n_{ij}=r)=W(v_i,v_j,r)$;
\item[$\rhd$] for each $j$, add $n_j$ self-loops  to $v_j$, where $\PR(n_{j}=r)=W(v_j,v_j,r)$;
\item[$\rhd$]
for each $j$ add a multi-star to
$v_j$ by adding edges of multiplicity $r$ at a rate $tS(v_j, r)$;
\item[$\rhd$] add isolated edges of multiplicity $r$ with rate $t^2I(r)$.
\end{itemize}
Discard all isolated vertices (as well as all labels), and output the resulting unlabeled graph.
\end{remark}

\noindent
Finally, we define sampling convergence of a sequence of multigraphs to a multi-graphex.

\begin{defn}[Convergence to multigraphex]\label{defn:samp-conv-graphex} \normalfont A sequence $(G_n)_{n\geq 1}$ of (multi)graphs is said to converge to a (multi)graphex $\cW$ if  for all $t>0$, $(\smpl(G_n,t/\sqrt{2e(G_n)}))_{n\geq 1}$ converges to $\GP_t(\cW)$ in distribution.
\end{defn}
Of course, it is not clear whether sampling convergence in the sense of Definition~\ref{defn:samp-conv} implies convergence to a multigraphex.
To address this question, we first introduce
an equivalent characterization of sampling convergence for multigraphs.
\begin{defn}[Random labeling]\label{defn:labelling} \normalfont A labeling of a multigraph $G$ into $[0,s]$, denoted by $\lbl_s(G)$, is a point process generated as follows: For a collection $(U_v)_{v\in \rV(G)}$ of independent and identically distributed uniform$[0,s]$, associate weight $U_v$ to vertex $v$. Then $\lbl_s(G): = \sum_{v,w\in \rV(G)} n_{vw}\delta_{(U_v,U_w)}$, where $n_{vw}$ denotes the number of edges between vertices $v$ and $w$. The canonical labeling of $G$, denoted by $\lbl(G)$, refers to the case $s = \sqrt{2e(G)}$.
\end{defn}

\begin{proposition} \label{prop:samp-embed-lbl-equiv}
Consider a sequence of multigraphs $(G_n)_{n\geq 1}$ with $e(G_n)<\infty$ for all $n\geq 1$ and $\lim_{n\to\infty}e(G_n)=\infty$. Then the following are equivalent:
\begin{enumerate}[(i)]
\item $(G_n)_{n\geq 1}$ is sampling convergent.
\item $(\lbl(G_n))_{n\geq 1}$ converges in distribution as random variables in $\cN(\R_+^2)$.
\end{enumerate}
Moreover, if the distributional limits of $(\smpl(G_n,r/\sqrt{2e(G_n)}))_{n\geq 1}$ and \\
 $(\lbl(G_n))_{n\geq 1}$ are given by $H_r$ and $\xi$, then $\lbl_r(H_r) \stackrel{\sss d}{=} \xi\vert_{[0,r)^2}$.
Further, $\xi$ is extremal.
Therefore, there exists a multigraphex $\cW$ (non-random) such that
$\xi \stackrel{\sss d}{=}
\xi_\cW$, and $(G_n)_{n\geq 1}$ is sampling convergent to $\cW$.
\end{proposition}
\noindent

The above proposition says that the limit of any sampling convergent sequence of graphs must be a multigraphex.  For simple graphs, an analogue of Proposition~\ref{prop:samp-embed-lbl-equiv} was proved in \cite[Section 3]{BCCV17},
relying in turn on \cite[Theorem 4.7]{VR15}.
Modulo the issues mentioned in Remark~\ref{rem:adjacency-proof}, the proof in the multigraph setting is very similar.
%
We provide an outline
in Appendix~\ref{sec:multigraph-convergence-suppl}.

\subsection{Limits of random graphs}
\paragraph*{Configuration model}
The \emph{configuration model} is the canonical model for generating a random multi-graph with a prescribed degree sequence.
This model was introduced by Bollob\'as~\cite{B80}  to choose a uniform simple $d$-regular graph on $n$ vertices, when $dn$ is even. The idea was later generalized for general degree sequences $\boldsymbol{d}$ by Molloy and Reed~\cite{MR95} and others (see \cite{RGCN1}).
Consider a sequence $\bld{d}=(d_1,d_2,\dots,d_n)$ such that $\ell_n = \sum_{i\in [n]}d_i$ is even, where $[n] = \{1,\dots,n\}$.
Equip vertex $j$ with $d_{j}$ stubs or \emph{half-edges}. Two half-edges create an edge once they are paired.
Therefore, initially there are $\ell_n=\sum_{i \in [n]}d_i$ half-edges.
Pick any one half-edge and pair it with a uniformly chosen half-edge from the remaining unpaired half-edges.
Keep repeating the above procedure until all the unpaired half-edges are exhausted.
The random graph constructed in this way is called the configuration model, and will henceforth be denoted by $\mathrm{CM}_{n}(\boldsymbol{d})$.
  Note that the graph constructed by the above procedure may contain self-loops and multiple edges.
  We define erased configuration model to be the graph obtained by collapsing all the multiple edges to single edges.
  We denote this graph by $\ECM$.
This is slightly different compared to the erased configuration model in \cite{RGCN1} since instead of deleting the loops, we merge multiple loops into single loops.

We will want to study the sampling limit of a sequence of configuration models,
given in terms of a sequence of sequences of $(\bld{d}_n)_{n\geq 1}$, but  for notational convenience, we suppress the index $n$ of the sequence $\bld{d}_n$, and just speak of the
limit of a
sequence $\CM$ random graphs.
Since isolated vertices are removed in the process of sampling, we will assume
 without loss of generality that $d_i>0$ for all $i$.
The following quantities determine this limit:
\begin{align*}
\rho_n: = \frac{1}{\sqrt{\ell_n}}\sum_{i\in [n]}\delta_{\frac{d_i}{\sqrt{\ell_n}}}, \quad b_n = \int_0^\infty (x \wedge 1) \rho_n(\dif x),
\end{align*}
with $\rho_n$ considered as a measure on $(0,\infty)$.
Throughout this discussion, we will assume that $\max_{1 \leq i \leq n} d_ i = o(\ell_n)$ and
$\ell_n = \omega( \log n) $. This restriction is purely technical, and might possibly be relaxed. We do not pursue this in this paper.

Next, we introduce the limiting graphex for
a sequence of configuration models.
Given any measure $\rho$ on $(0,\infty)$, we use $\bar{\rho}$ to denote the tail L\'evy intensity
\begin{align*}
\bar{\rho}(x) = \int_x^\infty \rho(\dif y);
\end{align*}
Defining  $\bar\rho^{-1}$  as its inverse, $\bar\rho^{-1}(y):= \inf \{x \in\R_+ : \bar{\rho}(x)\leq y\}$ we note that $\bar\rho^{-1}$ is a  c\`adl\`ag function from $(0,\infty)$ to $\R_+$. It will be convenient to extend $\bar\rho^{-1}$ to a function
defined on $\R_+$ by setting $\bar\rho^{-1}(0)=0$.
Finally, let $p(k;\lambda)$ be the probability that a Poisson $\lambda$
random variable takes the value $k$,
 $p(k;\lambda) = \e^{-\lambda}\lambda^k/k!$.
For any $a\in\R_+$ and any
measure $\rho$ on $(0,\infty)$ satisfying $\int_0^{\infty} (x\wedge 1) \dif \rho(x) < \infty$, we define
\begin{equation}
\label{multigraphon-CM}
\begin{aligned}
&W_{\sss \mathrm{CM}}(x,y,k) = \begin{cases}
p(k; \bar{\rho}^{-1}(x)\bar{\rho}^{-1}(y)), \quad &x\neq y\\
p(k; \bar{\rho}^{-1}(x)^2/2), &x=y.
\end{cases}
\\
&S_{\sss \mathrm{CM}} (x) = a \bar{\rho}^{-1}(x), \,\,\,\,\,\,
I_{\sss \mathrm{CM}} = \frac{a^2}{2}.
\end{aligned}
\end{equation}
It is easy to see, by direct computation, that the graphex $\cW_{\sss \mathrm{CM}}=(W_{\sss \mathrm{CM}}, S_{\sss \mathrm{CM}}, I_{\sss \mathrm{CM}})$ satisfies the integrability criteria in Definition \ref{def:multigraphex}, and thus the associated random adjacency measure is locally finite.

The following result derives necessary and sufficient conditions for the sampling convergence of $\CM$ random graphs, and characterizes the limiting objects.  To state our theorem, we introduce two random objects defined in terms the sequence $\bld{d}=(d_1,\dots,d_n)$ and the associated  c\`adl\`ag function $\bar{\rho}_n^{-1}$: a c\`adl\`ag process $(Y_n(t))_{t\geq 0}$ with
$$Y_n(t) = \frac{1}{\sqrt{\ell_n}}\sum_{i\in [n]}d_i \ind{U_i\leq t}$$
where $(U_i)_{i\in [n]}$ is an i.i.d.\ sequence of random variables
$U_i\sim \text{uniform}[0,\sqrt{\ell_n}]$,
and a completely random measure
\begin{eq}\label{mu-n-definition}
\mu_n = \sum_{i\geq 1} \bar{\rho}_n^{-1}(v_i) \delta_{\theta_i},
\end{eq}
where $\{(v_i,\theta_i)\}_{i\geq 1}$ is
a unit rate Poisson point process
on $\R_+^2$.

We write $\PR_{\sss \mathrm{CM}}$ to denote the product measure $\bigotimes_{n\geq 1} \PR_n$.
The probability measure $\PR_{\mathrm{ECM}}$ is defined analogously.

\begin{theorem}\label{thm:main-CM}
The following are equivalent.
\begin{enumerate}[(i)]
\item $\CM$ is sampling convergent a.s. $\PR_{\sss \mathrm{CM}}$.
\item There exists a  random measure $\mu$ such that
  $\cL(\mu_n)$ converges to $\cL(\mu)$ in $\cP(\cM(\R_+))$.
\item There exists $b$ and $\rho$ such that  $b_n \to b$ and $\rho_n \to \rho$ vaguely.
\item $(Y_n(t))_{t\geq 0}$ converges in distribution in $\mathbb{D}(\R_+,\R_+)$.
\end{enumerate}
Moreover, if (iii) holds, then $\int_0^{\infty} (x \wedge 1) \rho(\dif x)\leq b\leq 1$, and  $\mu$ is of the form
\eqref{eq:CRM-form} with
$a =b - \int (x\wedge 1)\rho(\dif x)$.
In this case,
 $(Y_n(t))_{t\geq 0}$ converges in distribution to $(Y(t))_{t\geq 0}$, where $Y(t)=\mu([0,t])$, and
 $\CM$ is sampling convergent to the multi-graphex $\cW_{\sss \mathrm{CM}}=(W_{\sss \mathrm{CM}}, S_{\sss \mathrm{CM}}, I_{\sss \mathrm{CM}})$, almost surely $\PRCM$,
where $\cW_{\sss \mathrm{CM}}$ is defined in \eqref{multigraphon-CM}.
\end{theorem}

\begin{remark} \normalfont
The graphon limit of dense configuration models was derived in \cite{RS12}. It is easy to see that Theorem~\ref{thm:main-CM} recovers the dense graph limit in  case $\ell_n = \Theta(n^2)$.
\end{remark}

\begin{remark}\label{rem:a=low-degrees}\normalfont
When proving the equivalence of (ii) and (iii), we will at the same time prove that
under the condition (iii) from the above theorem,
\begin{equation}
\label{a=low-degrees}
a=\lim_{\eps\to 0}\liminf_{n\to\infty} \int_0^\eps x \rho_n(dx)
=\lim_{\eps\to 0}\limsup_{n\to\infty} \int_0^\eps x \rho_n(dx)
\end{equation}
(in fact, we will show that the condition $b_n\to b$ in (iii) could be equivalently be
replaced by the condition that the second equality in \eqref{a=low-degrees} holds).
Since $\int_0^\eps x\rho_n(\dif x)=\frac 1{\ell_n}\sum_i d_i1_{d_i\leq \eps\sqrt{\ell_n}}$, the constant $a$ therefore represents the limiting fraction of half-edges with degrees $d_i=o(\sqrt{\ell_n})$, and the condition
$b_n\to b$ is the condition that this limiting fraction exists.
\end{remark}

\begin{remark}\label{rem:rank1}
Let $\xi_{\sss \mathrm{CM}}$ be the random adjacency measure
associated to the multigraphex $\cW_{\sss \mathrm{CM}}$, and let
 $\xi^*_{\sss \mathrm{CM}}:=\xi_{\sss \mathrm{CM}}\vert_{(x,y):y\leq x}$.
Then $\xi_{\sss \mathrm{CM}}$ has the following ``rank-one''
structure:  For any disjoint  set $A,B \in \sB(\R_+)$  of finite measure, the  distribution of $\xiCM^*(A\times A)$ is $\mathrm{Poisson}(\mu(A)^2/2)$ and that of $\xiCM(A\times B)$ is $\mathrm{Poisson}(\mu(A)\mu(B))$. See
Appendix~\ref{sec:properties} for the precise statement and proof.
\end{remark}

We obtain the following corollaries from Theorem~\ref{thm:main-CM}.
Define the graphon $W_{\sss \mathrm{ECM}}$ by
\begin{equation}\label{eq:ECM-graphon}
W_{\sss \mathrm{ECM}}(x,y) = \begin{cases}
1-\e^{- \bar{\rho}^{-1}(x)\bar{\rho}^{-1}(y)}, \quad &x\neq y\\
1-\e^{-\bar{\rho}^{-1}(x)^2/2}, &x=y,
\end{cases}
\end{equation}
as well as a re-scaled graphex $\cW_{\sss \mathrm{ECM}}^c := (W_{\sss \mathrm{ECM}}(\sqrt{c}\, \cdot, \sqrt{c}\, \cdot), \frac{1}{\sqrt{c}}S_{\sss \mathrm{CM}}(\sqrt{c} \,\cdot), \frac{1}{c} I_{\sss \mathrm{CM}})$.
Recall
from \cite[]{BCCV17} that any sequence of simple graphs with $\#$loops = $O(\sqrt{\#\text{edges}})$ has a convergent subsequence. The following corollary characterizes possible limit points for $\ECM$ under mild regularity conditions on the degree distribution.

 \begin{corollary} \label{cor:ECM-sampl-conv}
Suppose that $\rho_n\to \rho$ vaguely, $b_n\to b\in \R_+$.
Further, assume that
\begin{eq}\label{eq:ecm_edges_defn}
\int_0^\infty \int_0^\infty (1-\e^{-xy}) \rho_n(\dif x) \rho_n (\dif y) \to c
\qquad\text{as}\qquad n\to\infty
\end{eq}
for some $0< c < \infty$.
Then, as $n\to\infty$, $\ECM$ is sampling convergent to the graphex $\cW_{\sss \mathrm{ECM}}^c$ a.s. $\PR_{\sss \mathrm{ECM}}$.
Moreover, if the LHS of \eqref{eq:ecm_edges_defn} is bounded away from zero, then the limit of any a.s.~$\PR_{\sss \mathrm{ECM}}$ convergent subsequence of $\ECM$ is of the form $\cW^C_{\sss \mathrm{ECM}}$, for some constant $C>0$.
 \end{corollary}

As a further consequence of Theorem~\ref{thm:main-CM}, we study when the limit is a pure graphon or purely isolated edges. To this end,
we define the uniform tail regularity for a sequence of multi-graphs.
\begin{defn}\label{defn-stretched-cut-metric} \normalfont For a vertex $G$, let $d_v(G)$ denote the degree of vertex $v$.
A sequence of (multi)graphs $(G_n)_{n\geq 1}$ is uniformly tail regular if
for any $\varepsilon >0$, there exists $\delta > 0$ such that for all $n\geq 1$,
\begin{eq}
\label{tail-reg}
\frac{1}{e(G_n)}\sum_{v: d_v \leq \delta \sqrt{e(G_n)}} d_v(G_n) < \varepsilon.
\end{eq}
\end{defn}
\noindent Note that Definition~\ref{defn-stretched-cut-metric} is equivalent to \cite[Definition 13]{BCCH16} for simple graphs (see \cite[Remark 14]{BCCH16} and \cite[Lemma 9.3]{BCCL18}).
Also, recall the definition of stretched cut metric from \cite{BCCH16}.

\begin{corollary}\label{cor:pure-graphon}
Assume that $\rho_n \to \rho$, $b_n \to b \in \R_+$.
Then $I_{\sss\mathrm{CM}} = 0$ if and only if $\CM$ is uniformly tail regular a.s.~$\PR_{\sss\mathrm{CM}}$.
Moreover, if  $I_{\sss\mathrm{CM}} = 0$ and \eqref{eq:ecm_edges_defn} holds, then $\ECM$ is uniformly tail regular a.s. $\PR_{\sss \mathrm{ECM}}$.
In this case, $\ECM$ converges to $W_{\sss \mathrm{ECM}}^c $ in the stretched cut metric a.s.~$\PR_{\sss \mathrm{ECM}}$.
\end{corollary}

\begin{corollary}\label{cor:pure-isolate}
As $n\to\infty$, $\CM$ is sampling convergent a.s. $\PR_{\sss \mathrm{CM}}$ to $(0,0,I_{\sss \mathrm{CM}})$ if and only if $b_n \to b$, $\rho_n \to \rho$ vaguely, with $\rho=0$.
\end{corollary}

\paragraph*{Preferential attachment model}
We consider a generalization of the preferential attachment model.
This model was first introduced by Pittel~\cite{Pit10}, and the graph limit in the dense counter part of this model was studied in  \cite{BCLSV11,RS12}.
Let  $\bld{\delta} = (\delta_i)_{i\in [n]}$ be a sequence of non-negative real numbers,
and let  $\ell_{n,\delta}  = \sum_{i\in [n]} \delta_i$.
Initially, $\mathrm{PA}_n(\bld{\delta},0)$ is an empty graph on vertex set $[n]$.
Let $d_i(l)$ denote the degree of vertex $i$ in $\mathrm{PA}_n(\bld{\delta},l)$.
Given the graph $\mathrm{PA}_n(\bld{\delta},l)$ at time $l$, $\mathrm{PA}_n(\bld{\delta},l+1)$ is created by adding one edge to the graph with the end points being chosen with probability proportional to $d_i(l)+\delta_i$. More precisely the edge $(i,j)$ is added at step $l$ with probability
\begin{equation}
\begin{cases}
\frac{2(d_i(l)+\delta_i)(d_j(l)+\delta_i)}{(\elld+2l)^2}, &\quad \text{for }i\neq j,\\
\frac{(d_i(l)+\delta_i)^2}{(\elld+2l)^2}, &\quad \text{for } i = j.
\end{cases}
\end{equation}
The above process iterates $m_n$ times to yield $\PA$.  Throughout, we will assume that
$\max_i\delta_i=o(\elld)$ and  $\min\{\elld,m_n\} = \omega(\log n)$.

When $(\delta_i)_{i\in [n]}$ are integers, the above process can simply be described by an urn scheme where we start with an urn with $\delta_i$ balls of color $i$, for $i\in [n]$.
At step $l$, we select two balls with replacement from the urn. If the colors of the chosen balls are $i$ and $j$, we add an edge between vertices $i$ and $j$.
We also add one additional copy of the balls in the urn with color $i$ and $j$ before the next iteration.

\begin{remark}
\normalfont
The model in \cite{BCLSV11,RS12} is slightly different in the sense that the $l$-th edge is formed by first drawing the $(2l-1)$-th ball from the urn, replacing that ball in the urn and then drawing the $2l$-th, instead of drawing two balls together at step $l$. However, this   does not change the limiting result for the preferential attachment model.
\end{remark}

\noindent The next theorem states that the limit for preferential attachment model.
As we will see,
as long as $m_n=o(\elld^2)$,
the model behaves very much like a configuration model with degree sequence $(\bar d_i)_{i\in [n]}$ where $\bar d_i$ is the expected degree of $i$ at time $m_n$,
$$
\bar d_i=\E[d_i(m_n)]=\frac {2m_n}\elld \delta_i.
$$
To formalize this, we set
\begin{equation}
\rho_{n,\delta}:= \frac{1}{\sqrt{{2m_n}}}\sum_{i\in [n]} \delta_{\frac{\bar d_i}{\sqrt{2m_n}}},
\quad b_{n,\delta} = \int_0^\infty (x \wedge 1) \rho_{n,\delta}(\dif x),
\end{equation}
define $\mu_{n,\delta}$ in terms of $\bar\rho_{n,\delta}^{-1}$ instead of $\bar\rho_n^{-1}$ similarly as in \eqref{mu-n-definition}, and define
$$Y_{n,\delta}(t) = \frac{1}{\sqrt{2m_n}}\sum_{i\in [n]}\bar d_i
 \ind{U_i\leq t}= \frac{\sqrt{2m_n}}{\elld}\sum_{i\in [n]}\delta_i \ind{U_i\leq t}$$
where $(U_i)_{i\in [n]}$ is an i.i.d.\ sequence of unifrom random variables
in $[0,\sqrt{2m_n}]$.
\noindent
We will use sequences $(\bld{\delta}^n)_{n\geq 1}$  and define the measure $\PR_{\sss \mathrm{PAM}}$  analogous to $\PR_{\sss\mathrm{CM}}$.

\begin{theorem}\label{thm:PAM-limit}
Assume that
$m_n=o(\elld^2)$.
Then $\PA$, $\mu_{n,\delta}$, $b_{n,\delta}$, $\rho_{n,\delta}$, and $(Y_{n,\delta}(t))_{t\geq 0}$
satisfies the the class of equivalent statements (i)--(iv) in {\rm Theorem~\ref{thm:main-CM}} a.s. $\PR_{\sss \mathrm{PAM}}$.
Moreover, if $\rho_{n,\delta} \to \rho $ and $a = \lim_{n\to\infty}\int (x\wedge 1)\rho_{n,\delta}(\dif x) - \int (x\wedge 1)\rho(\dif x)$, then $\PA$ is $\PR_{\sss \mathrm{PAM}}$ almost surely sampling convergent to the graphex  $\cW_{\sss \mathrm{CM}}$ is defined  in \eqref{multigraphon-CM}.
\end{theorem}
\begin{remark}\normalfont
Consider the case where $(\delta_i)_{i\in [n]}$ is a collection of positive integers with $\elld$ being even. Then, if
$\elld/2m_n\to 1$,
 the sampling limits of $\PA$ and $\mathrm{CM}_n(\bld{\delta})$ are identical.
Further, if $\delta_i=1$ for all $i\in [n]$, then $\rho$ is the zero measure and $b=1$, corresponding to the
limiting graphex $\cW_{\sss \mathrm{CM}}=(0,0,1/2)$ and a
sampling limit
 consisting of just isolated edges.
\end{remark}

\begin{remark}\normalfont
\cite{BCLSV11,RS12} derived the graphon limit for this model in the setting $m_n = \Theta (\elld^2)$, and $\delta_i = \kappa$ for all $i\in [n]$, where $\kappa$ is a universal constant independent of $n$. A comparison between their results and Theorem~\ref{thm:PAM-limit} shows that the limiting graphons are very different. In particular, a naive extrapolation based on Theorem~\ref{thm:PAM-limit} turns out to be incorrect. Intuitively, this discrepancy is explained by non-trivial fluctuations of the vertex degrees around their expectations. As a result, the measure $\rho_{n, \delta}$ introduced above, does not adequately capture the degree characteristics in the dense setting. As a consequence, this establishes that our assumption $m_n = o(\ell_{n,\delta}^2)$ in Theorem~\ref{thm:PAM-limit} is, in fact, optimal.
\end{remark}

\paragraph*{Generalized random graph}
Given a weight sequence $(w_i)_{i\in [n]}$, the generalized random graph model, denoted by $\mathrm{GRG}_n(\bld{w})$, is obtained by connecting vertices $i$ and $j$ independently with probability
\begin{equation}\label{pij}
p_{ij} = \frac{w_iw_j}{L_n+w_iw_j},
\end{equation}where $L_n = \sum_{i\in [n]}w_i$. Throughout, we will assume $L_n = \omega(\log n)$.
This model has been of considerable theoretical interest since, conditionally on the degree sequence $\bld{d}$, this gives a uniformly chosen simple graph with degree sequence $\bld{d}$~\cite{BDM06,RGCN1}.  This is also related to the $\beta$-model, studied in \cite{CDS11}.

\begin{assumption}\label{assumption-weight}\normalfont
\begin{enumerate}[(i)]
\item $\rho_{n,w} = \frac{1}{\sqrt{L_n}}\sum_{i\in [n]} \delta_{\frac{w_i}{\sqrt{L_n}}}$ converges vaguely to some measure $\rho_w$.
\item $\lim_{\varepsilon\to 0}
\limsup_{n\to\infty}
\int_0^\varepsilon x \rho_{n,w}(\dif x)
=\lim_{\eps\to0}\liminf_{n\to\infty}
\int_0^\varepsilon x \rho_{n,w}(\dif x)
= a$ for some constant $a$.

\item  $\liminf_{n \to \infty} \int_0^{\infty} \int_0^{\infty} \frac{xy}{1+ xy} \rho_n (\dif x) \rho_n (\dif y) >0$.
\end{enumerate}
\end{assumption}
\noindent
Define the graphex $\cW_{\sss\mathrm{GRG}}^{C} = (W_{\sss\mathrm{GRG}}^{C},S_{\sss\mathrm{GRG}}^{C} ,I_{\sss\mathrm{GRG}}^{C})$, given by
\begin{align*}
W_{\sss\mathrm{GRG}}^C(x,y) = \begin{cases}
\frac{\bar{\rho}_w^{-1}(\sqrt C x)\bar{\rho}_w^{-1}(\sqrt C y)}{1+\bar{\rho}_w^{-1}(\sqrt C x)\bar{\rho}_w^{-1}(\sqrt C y)}, &\text{if } x \neq y\\
0 &\text{o.w.}
\end{cases} \\
 \quad S_{\sss\mathrm{GRG}}^C(x) = \frac{a}{\sqrt C}\bar{\rho}_w^{-1}(\sqrt C x), \,\,\,\,\,\, \quad I_{\sss\mathrm{GRG}}^C = \frac{a^2}{2C}. \nonumber
\end{align*}

\noindent
We will use sequences $(\bld{w}_n)_{n\geq 1}$ and suppress the dependence on $n$ for notational convenience. Further, we use  $\PR_{\sss \mathrm{GRG}}$ to denote the joint distribution of sequences of $\GRG$ random graphs, where the graphs are sampled independently for each $n$.
\begin{theorem}\label{thm:GRG-limit}
Suppose that
$\GRG$ satisfies {\rm Assumption~\ref{assumption-weight}}, and \begin{equation}\label{eq:GRG-edges-main}
\lim_{n\to\infty}\int_{0}^\infty \int_0^\infty \frac{xy}{1+xy} \rho_n(\dif x ) \rho_n( \dif y) = c>0.
\end{equation}
Then, as $n\to\infty$, $\GRG$ is sampling convergent to $\cW_{\sss\mathrm{GRG}}^{c} $  a.s.~$\PR_{\sss\mathrm{GRG}}$.
\end{theorem}

\begin{corollary}\label{cor:grg-sub}
If $\GRG$ satisfies {\rm Assumption~\ref{assumption-weight}},
 then the limit of any a.s.~$\PR_{\sss \mathrm{GRG}}$ convergent subsequence of $\GRG$ is of the form $\cW_{\sss \mathrm{GRG}}^C$, for some constant $C>0$.
\end{corollary}

\begin{corollary}
\label{lemma:grg_cut-metric}
 For $a=0$, $\GRG$ converges to $W_{\sss\mathrm{GRG}}^{c}$ in the stretched cut metric a.s.~$\PR_{\sss\mathrm{GRG}}$.
\end{corollary}

\paragraph*{Bipartite Configuration Model}
In this section, we describe the sampling limit of bipartite configuration models.
Let us introduce the model first.
Consider two sets of vertices $V_1$ and $V_2$ with associated degree sequences $\bld{d} = (d_{ij})_{i\in V_{j}, j=1,2}$ such that $\sum_{i\in V_{1}}d_{i1}=\sum_{i\in V_{2}}d_{i2}:= \ell_n/2$. Equip the $i$-th vertex in $V_j$ with $d_{ij}$ half-edges.
A bipartite configuration model is generated by sequentially selecting unpaired half-edges one-by-one from $V_1$, and  pairing it with a uniformly chosen unpaired half-edge from $V_2$.
Replacing the paired half-edges by edges, one gets a bipartite random graph, which we denote by $\BCM$.
The probability measure $\PR_{\sss \mathrm{BCM}}$ is defined analogous to $\PR_{\sss\mathrm{CM}}$.

Given a degree sequence $\bld{d}$, the following quantities determine the limit of  $\BCM$: For $j=1,2$, define
\begin{equation}
\rho_{nj}: = \frac{1}{\sqrt{\ell_n}}\sum_{i\in V_j}\delta_{\frac{d_i}{\sqrt{\ell_n}}}, \quad b_{nj} := \int_0^\infty (x \wedge 1) \rho_{nj}(\dif x).
\end{equation}
Throughout, we will assume that $\max_{1 \leq i \leq n, j =1,2} d_ {ij} = o(\ell_n)$. We do not try to relax this restriction here.

\noindent
Next, we introduce the limiting graphex for $\BCM$.
In this case, we consider the feature space $\Omega = \R_+\times \{0,1\}$. We equip $\{0,1\}$ with the counting measure, and the feature space $\Omega$ naturally inherits the product measure on the two component spaces.
For measures $\rho_j$ on $\R_+$ satisfying $\int_0^{\infty} (x\wedge 1) \dif \rho_j(x) < \infty$, we define
\begin{eq}\label{multigraphon-BCM}
&W_{\sss \mathrm{BCM}}\bigg(
\Big(\begin{matrix}x\\ c_1
\end{matrix}\Big),
\Big(\begin{matrix}y\\ c_2
\end{matrix}\Big),k\bigg) = \begin{cases}
p(k; \bar{\rho}_1^{-1}(x)\bar{\rho}_2^{-1}(y)), \quad &c_1\neq c_2\\
0, &c_1=c_2.
\end{cases} \\
&S_{\sss \mathrm{BCM}} \bigg(
\Big(\begin{matrix}
x\\c
\end{matrix}\Big)\bigg) = \begin{cases}
a_{2} \bar{\rho}_1^{-1}(x), & c=0\\
a_1 \bar{\rho}_2^{-1}(x), & c = 1
\end{cases}, \qquad
I_{\sss \mathrm{BCM}} = a_1a_2.
\end{eq}
The extra coordinate $\{0,1\}$ in the feature space encodes the partition of a sampled vertex.
Two vertices in the same partition cannot share an edge, and thus the graphon is zero whenever $c_1 = c_2$.

The following result derives necessary and sufficient conditions for the sampling convergence of $\BCM$ random graphs, and characterizes the limiting objects. We note that while the proof of this result is related to that of Theorem \ref{thm:main-CM} and the associated corollaries, these results help provide further intuition into sampling convergence, and provide interesting examples of possible limits that may be obtained under this notion of convergence.

\begin{theorem}\label{thm:main-BCM}
The following are equivalent.
\begin{enumerate}[(i)]
\item $\BCM$ is sampling convergent a.s. $\PR_{\sss \mathrm{BCM}}$.
\item For $j = 1,2$, there exists $b_j$ and $\rho_j$ such that  $b_{nj} \to b_j$ and $\rho_{nj} \to \rho_j$ vaguely.
\end{enumerate}
If (ii) holds, then $\int_0^{\infty} (x \wedge 1) \rho_j(\dif x)  \leq b_j$
and $\BCM$ is sampling convergent to the multi-graphex $\cW_{\sss \mathrm{BCM}}$, almost surely $\PR_{\sss \mathrm{BCM}}$, with  $a_j =b_j - \int (x\wedge 1)\rho_j(\dif x)$.

\end{theorem}

\noindent
We observe here  that the bipartite structure
allows for interesting sampling limits as described below.
A comparison between Theorem~\ref{thm:main-CM} and Theorem~\ref{thm:main-BCM} shows that in the non bipartite model, $S_{\sss \mathrm{CM}} \neq 0$ implies that both $W_{\sss \mathrm{CM}} \neq 0$ and $I_{\sss \mathrm{CM}} \neq 0$, while this is not necessarily true in the bipartite case.
\begin{remark}
\normalfont
In the special case $\rho_2=0$, $\rho_1 \neq 0$, $a_1 =0$ and $a_2 \neq 0$, the corresponding limit $\cW_{\sss \mathrm{BCM}} = (0, S_{\sss \mathrm{BCM}} , 0)$.
This further illustrates that a configuration model type construction might also yield graphexes with pure star part.
\end{remark}

\begin{remark}
\normalfont
 For degree sequences with $\rho_1=0$, $\rho_2 \neq 0$, $a_1, a_2 \neq 0$, we have a sampling limit with $W_{\sss \mathrm{BCM}} =0$ while $S_{\sss \mathrm{BCM}} \neq 0$ and $I_{\sss \mathrm{BCM}} \neq 0$. Finally, if $a_1=0$, $a_2 \neq 0$ and $\rho_1, \rho_2 \neq 0$, the limiting graphex is of the form $\cW_{\sss \mathrm{BCM}} = (W_{\sss \mathrm{BCM}}, S_{\sss \mathrm{BCM}},0)$.
\end{remark}

\subsection{Discussion}\label{sec:discussion}

\paragraph*{Background}
\citet{DJ07}, and \citet{austin2008exchangeability} identified a beautiful connection between the theory of graph limits, and convergence of exchangeable random arrays.
For dense graphs, the notion of Left convergence is characterized by the convergence of all subgraph densities.
Equivalently, one may permute the vertex labels of a graph $G_n$ uniformly at random, and study the properties of the resulting permuted adjacency matrix.
In the limit, these permuted matrices converge weakly to infinite exchangeable random arrays, and their laws are characterized by the celebrated Aldous-Hoover Theorem \cite{Ald81,Hoo79}.
Further, the limiting law of the array has a one-to-one correspondence with the limiting graphon for the dense graph sequence.
However, in contrast, for sparse graph sequences, these matrices converge to the zero array, and this framework fails to provide non-trivial information about the graph sequence.
Identifying the exchangeable structures that characterize the limits for sparse graphs remained an open question for a decade.

\citet{CF17} introduced a family of random graph models based on a completely random measure, and introduced a notion of exchangeability for dynamically growing random graphs, via.~the exchangeability of their adjacency measure on~$\R_+^2$.
Extending this idea, and using the Kallenberg representation theorem \cite{Kal90,K2005} for exchangeable point processes on $\R_+^2$ as a conceptual cornerstone, \cite{VR15,VR16} (see
also \cite{J2017} for a review and some extensions of the results of \cite{VR16}) introduced a very general class of exchangeable random graph models. They further examine structural properties of these graphs, and address questions related to statistical inference under these generative models.
In parallel, \cite{BCCH16} generalizes graph limit theory by introducing the notion of convergence in \emph{stretched cut metric} for a sequence of sparse graphs.
Finally, \cite{BCCV17} formalized the relation between convergence of sparse graphs, and the convergence of corresponding limiting adjacency measure by introducing the notion of \emph{sampling convergence}, a generalization of Left convergence for sparse graph sequences, and established that
the limiting adjacency measure correspond exactly to graphex processes in \cite{VR15,VR16}.
Further, they establish that
under the assumption of uniform tail regularity
sampling convergence is equivalent to convergence under the stretched cut metric
from \cite{BCCH16}.

In this paper, we utilize these recent advances to study structural properties of random graphs, while simultaneously establishing the usefulness of this nascent theory.

\begin{remark} \normalfont
In a recent paper,
\citet{BCCL18} proposed and studied the weak kernel metric on graphexes. This metric generalizes the cut metric for graphons, and metrizes sampling convergence without additional regularity conditions. Further, two graphexes at zero distance in this metric lead to identically distributed graphex processes, and graphexes are equivalent in this sense if and only if they can be related by measure preserving transformations. It would be interesting to provide an analogous metric for multigraphexes, but this is beyond the scope of this paper, and also somewhat orthogonal to our purpose here.
\end{remark}

\paragraph*{Insights on the graph structure}
Recall  the definition of a multigraphex (Definition~\ref{def:multigraphex}). We take this opportunity to provide further intuition for the components of a multigraphex, and what they imply for the multigraph sequence converging to this multigraphex.
A sequence $(G_n)_{n\geq 1}$ of multi-graphs with $m_n = e(G_n) = o(n^2)$, is composed of three main parts:
\begin{enumerate}[(1)]
\item A dense core where the vertices have degree $\Theta(\sqrt{m_n})$. If the dense part contributes a positive proportion of edges, (i.e., there are $\Theta(\sqrt{m_n})$ many vertices of degree $\Theta(\sqrt{m_n})$) then this part gives rise to the graphon, and thus gives the leading contribution to the subgraph densities for subgraphs
    that are more complex than isolated edges or stars.
\item A sparse part where the vertices have  degree  $o(\sqrt{m_n})$.
For the purpose of this discussion, assume that the edges out of these vertices are simple.
Then
the probability that after sampling, the degree of fixed vertex $i$ is two or larger can be upper-bounded by
$p^3 d_i^2$,
where $d_i$ is the degree of the vertex $i$ before sampling; as a consequence,
the expected number of low degree vertices which after sampling have degree at least two is bounded by
$\sum_ip^3d_i^2=o(\sqrt{m_n})\sum_ip^3d_i=o(\sqrt{m_n})p^3 m_n=o(1)$.  This shows that after sampling
the low degree vertices will either have degree one or become isolated.  Therefore, edges within the
sparse part will appear as isolated edges in $\smpl(G_n,t/\sqrt{2m_n})$, contributing to the isolated edge constant $I$.
\item Connections between dense and sparse part.
Since the surviving vertices in the sparse part have
degree one after sampling, these edges contribute to the edge and star densities and thus they appear as stars or isolated edges in $\smpl(G_n,t/\sqrt{2m_n})$.
\end{enumerate}
Note that the vertices of degree $\Omega(\sqrt{m_n})$ do not contribute anything to the graphex limit due to the fact that the probability of such a vertex being observed in the sampling is $o(1)$.
(Note that in general, when the edges out of the low degree vertices have non-trivial multiplicities, we could also get isolated multi-edges as well as stars with edges that have multiplicity bigger than one).

A visualization of $\smpl(G_n, t/\sqrt{2 m_n})$ is given by Figure~\ref{fig:sampled-graph}.
\begin{figure}
\begin{center}
\includegraphics[height = 7cm, width = 7cm]{sampled-graph}
\end{center}
\caption{Typical structure of $\smpl(\CM, t/\sqrt{\ell_n})$.}
\label{fig:sampled-graph}
\end{figure}
The asymptotic structure of $\smpl(G_n,t/\sqrt{2m_n})$ constitutes a network between the dense part described by $W$, stars centered at the high degree vertices described by $S$ representing the edges between the dense and sparse parts, and isolated edges described by $I$ arising from the sparse part.

\paragraph*{Heuristic Derivation of Sampling Limit for the Configuration Model.}

We start by noting that $p$-sampling with $p=t/\sqrt{2e(G)}$ is asymptotically equivalent to what one might want to call Poisson sampling, defined
by first choosing $k$ according to a Poisson random variable with expectation $t|V(G)|/\sqrt{2e(G)}$,
and then choosing $k$ vertices from $V(G)$, uniformly at random with replacement, which in turn is equivalent
to considering a Poisson process of rate  $p=t/\sqrt{2e(G)}$ on $V(G)$.  In the simple graph setting, this
follows from Lemma 5.4 in \cite{BCCV17}, but an inspection of the proof shows that the lemma holds in the multi-graph setting considered here as well.  We also note that for the configuration model, $2e(G_n)=\ell_n(1+o(1))$,
see Lemma~\ref{CM:edge_count} below for the precise statement.
Finally, we couple the Poisson process on $V(G)$ to a Poisson process as the one in Remark~\ref{rem:Multigrphex-distr}, i.e. a Poisson process  $(x_i)$ of rate $t$ on $\R_+$ by
assigning adjacent intervals of width $1/\sqrt{\ell_n}$ to each vertex.  The degree of the vertex $v$ corresponding
to $x_i$ can then easily be seen to be equal to $d_v=\bar\rho_n^{-1}(x_i)\sqrt{\ell_n}$.  Note also that $\bar\rho_n^{-1}(x_i)\sqrt{\ell_n}=0$ if $x_i$ does not correspond to any vertex $v\in [n]$, i.e. if
$x_i\notin[0,n/\sqrt{\ell_n}]$.

Linking back to the above insights on the graph structure of the sampled graph,
we next note
that $\int x\rho_n(\dif x)=\frac 1{\ell_n}\sum_i d_i=1$ and that
\[
\int x\rho(\dif x)\leq\liminf_{n\to\infty}\int x\rho_n(\dif x)=1
\]
by Fatou's lemma.  We therefore interpret $\int x\rho(\dif x)$ as the limiting fraction of (half)-edges whose endpoints have degrees of order $\Theta(\sqrt{\ell_n})$.  Edges between vertices in this part therefore contributed to the graphon part of the limiting graphex.  To ``derive'' the concrete form of this liming graphon, we need one more fact, established in Proposition \ref{lem:poisson-approximation} below.  It states
that in the configuration model $\CM$, the
number of edges created between two disjoint sets of half-edges $S$ and $S'$ of size $O(\sqrt{\ell_n})$ is approximately distributed according to
$\mathrm{Poisson}(|S||S'|/\ell_n)$, and the number of edges connecting such a set $S$ to itself is approximately distributed according to
$\mathrm{Poisson}(|S|^2/2\ell_n)$.  Applied two Poisson points $x_i,x_j$ such that the degrees of the corresponding vertices $v,v'$ are of order
$\sqrt{\ell_n}$, we  then expect to see
$\mathrm{Poisson}(d_vd_{v'}/\ell_n)=\mathrm{Poisson}(\bar\rho_n^{-1}(x_i)\bar\rho_n^{-1}(x_j))$ many edges between
$i$ and $j$, and a loop of multiplicity $\mathrm{Poisson}(d_v^2/2\ell_n)=\mathrm{Poisson}((\bar\rho_n^{-1}(x_i))^2)$
at the vertex $i$, explaining the form of the limiting graphon.

Next, observing that $\int_M^\infty \rho(\dif x)\leq \frac 1M$ and $\int_M^\infty \rho_n(\dif x)\leq \frac 1M$ by Markov's inequality, we see that the high degree vertices don't contribute to
$b_n$ or $b$, showing that $a=\lim_{n\to\infty} b_n-b$ is the liming fraction of half-edges belonging to low degree vertices.
Considering  the set of half-edges $S_L$ attached to some low
degree vertex,  let $S_L^p$ be the set of half-edges surviving after sampling. Then $S_L$
contains approximately $a\ell_n$ many half-edges, showing that $|S_L^p|$ is approximately equal
to $pa\ell_n=ta\sqrt{\ell_n}$.  The number of full edges formed between these is then approximately equal to
$\mathrm{Poisson}((ta\sqrt{\ell_n})^2/2\ell_n)=\mathrm{Poisson}((ta)^2/2)$, motivating the appearance
of the term $a^2/2$ in \eqref{multigraphon-CM}.

Finally, to derive the form of the star intensity $S$ in \eqref{multigraphon-CM}, we consider the edges
between the high and low degree vertices.  A vertex $v$ corresponding to a Poisson point $x_i$ such that
$d_v=\bar\rho_n^{-1}(x_i)\sqrt{\ell_n}$ is of order $\sqrt{\ell_n}$ then approximately has degree
$\mathrm{Poisson}(d_v|S_L^p|/\ell_n)\approx\mathrm{Poisson}(ta \bar\rho_n^{-1}(x_i))$ into $S_L^p$,
explaining the appearance of the term $S$ in \eqref{multigraphon-CM}.

To relate the results for the configuration model to those of the erased configuration model
we  use that  a Poisson random variable with rate $w$ is non-zero with probability $1-e^{-w}$.  This
in turn implies that
asymptotically, the number of non-loop edges in the erased configuration model
is by a factor $c$ smaller than the number of non-loop edges in the original configuration model,
with $c$ given by \eqref{eq:ecm_edges_defn}.  Since sampling convergence of a sequence $G_n$
involves a random coin flip with probabilities
$p=t/\sqrt{2e(G_n)}$, we have to rescale time by a factor $\sqrt c$ when translating our results for the configuration
model to that of the erased configuration model.  This
 leads to the
graphex $\cW_{\sss \mathrm{ECM}}^c$ in Corollary~\ref{cor:ECM-sampl-conv}.

The limit for the bipartite configuration model can be motivated using analogous heuristics.

\paragraph*{Heuristic Derivation of Sampling Limit for the Preferential Attachment Model}

It turns out that the preferential attachment model behaves very much like a configuration model
with degree sequence equal to the expected degrees at time $m_n$, $(\bar d_i)_{i\in[n]}$.
The proof details are different, with
Proposition \ref{lem:poisson-approximation} replaced by Proposition~\ref{prop:PAM-TV} below as well as other differences in the details, but
the essence will again be that we control the dependence of the number of edges between different sets of vertices
and approximate them by suitable Poisson random variable, eventually giving the same limiting graphex as the
configuration model with degree sequence $(\bar d_i)_{i\in[n]}$; see also Remark~\ref{rem:PA=CM-bard} in Section~\ref{sec:proof_PA} below.

\paragraph*{Heuristic Derivation of the Sampling Limit for the Generalized random graph}

It will be convenient to sample vertices with
probability $p'=t/\sqrt{L_n}$ rather than with probability $p=1/\sqrt{2m_n}$ where $m_n$
is the number of
non-loop edges in $\GRG$.
It turns out that, asymptotically, this just corresponds to rescaling of time
by a factor $\sqrt c$, a fact which  follows from the observation that
\[
\frac {1}{L_n}\E[2m_n]=\frac 1{L_n}\sum_{i,j}\frac{w_iw_j}{L_n+w_iw_j}=\int_{0}^\infty \int_0^\infty \frac{xy}{1+xy} \rho_n(\dif x ) \rho_n( \dif y) = c+o(1)
\]
(plus a concentration argument).  This explains the rescaling by $\sqrt c$ in  $\cW_{\sss\mathrm{GRG}}^{c}$, but obviously, not yet the particular form of the limiting
graphex.

To derive the latter, we proceed very similar to our heuristic derivation for the configuration model,
except  that we now consider a core of vertices defined by the  \emph{weights} of the vertices.  Specifically, we consider a core of vertices with weights
$w_i=\Theta(\sqrt{L_n})$ and  a set of low-weight vertices with weights $w_i=o(\sqrt{L_n})$.
It is then again not hard
to argue that  the low weight vertices will have degree at most $1$ after sampling, and it is also clear that asymptotically, the sum of the weights of all low weight vertices is  $aL_n$  with $a$ as in Assumption~\ref{assumption-weight}.

Furthermore, following the steps in our
heuristic derivation of the limiting graphex for the configuration model, replacing the
Poisson number of edges $\mathrm{Poisson}(d_vd_{v'}/\ell_n)$ between two vertices
of degree $d_i,d_j$ by $\mathrm{Bern}(p_{ij})$ (with $p_{ij}$ given in
\eqref{pij}), the reader can now easily ``derive'' the form of the limiting graphon for $\GRG$.
To obtain the other two parts of the limiting graphex, we approximate the probability \eqref{pij} for an edge
between two low weight vertices (or a low and high weight vertex) by
${w_iw_j}/{L_n}$
and approximate the sum of independent Bernoulli random variables by a Poisson random variable; using these
approximations, the ``derivation'' of the limiting graphex is now very similar to that for the limiting
graphex for the configuration model.

\paragraph*{Relation to Caron-Fox graph process}
Corollary~\ref{cor:pure-graphon} establishes that the sampling limit of certain $\ECM$ random graphs is given by the random graph model introduced by \citet{CF17}
(see \cite[Section~3]{CF17}). Thus our result gives a new perspective on the Caron-Fox random graph. Indeed, certain Caron-Fox graphs may be looked upon as sampling limits of suitable $\ECM$ random graphs.
\cite{BCCV17} characterizes graphex processes as the limits of sampling convergent graph sequences, and thus conceptually clarifies the innate importance of these processes. Our result has a similar conceptual interpretation, in that it identifies a prominent graphex process, i.e., the Caron-Fox process, as the sampling limit of a natural sequence of random graphs.
Put differently, rather than obtaining the model by first postulating exchangeability
of a rather abstract random measure on $\R_+^2$, then invoking Kallenberg's representation theorem
and finally making further simplifications  to arrive  at the final model, our results \emph{derive} the Caron Fox graph as a  sub-sample of an underlying
latent configuration model.
In turn, this further reinforces the importance of the Caron-Fox model, and provides some practical insights into its suitability as a model in real applications.

\paragraph*{Outline}
The rest of the paper is structured as follows. We prove Theorem~\ref{thm:main-CM} and the associated corollaries in Section~\ref{sec:proof-CM}, Theorem \ref{thm:PAM-limit} in Section~\ref{sec:proof_PA}, Theorem \ref{thm:GRG-limit} in Section~\ref{sec:grg_proof}, and Theorem~\ref{thm:main-BCM} in Section~\ref{sec:bcm_proof}.  For completeness, we collect some properties of Completely random measures in Appendix \ref{sec:crm}. In Appendix~\ref{sec:properties}, we compute some functionals of specific random adjacency measures arising in the proofs of Theorems~\ref{thm:main-CM} and~\ref{thm:GRG-limit} respectively.
Finally, Appendix~\ref{appendix:rescaling} establishes some facts about random adjacency measures under rescaling.

\section{Proof for configuration model results}
\label{sec:proof-CM}
Our proofs rely on one lemma and three propositions.
For any (multi)-graph $G$, let $e(G)$ denote the number of non-loop edges in $G$.

\begin{lemma}[{Non-loop edges in $\CM$}]\label{CM:edge_count}
 As $n\to\infty$, $\frac{e(\CM)}{\ell_n} \to 1/2$ a.s. $\PRCM$. Further, as $n \to \infty$,  for all $\varepsilon>0$
\begin{align}
\frac{2\E[e(\ECM)]}{\ell_n} - \int_0^{\infty} \int_0^{\infty} (1- \e^{-xy}) \rho_n(\dif x) \rho_n (\dif y) \to 0. \label{eq:ecm_expected}\\
\PR_{\sss\mathrm{ECM}}( |e(\ECM) - \E[e(\ECM)] | > \varepsilon \ell_n  ) \leq 2\exp( - C_0 \varepsilon \ell_n) \label{eq:ecm_fluctuations}
\end{align}
for some universal constant $C_0 >0$.

\end{lemma}

\begin{proposition}\label{lem:poisson-approximation}
Let $\cE_n(S,S')$ denote the number of edges created between the set of half-edges $S$ and $S'$ in the construction of $\CM$.
Consider $k$ disjoint subsets of half-edges $(S_j)_{j\in [k]}$ such that $|S_j| = s_j = O(\sqrt{\ell_n})$ for all $j\in [k]$.
Let $\bld{\cE}_n = (\cE_n(S_i,S_j))_{1\leq i \leq j \leq k}$, $\bld{\cE} := (\cE_{ij})_{1\leq i\leq j\leq k}$, where $\bld{\cE}$ is an independent collection and $\cE_{ij}\sim \mathrm{Poisson} (s_is_j/\ell_n)$ for $i\neq j$, $\cE_{ii} \sim \mathrm{Poisson}(s_i^2/2\ell_n)$.
Then, as $n\to\infty$,
\begin{equation}\label{eq:total-var-conv-edges}
\dTV(\bld{\cE}_n,\bld{\cE}) \to 0,
\end{equation}where $\dTV(\cdot,\cdot)$ denotes the total variation distance.
Moreover, if $S_j$'s are random disjoint subsets chosen independently of $\CM$ and
satisfying
$\E[s_j] = O(\sqrt{\ell_n})$ for all $j\in [k]$, then  $\lim_{n\to\infty}\dTV(\bld{\cE}_n,\bld{\cE}) = 0$,
where both $\bld{\cE}_n$ and $\bld{\cE}$ refer to the joint distribution, including in particular the
randomness stemming from the random sets $S_j$'s.
\end{proposition}

To state the next proposition, we recall the definition of $Y_n(t)$ from
Theorem~\ref{thm:main-CM}.  For $A\in \sB(\R_+)$, let $V_n(A)$ be the set of vertices obtained by labeling the vertices in $[n]$ uniformly from $[0,\sqrt{\ell_n}]$ and then retaining the vertices with labels in $A$.  This induces a random measure $\bar{S}_n$
on $\R_+$ via  $\bar{S}_n(A)=\frac{1}{\sqrt{\ell_n}} \sum_{i\in V_n(A)} d_i$ that is related to $Y_n(t)$ via $Y_n(t)=\bar{S}_n([0,t])$.
As we will see, the next proposition immediately implies that distributional convergence of the
measure $\mu_n$ defined in Theorem~\ref{thm:main-CM} is equivalent to the convergence of the finite-dimensional distributions of $(Y_n(t))_{t\geq 0}$.

\begin{proposition}\label{lem:conv-sum-degree}
For any disjoint collection of sets $(A_j)_{j\in [k]}$ from $\sB(\R_+)$, and $\alpha = (\alpha_j)_{j\in [k]}\in \R^k$, define $\Phi_{n}(A_1,\dots,A_k):= \E[\e^{\ri\sum_{j\in [k]}\alpha_j \bar{S}_n(A_j)}] $.
If $\max_i\delta_i=o(\ell_n)$, then
\begin{equation}\label{eq:char-funct-degree}
\Phi_{n}(A_1,\dots, A_k) = \exp\Big(
\sum_{j\in [k]} \Lambda(A_j) \int (\e^{\ri \alpha_j x}-1)\rho_n(\dif x)
+o(1)
\Big).
\end{equation}
\end{proposition}

\begin{proposition}\label{lem:concentration-PP}
Let $\xi_n$  denote the point process $\lbl_{\sqrt{\ell_n}}(\CM)$ on $\R_+^2$.
For any $A,B\in \sB(\R_+)$, $l\in \N^*$, and $\delta>0$,
\begin{align}\label{switch-bd}
&\PR \Big( \big|\PR\big(\xi_{n}(A\times {B}) = l \vert \CM\big) - \PR\big(\xi_{n}(A\times B) = l\big)\big|>\delta \Big) \nonumber\\
& \leq 2\exp \Big(-\frac{\delta^2\ell_n}{{(12\Lambda(A)\Lambda(B))^2}}\Big).
\end{align}
Consequently, $\PR_{\sss \mathrm{CM}}$ a.s., $d_{L}( \cL(\xi_n | \CM), \cL(\xi_n)) \to 0$ as $n \to \infty$, where $d_{L}(\cdot, \cdot)$ denotes the L\'{e}vy- Prohorov metric on $\cP(\cM(\R^2_+))$.
\end{proposition}

We first establish Theorem~\ref{thm:main-CM} and its corollaries given Lemma~\ref{CM:edge_count},
Propositions \ref{lem:poisson-approximation}, \ref{lem:conv-sum-degree} and \ref{lem:concentration-PP}, and defer the proofs of the lemma and propositions to the end of the section.

\paragraph*{Proofs of Theorem~\ref{thm:main-CM}, Corollaries~\ref{cor:ECM-sampl-conv}, \ref{cor:pure-graphon} and \ref{cor:pure-isolate}.}\hss

\noindent {\em Proof of Theorem~\ref{thm:main-CM}.} {$(ii) \Leftrightarrow (iii)$.}
If $\{(v_i,\theta_i)\}_{i\geq 1}$ is a unit rate Poisson point process on $\R_+^2$,
then $\{(\bar{\rho}_n^{-1}(v_i),\theta_i)\}_{i\geq 1}$ is a Poisson process with intensity
measure $\rho_n\times\lambda$, showing that  $\mu_n\sim \CRM(0,\rho_n\times\lambda)$.
Let $X_n(t) = \mu_n([0,t])$.
Then  $(X_n(t))_{t\geq 0}$ is a L\'evy process
(see Appendix~\ref{sec:crm} for the definition and some important properties of L\'evy processes)
with characteristic function
\begin{eq}
\label{X_n-characteristic}
\E[\e^{\ri \theta X_n(t)}] &= \E[\e^{\ri \theta\mu_n([0,t])}]=
\exp \bigg(t \int (\e^{\theta\ri x}-1 )\rho_n(\dif x)\bigg)\\
&=\exp \Bigg( t\bigg(\ri\theta b_n + t\int (\e^{\ri \theta x}-1
- \ri \theta(x\wedge 1))\rho_n(\dif x)\bigg)\Bigg),
\end{eq}
where the third step follows using Lemma~\ref{lem:CF-CRM}.
Using the standard terminology for the theory of L\'evy processes, see Appendix~\ref{sec:crm}, $(X_n(t))_{t\geq 0}$ is a L\'evy process with
characteristics $(b_n,0,\rho_n)$.
By  \cite[Lemma 11.1.XI]{DV08} statement $(ii)$ is equivalent to assuming that $(X_n(t))_{t\geq 0}$ converges as a stochastic process in $\mathbb{D}(\R_+,\R_+)$.
Therefore, \cite[Chapter VII, Corollary 3.6]{JS03} (restated as Lemma~\ref{lem-js03} in Appendix~\ref{sec:crm} for the special case of L\'evy processes) implies that (ii) is equivalent to the following two conditions:  $b_n \to b$
and $\int f \dif \rho_n \to \int f \dif \rho$ for all bounded continuous functions $f : \R_+ \to \R$ vanishing near zero,
which by Lemma~\ref{lem:bn} is equivalent to (iii).  Note that
Lemma~\ref{lem:bn} also proves Remark~\ref{rem:a=low-degrees}.

\paragraph*{$(ii)\Leftrightarrow (iv)$.}

By Proposition~\ref{lem:conv-sum-degree} and \eqref{X_n-characteristic}, convergence of the
finite dimensional distributions of $(Y_n(t))_{t\geq 0}$ is equivalent to convergence of the characteristic
functions of $X_n(t)$ for all $t$.  Since $X_n(t)$ is a  L\'evy process, this in turn is equivalent to
the convergence of this process in law, which is equivalent to $(ii)$.

All that remains to show is therefore tightness of
$(Y_n(t))_{t\geq 0}$ in $\mathbb{D}(\R_+,\R_+)$.
To this end, we  note that for $t>u>s$
\begin{align}\label{eq:tightness-Y-n}
\E[(Y_n(t)-Y_n(u))(Y_n(u)-Y_n(s))]
&=\frac 1{\ell_n^2}\sum_{i\neq j}d_id_j(t-u)(u-s) \nonumber\\
&\leq \frac 14(t-s)^2.
\end{align}
Tightness of $(Y_n(t))_{t\geq 0}$ follows using
\cite[Chapter VI, Theorem 4.1]{JS03}.

\paragraph*{$(i)\Leftrightarrow (iv)$.}
By Proposition~\ref{prop:samp-embed-lbl-equiv}, statement (i) is equivalent to $\PR_{\sss \mathrm{CM}}$ almost sure
convergence of  $(\lbl(\CM))_{n\geq 1}$ in distribution.
To compare this to $\PR_{\sss \mathrm{CM}}$ almost sure convergence of
$\xi_n$ in distribution, we will use
the fact that by Lemma~\ref{CM:edge_count},
 $2e(\CM)/\ell_n\to 1$ $\PR_{\sss \mathrm{CM}}$ almost surely.
This in turn implies that $\PR_{\sss \mathrm{CM}}$ almost sure convergence of  $(\lbl(\CM))_{n\geq 1}$ in distribution is
equivalent to $\PR_{\sss \mathrm{CM}}$ almost sure convergence of
$\xi_n$ in distribution (the formal argument requires a technical lemma, Lemma~\ref{lemma:rescaling}
from Appendix~\ref{appendix:rescaling}).  Combining these facts
with Proposition~\ref{lem:concentration-PP} and the fact that  $\ell_n=\Omega(\log n)$ we concluded that (i) is equivalent to the statement that
$\cL(\xi_{n}(A_i \times A_j))_{1\leq i\leq j\leq  k}$
converges
for any $k\geq 1$ and disjoint sets $A_i\in \sB(\R_+)$, $i=1,\dots,k$.

Given a collection of labels $(U_j)_{j=1,\dots,n}$ chosen i.i.d.\ uniformly at random in
$[0,\sqrt{\ell_n}]$, let $V_n(A_i)$ be the set of vertices with label in $A_i$, and let
$S_i$ be the set of half-edges whose endpoint is in $V_n(A_i)$.  As before, let
$\bar{S}_n(A)=\frac{1}{\sqrt{\ell_n}} \sum_{j\in V_n(A)} d_j$.  Then
$|S_i|=\bar{S}_n(A_i)\sqrt{\ell_n}$ and
$\E[|S_i|]= \Lambda(A_i)\sqrt{\ell_n}$,
so by
Proposition~\ref{lem:poisson-approximation} convergence of
$\cL(\xi_{n}(A_i \times A_j))_{1\leq i\leq j\leq  k}$
is equivalent to distributional convergence  of
$\{\mathrm{Poisson} (\bar S_n(A_i)\bar S_n(A_j))\}_{i,j\in [k]}$,
which in turn is equivalent to distributional convergence
of  $\{\bar S_n(A_i)\bar S_n(A_j)\}_{i,j\in [k]}$.
The latter clearly implies convergence of the random vector $\{(\bar S_n(A_i))^2\}_{i\in [k]}$,
and since $\bar S_n(A_i)\geq 0$, this in turn implies convergence of $\{\bar S_n(A_i))\}_{i\in [k]}$.
Conversely, the latter clearly implies convergence of $\{\bar S_n(A_i)\bar S_n(A_j)\}_{i,j\in [k]}$, so we have shown that
(i) is equivalent to convergence of the finite dimensional distributions of $(Y_n(t))_{t\geq 0}$.  To show that this
is equivalent to (iv), we use the tightness condition \eqref{eq:tightness-Y-n}.

Finally, to obtain the required descriptions for the limiting objects, note that $(ii)$ implies $\cL(\mu_n)$ converges weakly.
Since $\mu_n$ is completely random, it follows that there exists a completely random measure $\mu$ such that $\cL(\mu_n) \to \cL(\mu)$ in $\cP(\cN(\R_+))$. Thus, $\mu$ admits a representation \eqref{eq:CRM-form}.
Moreover, the convergence of the characteristics of the L\'evy process $(X_n(t))_{t\geq 0}$ yields that $a = \lim_{n\to\infty}\int(x\wedge 1)\rho_n(\dif x) - \int(x\wedge 1)\rho(\dif x)$.  Finally, Proposition~\ref{lem:conv-sum-degree}, \eqref{X_n-characteristic} and convergence of the L\'evy process $(X_n(t))_{t\geq 0}$
to a L\'evy process with characteristics $(a,0,\rho)$
 gives that
\begin{equation}
\Phi_{n}(A_1,\dots, A_k) \to \exp\Big(\sum_{j\in [k]} \Lambda(A_j)\Big(a+ \int (\e^{\ri \alpha_j x}-1)\rho(\dif x)\Big)\Big),
\end{equation}
for any disjoint collection of sets $(A_j)_{j\in [k]}$ from $\sB(\R_+)$.
This shows that $\bar{S}_n$ converges to the completely random measure $\mu$, and that $(Y_n(t))_{t\geq 0}$ converges to the L\'evy process $(\mu([0,t]))_{t\geq 0}$.  Using the convergence of
$\bar{S}_n$ to $\mu$ and following the argument from the proof of {$(i)\Leftrightarrow (iv)$} we then get that in distribution,
$(\xi_{n}(A_i \times A_j))_{1\leq i\leq j\leq  k}$ converges to
$\{\mathrm{Poisson} (\mu(A_i)\mu(A_j))\}_{i,j\in [k]}$.  As established in Lemma~\ref{lem:xi_WCM}, this is equal in distribution to
$(\xi_{\sss \cW_{\sss \mathrm{CM}}}(A_i \times A_j))_{1\leq i\leq j\leq  k}$, as required.

\qed

\begin{proof}[Proof of Corollary~\ref{cor:ECM-sampl-conv}]
To establish this corollary, we first note that
\begin{align}
0\leq \int_0^{\infty} \int_0^{\infty} (1 - e^{-xy}) \rho_n(\dif x) \rho_n (\dif y) \leq 1, \nonumber
\end{align}
using $1 - e^{-x} \leq x$ for $x \geq 0$, and $\int_0^{\infty} x \rho_n(\dif x) =1$.
Thus this sequence is compact, and  equivalently, every sequence has a convergent subsequence. Let us assume
\begin{align}
\int_0^{\infty} \int_0^{\infty} (1 - e^{-xy}) \rho_n (\dif x) \rho_n(\dif y) \to c >0, \nonumber
\end{align}
along a subsequence. The proof of Theorem~\ref{thm:main-CM} implies that $\lbl_{\sqrt{l_n}}(\ECM)$ converges weakly to the random adjacency measure corresponding to the graphex $\cW_{\sss \mathrm{ECM}}^1$. The proof is now complete, once we use Lemma~\ref{CM:edge_count} and Lemma~\ref{lemma:rescaling}.
\end{proof}
\noindent
Next, we prove Corollary~\ref{cor:pure-graphon}. It is easy to see that in this specific case, the result follows almost directly from \cite[Remark 14]{BCCH16}. Below, we provide a more detailed  proof from first principles. We feel that this proof is more intuitive, and also more generally applicable, as evidenced by its easy adaptation to establish Corollary~\ref{lemma:grg_cut-metric} in Section~\ref{sec:grg_proof}.

\begin{proof}[Proof of Corollary~\ref{cor:pure-graphon}]
By Theorem~\ref{thm:main-CM},
$I_{\sss\mathrm{CM}} = 0$ if and only if
$a=0$, which by Remark~\ref{rem:a=low-degrees}
is equivalent to
\begin{align}
\lim_{\varepsilon\to 0} \limsup_{n\to\infty} \int_0^\varepsilon x\rho_n(\dif x) =0, \label{eq:unif_reg_tail}
\end{align}
which by Definition~\ref{defn-stretched-cut-metric}   is equivalent to uniform tail regularity of
$\CM$.

We now establish that \eqref{eq:unif_reg_tail} implies that the sequence $\ECM$ is uniform tail regular $\PR_{\sss \mathrm{ECM}}$ a.s.
To this end, set $V_{>\eps} = \{ i : d_i > \varepsilon \sqrt{\ell_n} \}$,
$V_{\leq\eps}=V^c_{>\eps} $
and denote by $\mathcal{E}(V_{>\eps})$ the number of edges in $\CM$ with both end points in $V_{>\eps}$. We have, by direct computation,
\begin{align*}
\E[ \mathcal{E}(V_{> \eps}) ]
= \frac 12(1+o(1))\ell_n \Big( 1- \int_0^{\varepsilon } x \rho_n(\dif x) \Big)^2= \ell_n \Big(\frac 12 - \Err(\varepsilon,n)\Big), \nonumber
\end{align*}
where $\Err( \varepsilon,n)$ is an error term such that $\lim_{\varepsilon \to 0} \limsup_{n \to \infty} \Err(\varepsilon,n) =0$.
By Lemma~\ref{lem:switching-lemma}  below,
 there exist a constant $C_0>0$ such that for all $\delta_n>0$
\begin{align}
\PR\Big( | \mathcal{E}(V_{> \eps}) - \E[ \mathcal{E}(V_{> \eps})] | \geq
\delta_n\ell_n\Big)
\leq 2\exp ( - C_0 \ell_n\delta_n^2).
\nonumber
\end{align}
Since $\ell_n=\omega(\log n)$, we can choose $\delta_n$ in such a way that
this bound is summable and $\delta_n=o(1)$.  Thus
$\PR_{\sss \mathrm{CM}}$ a.s.,
\begin{align}
e(\CM)-\mathcal{E}(V_{> \eps})
\leq \Err(\varepsilon,n)\ell_n +o(\ell_n).
\label{eq:good_edges_lower}
\end{align}
Let $\mathcal{E}_{\sss \mathrm{ECM}}(\varepsilon)$ denote the number of edges in $\ECM$ with at least one end point in $V_{\leq\varepsilon}$.
Since this is bounded by the number of edges in $\CM$ with at least one end point in $V_{\leq\eps}$,
 \eqref{eq:good_edges_lower} implies that $\PR_{\sss \mathrm{ECM}}$ a.s., $\mathcal{E}_{\sss \mathrm{ECM}}(\varepsilon) \leq \ell_n \Err(\varepsilon,n) + o(\ell_n)$.
Now, recall the definitions  of stretched canonical graphon  \cite[Section 2.3]{BCCH16} and uniformly tail regular graphs \cite[Definition 13]{BCCH16}.
Let $W_n$ denote the  stretched canonical graphon for $\ECM$.
Set $\varepsilon >0$, and consider $U_n$ to be the set corresponding to the vertices in $V_{>\eps}$. With this choice of $U_n$,
\begin{equation}\label{tail-regularity-1}
\| W_n - W_n \mathbbm{1}\{U_n \times U_n\} \|_{1} \leq \Err(\varepsilon, n) + o(1),
\end{equation}
where $\|\cdot\|_1$ denotes the $L_1$ norm.
Note here that for constructing the stretched canonical graphon, the space is scaled with $\sqrt{e(\ECM)}$, rather than $\sqrt{\ell_n}$, but that does not change the order in \eqref{tail-regularity-1} due to Lemma~\ref{CM:edge_count}.
To derive the required claim, it suffices to show that the Lebesgue measure of the set corresponding to the vertices in $U_n$ is uniformly bounded in $n$. Upon direct computation, we note that this measure is exactly $\rho_n((\varepsilon, \infty)) \leq \frac{1}{\varepsilon} \int_0^{\infty} x \rho_n(\dif x) = \frac{1}{\varepsilon}$. This establishes
that $\ECM$ is uniformly tail regular a.s. $\PR_{\sss \mathrm{ECM}}$.
Finally, note that convergence in stretched cut metric follows immediately from \cite[Theorem 5.5]{BCCV17}.
\end{proof}

\begin{proof}[Proof of Corollary~\ref{cor:pure-isolate}]
The corollary is an immediate consequence of Theorem~\ref{thm:main-CM}.

\end{proof}

\paragraph*{Proofs of Propositions~\ref{lem:conv-sum-degree} and \ref{lem:concentration-PP} and Lemma~\ref{CM:edge_count}.}\hss

\begin{proof}[Proof of Proposition~\ref{lem:conv-sum-degree}]
Let $l_i$ denote the label for vertex $i$ that is uniformly distributed over the interval $[0,\sqrt{\ell_n}]$, independently over $i\in [n]$.
Note that,
\begin{align*}
\Phi(A_1,\dots,A_k) &= \prod_{i\in [n]}\E\Big[\e^{ \frac{\ri d_i}{\sqrt{\ell_n}}\sum_{j\in [k]}\alpha_j\ind{l_i\in A_j}}\Big] \\
&= \prod_{i\in [n]}\bigg(1-\frac{1}{\sqrt{\ell_n}}\sum_{j\in [k]}\Lambda(A_j) +\frac{1}{\sqrt{\ell_n}}\sum_{j\in [k]}\Lambda(A_j)\e^{\frac{\ri d_i\alpha_j}{\sqrt{\ell_n}} }\bigg) \\
&= \prod_{i\in [n]}\exp \bigg(\frac{1}{\sqrt{\ell_n}}\sum_{j\in [k]}\Lambda(A_j) \bigg(\e^{\frac{\ri d_i\alpha_j}{\sqrt{\ell_n}} }-1\bigg)+O(d_i^2/\ell_n^2)\bigg) \\
&= \exp\Big(\sum_{j\in [k]}\Lambda(A_j)\int (\e^{\ri \alpha_j x}-1)\rho_n(\dif x) +o(1)\Big),
\end{align*}
where in the last step we used that $\sum_id^2_i\leq \max_i d_i\sum_i d_i=o(\ell_n^2)$.  This
completes the proof.
\end{proof}

In the next two proofs, we will use the following simple switching lemma.

\begin{lemma}[Switching lemma]\label{lem:switching-lemma} Let $P$ be a perfect matching of the half-edges, and $P'$ be another perfect matching which can be obtained from $P$ by one switch.
Let $X$ be a function on perfect matchings satisfying the Lipchitz condition $|X(P) - X(P')| \leq c$.
Then, for any $\varepsilon>0$,
\begin{equation}
\PR(|X-E[X]|>\varepsilon) \leq 2 \e^{-\frac{\varepsilon^2}{\ell_n c^2}}.
\end{equation}
\end{lemma}

\begin{proof}
This  follows using identical arguments as \cite[Theorem 2.19]{Wor99}. Note that \cite[Theorem 2.19]{Wor99} was only stated for random regular graphs, but the same argument works for the general configuration model as well.
\end{proof}

\begin{proof}[Proof of Proposition~\ref{lem:concentration-PP}]
Recall that an instance of $\CM$ is generated by choosing a uniformly random matching $P$ of the $\ell_n$ half-edges
corresponding to $\bld{d}$.  Denoting the multi-graph corresponding to a matching $P$ by $G_n=G_n(P)$, we
 apply Lemma~\ref{lem:switching-lemma}
 with $X_n (P)= \PR(\xi_{\sss G_n}(
 P
= l \vert G_n)$.
Consider two matchings $P$ and $P'$ that differ by at most one switch, and label the half-edges involved in this switch by $1,\dots, 4$ in such a way that in $P$,  $13$ and $24$ are matched, and in $P'$, $14$ and $23$ are matched.  Then
$\xi_{\sss G_n}( A \times B)$ remains unchanged between $P$ and $P'$ unless at least one of the
half edges $1,\dots,4$ has a label in $A$, and a second, different one, has a label in $B$.
A union bound then easily shows that
 $|X_n(P)-X_n(P')|\leq 12\Lambda(A)\Lambda(B)/\ell_n$. This proves \eqref{switch-bd}.

The proof of the final statment follows by observing that $\cM( \R^2_+)$, equipped with vague topology is a Polish space, and that weak convergence on such spaces is determined by countable classes of sets.
\end{proof}

\begin{proof}[Proof of Lemma~\ref{CM:edge_count}]

The expected number of loops in $\CM$ is given by
\begin{eq}\label{expt-loop}
 &\frac{1}{2}\E\bigg[\sum_{i\in [n]}\sum_{j\in [d_i]} \ind{j\text{-th half-edge of the }i\text{-th vertex creates a loop}}\bigg] \\
 &= \frac{1}{2}\sum_{i\in [n]}\sum_{j\in [d_i]} \frac{d_i-1}{\ell_n-1} \leq  \sum_{i\in [n]} \frac{d_i^2}{\ell_n}\leq  \max_{i\in [n]} d_ i = o(\ell_n).
\end{eq}
This immediately proves that $\E[e(\CM)]/ \ell_n \to 1/2$ as $n \to \infty$.
To prove almost sure convergence, we use Lemma~\ref{lem:switching-lemma}.
Note that if $P$ and $P'$ are two perfect matchings differing by at most one switch, $e(\CM)/\ell_n$ might change by at most $2/\ell_n$. This gives the required concentration.

Next, we compute the expected number of edges in $\ECM$. To this end, we compute first the expected number of multiple edges. Let $X_{ij}$ denote the number of edges between $i$ and $j$. Thus the total number of multiple edges is given by
\begin{equation}
\label{mul-edge-formula}
\sum_{i<j} (X_{ij}-1)\ind{X_{ij}\geq 2}.
\end{equation}
Now,
\begin{align}
&\E[(X_{ij}-1) \ind{X_{ij}\geq 2}] = \E[X_{ij}\ind{X_{ij}\geq 2}] -\PR(X_{ij}\geq 2)\nonumber \\
&= \E[X_{ij}] - \E[X_{ij}\ind{X_{ij}\leq 1}] - \PR(X_{ij} \geq 2) = \E[X_{ij}] -\PR(X_{ij} \geq 1)\nonumber \\
&= \E[X_{ij}] -1 +\PR(X_{ij} =0) =   \frac{d_id_j}{\ell_n-1}-1+\e^{-\frac{d_id_j}{\ell_n}} +
O\Big(\frac{d_i^2d_j+d_j^2d_i}{\ell_n^2}\Big),\label{mul-edge-i-j}
\end{align} where we have used that
\begin{equation}
0\leq \PR(X_{ij} = 0) -
\prod_{t=0}^{d_i-1}\Bigl(1-\frac{d_j}{\ell_n-1-2t}\Bigr)
\leq \frac{d_i^2d_j}{(\ell_n-2d_i)^2},
\end{equation}
see, e.g.,  \cite[(4.9)]{HHM05}, together with the fact that
\[
\Bigl(1-\frac{d_j}{\ell_n-1-2t}\Bigr)=\exp\Bigl(-\frac{d_j}{\ell_n}\Bigr)\Bigl(1+O\Bigl(\frac {d_id_j+d_j^2}{\ell_n^2}\Bigr)\Bigr).
\]

Therefore, \eqref{mul-edge-formula} and \eqref{mul-edge-i-j} together with \eqref{expt-loop} imply that the expected number of edges in $\ECM$ is
\begin{eq}
&\frac{\ell_n}{2} - \sum_{i<j} \bigg( \frac{d_id_j}{\ell_n}-1+\e^{-\frac{d_id_j}{\ell_n}}  \bigg)	 + o(\ell_n) = \frac{1}{2} \sum_{i\neq j}\big(1-\e^{-\frac{d_id_j}{\ell_n}}\big) +o(\ell_n) \\
&= \frac{\ell_n}{2} \int_0^\infty \int_0^\infty (1-\e^{-xy}) \rho_n(\dif x) \rho_n (\dif y) +o(\ell_n),
\end{eq}where we have again used the fact that $\max_{i\in [n]}d_i = o(\ell_n)$. This proves \eqref{eq:ecm_expected}.

To prove \eqref{eq:ecm_fluctuations}, we again use Lemma~\ref{lem:switching-lemma}, noting that the number of edges in $\ECM$ becomes fixed once the uniform matching of half-edges has been fixed, and that a switch can alter this function by at most a bounded constant.
This completes the proof.
\end{proof}

\subsection{Edge counts for configuration model} \label{sec:edge-count-CM}
\begin{proof}[Proof of Proposition~\ref{lem:poisson-approximation}]We prove this for $k=2$ and the general case follows similarly.
Thus, we need to show that
\begin{align*}
 \dTV\bigg(\cL(\cE_n(S_1,S_1),\cE_n(S_1, S_2),\cE_n(S_2,S_2)), \poi\Big(\frac{s_1^2}{2\ell_n}\Big)\otimes \poi\Big(\frac{s_1s_2}{\ell_n}\Big) \otimes \poi\Big(\frac{s_2^2}{2\ell_n}\Big)\bigg)\to 0.
\end{align*}
For simplicity, we write $S=S_1\cup S_2$ and  $s=|S| = s_1+s_2$.
Let us enumerate the half-edges in $S_1$ arbitrarily by $\{1,\dots,s_1\}$ and the half-edges in $S_2$ by $\{s_1+1,\dots, s\}$.
We first pair the half-edges of $S_1$ and then the remaining unpaired half-edges of $S_2$.
Consider sequential pairing of the half-edges of $S$ and at step $\alpha$, $\alpha=1,\dots,{s}$, we take the half-edge labeled $\alpha$ and if it is not already paired to some previous half-edge, we pair it with another unpaired half-edge chosen uniformly at random. Let $I_\alpha$ denote the three dimensional Bernoulli random vector where
\begin{equation}\label{eq:defn-0}
I_\alpha = \begin{cases}(1,0,0), &\text{if }\alpha\in S_1 \text{, half-edge }\alpha \text{ is not paired previously}\\ &\text{  and it is paired with another half-edge in }S_1 \\
(0,1,0), &\text{if }\alpha\in S_1 \text{, half-edge } \alpha \text{ is not paired previously}\\ &\text{  and it is paired with another half-edge in }S_2 \\
(0,0,1), &\text{if }\alpha\in S_2 \text{, half-edge }\alpha \text{ is not paired previously}\\ &\text{  and it is paired with another half-edge in }S_2 \\
(0,0,0), &\text{ otherwise.}
 \end{cases}
 \end{equation}
Note that $\sum_{\alpha = 1}^{s}I_\alpha = (\cE_n(S_1,S_1),\cE_n(S_1, S_2),{\cE_n(S_2,S_2)})$.
We first couple the $I_\alpha's$ to independent multivariate Bernoulli random variables, and then use Stein's method to obtain multivariate Poisson approximation for \\
 $(\cE_n(S_1,S_1),\cE_n(S_1, S_2),\cE_n(S_2,S_2))$.

\paragraph*{Coupling}
We approximate the collection $(I_\alpha)_{\alpha=1}^{s}$ with a collection of independent
random variables $(\hat{I}_\alpha)_{\alpha=1}^{s}$.
To this end, we describe an algorithm that sequentially pairs the half-edges and keeps track of a special set of half-edges, called \emph{bad half-edges}.
Let $\mathcal{B}_\alpha$ denote the set of bad half-edges at step $\alpha$.
Initially, all the half-edges are \emph{non-bad}, i.e., $\mathcal{B}_0 = \varnothing$.
As before, the half-edges of $S_1$ take labels in $\{1,\dots, s_1\}$ and the half-edges of $S_2$ take labels in $\{s_1+1,\dots, s\}$.
At stage $\alpha$, we pair the half-edge labeled $\alpha$ (call it $e_\alpha$) in
$S$
to a \emph{uniformly chosen} half edge from $\{e_1,\dots,e_{\alpha}\}^c$ (call it $f_\alpha$).
If $e_{\alpha} \in \mathcal{B}_{\alpha-1}$, we set
$\mathcal{B}_{\alpha}=\mathcal{B}_{\alpha-1}$.
If  $e_{\alpha} \in \mathcal{B}_{\alpha-1}^c$ and $f_\alpha\in \mathcal{B}_{\alpha-1}^c$, we set $\mathcal{B}_\alpha = \mathcal{B}_{\alpha-1}\cup \{f_{\alpha}\}$ and $g_\alpha=f_\alpha$.  Finally, if $e_{\alpha} \in \mathcal{B}_{\alpha-1}^c$ and $f_\alpha\in \mathcal{B}_{\alpha-1}$, we choose $g_{\alpha}$ uniformly at random from $\mathcal{B}_{\alpha-1}^c\cap \{e_1,\dots,e_\alpha\}^c$ independently and set
 $\mathcal{B}_\alpha = \mathcal{B}_{\alpha-1}\cup \{ g_\alpha\}$.
Under this scheme, we define
\begin{eq}\label{eq:defn-I}
\hat{I}_\alpha =
\begin{cases}(1,0,0), &\text{if }e_\alpha\in S_1,\ f_\alpha \in S_1 \\
    (0,1,0), &\text{if }e_\alpha\in S_1,\ f_\alpha \in S_2 \\
    (0,0,1), &\text{if }e_\alpha\in S_2,\ f_\alpha \in S_2 \\
    (0,0,0), &\text{ otherwise.}
 \end{cases}, \quad
I_\alpha =
\begin{cases}
    (1,0,0)& \text{if }e_\alpha\in S_1\cap\mathcal{B}_{\alpha-1}^c,\ g_\alpha \in S_1\\
    (0,1,0)& \text{if }e_\alpha\in S_1\cap \mathcal{B}_{\alpha-1}^c,\ g_\alpha \in S_2\\
    (0,0,1)& \text{if }e_\alpha\in S_2\cap \mathcal{B}_{\alpha-1}^c,\ g_\alpha \in S_2\\
    (0,0,0) & \text{otherwise}.
 \end{cases}
 \end{eq}
 Note that $(\hat{I}_\alpha)_{\alpha=1}^{s}$ is an independent collection, and for $\alpha=1,\dots,s_1$, $\hat{I}_\alpha=(1,0,0)$ with probability $(s_1-\alpha)/(\ell_n-\alpha)$,
$\hat{I}_\alpha=(0,1,0)$ with probability $s_2/(\ell_n-\alpha)$, and zero otherwise.
Furthermore,
for $\alpha = s_1+1,\dots, s$,  $\hat{I}_\alpha=(0,0,1)$  with probability $(s_2-(\alpha-s_1))/(\ell_n-\alpha)$, and zero otherwise.
Moreover, the distribution of $(I_\alpha)_{\alpha=1}^{s}$ is same as described in \eqref{eq:defn-0}.

Next, we investigate the probability that the two random variables in \eqref{eq:defn-I} are unequal.
Firstly, if $e_\alpha, f_\alpha \in \mathcal{B}_{\alpha-1}^c$, then $\hat{I}_\alpha = I_\alpha.$
Next, at stage $\alpha$, $e_\alpha \in \mathcal{B}_{\alpha-1}$ if it was already paired previously to a non-bad half-edge.
Now, there were at most $s$ previous steps performed and at each of those steps, the probability of pairing with the half-edge labeled $\alpha$ is at most $1/(\ell_n-s)$, so that
$\PR(e_\alpha\in \mathcal{B}_{\alpha-1}) \leq {s/(\ell_n-s)}$.
Moreover, conditional on the fact that $e_\alpha \in \mathcal{B}_{\alpha-1}$, $\hat{I}_\alpha\neq I_\alpha$ if and only if $\hat{I}_\alpha\neq (0,0,0)$ and $\PR(\hat{I}_\alpha\neq{ (0,0,0)} \vert e_\alpha\in \mathcal{B}_{\alpha-1}) = \PR(\hat{I}_\alpha\neq {(0,0,0)}) \leq {s/(\ell_n-s)}$.
Thus,
\begin{equation}\label{eq:e-alpha-bad}
\PR(e_\alpha \in \mathcal{B}_{\alpha-1}\text{ and }I_\alpha\neq \hat{I}_\alpha) \leq {s^2/(\ell_n-s)^2}.
\end{equation}
On the other hand,
if $f_\alpha \in \mathcal{B}_{\alpha-1}$ and $e_\alpha \in \mathcal{B}_{\alpha-1}^c$ (we call this event $F_\alpha$),
then the event $f_{\alpha}\in S^c$ implies $g_\alpha\in S$.  Therefore
\begin{align}\label{eq:f-alpha-bad}
&\PR({F_\alpha}\text{ and }I_\alpha\neq \hat{I}_\alpha)
\leq \PR( F_\alpha, f_{\alpha}\in S) +  {\PR(F_\alpha, f_{\alpha}\in S^c, g_\alpha\in S)}\nonumber\\
&= \mathrm{(I)} + \mathrm{(II)}.
\end{align}
First we obtain an upper bound on
$\mathrm{(II)}$.
 Noting that
 $\PR(F_\alpha)\leq \PR(f_\alpha \in \mathcal{B}_{\alpha-1}) \leq {s/(\ell_n-s)}$
 and
 $\PR(g_{\alpha} \in {S}\vert F_\alpha) \leq {s/ (\ell_n-s)}$ due to the choice of $g_\alpha$, it follows that
\begin{equation}\label{bound-term-1}
\mathrm{(II)} \leq {s^2/(\ell_n-s)^2}.
\end{equation}
To derive an upper bound on $\mathrm{(I)}$,
we use the fact that, conditioned on
$|\mathcal{B}_{\alpha-1} \cap {S}|\leq \ell_n^{1/3}$,
the probability that $f_\alpha \in \mathcal{B}_{\alpha-1}\cap {S}$
is
at most $\ell_n^{1/3}/(\ell_n-s)$.
Thus,
\begin{align}
\mathrm{(I)}
& \leq \PR(f_{\alpha}\in \mathcal{B}_{\alpha-1}\cap {S}) \leq \frac{\ell_n^{1/3}}{\ell_n-s} + \PR(|\mathcal{B}_{\alpha-1}\cap {S}|>\ell_n^{1/3})
\\
&\leq \frac{\ell_n^{1/3}}{\ell_n-s} + \PR({\cE_n(S,S})>\ell_n^{1/3}),
\end{align}
where the last inequality follows using the fact that an element is added to $\mathcal{B}_{\alpha-1} \cap {S}$ if and only if either $e_\alpha,f_{\alpha} \in \mathcal{B}_{\alpha-1}^c\cap {S} $ or $e_\alpha,g_{\alpha} \in \mathcal{B}_{\alpha-1}^c \cap {S}$ but $f_{\alpha} \in \mathcal{B}_{\alpha-1}$, and in both cases $\mathcal{E}_n({S,S})$ increases by 1.
Now, in the sequential pairing scheme for creating the configuration model in \eqref{eq:defn-0}, let $\mathscr{F}_\alpha$ denote the sigma-algebra with respect to which the pairing obtained upto time $\alpha$ is measurable.
Let $X_\alpha = \E[\mathcal{E}_n(S,S)\vert \mathscr{F}_\alpha]$, so that $X_{s} = \mathcal{E}_n(S,S)$ and $X_0  = \E[\mathcal{E}_n(S,S)]$.
Thus $(X_\alpha)_{\alpha = 1}^{s}$ is the Doob-martingale for $\mathcal{E}_n(S,S)$ and $|X_\alpha - X_{\alpha-1}| \leq 1$.
Using Azuma-Hoeffding inequality \cite[Theorem 2.25]{JLR00} and the fact that
$\E[\mathcal{E}_n(S,S)] \leq
s^2
/\ell_n = O(1)$,
 it follows that for all sufficiently large $n$
\begin{align*}
\PR(\mathcal{E}_n(S,S) > \ell_n^{1/3}) \leq \PR(\mathcal{E}_n(S,S)-\E[\mathcal{E}_n(S,S)] > \frac 12\ell_n^{1/3}) \leq \e^{-\ell_n^{2/3}/4s}, \nonumber
\end{align*}
%
%
%
%
%
and therefore,
\begin{equation}\label{bound-term-2}
\mathrm{(I)} \leq \frac{\ell_n^{1/3}}{\ell_n-s} + \e^{-c\ell_n^{1/6}}
\end{equation}
for some constant $c>0$.

Combining \eqref{eq:f-alpha-bad}, \eqref{bound-term-1} and \eqref{bound-term-2}, we get that
 for all sufficiently large~$n$
\begin{equation}
\PR(F_\alpha
\text{ and }I_\alpha\neq \hat{I}_\alpha) \leq \frac{
\ell_n^{1/3}}{\ell_n-s} +
\e^{-c\ell_n^{1/6}} + \frac{s^2}{(\ell_n-s)^2}.
\end{equation}
With the help of \eqref{eq:e-alpha-bad} this implies that for all sufficiently large $n$
\begin{align}
&\PR\bigg(\sum_{\alpha =1}^{s}\hat{I}_\alpha \neq \sum_{\alpha =1}^{s}I_\alpha\bigg)
\leq \sum_{\alpha=1}^{s}\PR( \hat{I}_\alpha \neq I_{\alpha}) \nonumber\\
 &\hspace{1cm}\leq
\frac{s^3}{(\ell_n-s)^2}+\frac{s\ell_n^{1/3}}{\ell_n-s} + s\e^{-c\ell_n^{1/6}}
=: \Err_1.\label{coupling-indept}
\end{align}
Since $s = O(\sqrt{\ell_n})$, $\mathrm{Err}_1\to 0$.
\paragraph*{Multivariate Stein's method}
We will use the result for multivariate Poisson approximation \cite[Theorem 1]{Bar88}. Let
\begin{eq}
\lambda_{11} = \sum_{\alpha=1}^{s}\PR(\hat{I}_\alpha = (1,0,0)) &= \sum_{\alpha=1}^{s_1} \frac{s_1-\alpha}{\ell_n-\alpha},
\\
 \lambda_{12} = \sum_{\alpha=1}^{s}\PR(\hat{I}_\alpha = (0,1,0)) &= \sum_{\alpha=1}^{s_1} \frac{s_2}{\ell_n-\alpha}
 ,\\
 \lambda_{22} = \sum_{\alpha=1}^{s}\PR(\hat{I}_\alpha = (0,0,1)) &= \sum_{\alpha=s_1+1}^{s} \frac{s_{2}-(\alpha-s_1)}{\ell_n-\alpha}
, \\
c(\lambda)=c(\lambda_{11},\lambda_{12},\lambda_{22}) &= \frac{1}{2} + \max\{0,\log (2(\lambda_{11}+\lambda_{12}+\lambda_{22}))\}.
\end{eq}
\cite[Theorem 1]{Bar88} together with an easy calculation using the fact that $\lambda_{11}\geq s_1(s_1-1)/2\ell_n$, $\lambda_{12}\geq s_1s_2/\ell_n$ and $\lambda_{22}\geq s_2(s_2-1)/2\ell_n$ then yields
\begin{eq}\label{poi-distance}
&\dTV\bigg(\sum_{\alpha=1}^{s}\hat{I}_\alpha, \poi(\lambda_{11}) \otimes \poi(\lambda_{12}) \otimes \poi(\lambda_{22})\bigg)
 \\
    & \leq
     \frac{1}{(\ell_n-s)^2}
\min\bigg\{ s^3, c(\lambda) \Big(
    {2\ell_n}(2s_1-1)
    +{\ell_n } s_2
    + {2\ell_n}(2s_2-1)
    \Big)\bigg\}
    \\
&\leq   \frac{1}{(\ell_n-s)^2} \min\big\{s^3 ,5c(\lambda) s\ell_n \big\}
=: \Err_2.
\end{eq}
Combining \eqref{coupling-indept} and \eqref{poi-distance}
with the bound
\begin{eq}\label{Poiss-Poiss}
\big|\lambda_{11} - \frac {s_1^2}{2\ell_n}\big|
+\big|\lambda_{12} - \frac {s_1s_2}{\ell_n}\big|
+\big|\lambda_{22} - \frac {s_2^2}{2\ell_n}\big|
\leq\max\bigg\{\frac s{\ell_n},\frac{s^3}{2\ell_n(\ell_n-s)}\bigg\}
=: \Err_3,
\end{eq}
 we obtain that
\begin{eq}\label{eq:total-var-1}
\dTV\bigg(\sum_{\alpha=1}^{s}{I}_\alpha, \poi(\frac {s_1^2}{2\ell_n}) \otimes \poi(\frac {s_1s_2}{\ell_n}) \otimes \poi(\frac {s_2^2}{2\ell_n})\bigg) & \leq \Err_1+\Err_2+(1-e^{-\Err_3})
\\
&\leq \Err_1+\Err_2+\Err_3.
\end{eq}
When $S_1$ and $S_2$ are fixed subsets with $s_1,s_2=O(\sqrt{\ell_n})$,
the proof of Proposition~\ref{lem:poisson-approximation} now follows
from \eqref{eq:total-var-1}.

Let us now consider the case where $S_j$'s are random sets.
Observe that
that for $\eps$ sufficiently small,
$\mathrm{Err}_i = o(1)$ if $s\leq \ell_n^{1/2+\eps}$.  If we condition on the sets
$(S_j)_{j\in [k]}$ and assume that $s\leq \ell_n^{1/2+\eps}$, we can therefore couple
$\bld{\cE}_n$ and $\bld{\cE}$ in such a way that
$\bld{\cE}_n\neq\bld{\cE}$ with probability $o(1)$; if $s> \ell_n^{1/2+\eps}$, we couple them
arbitrarily.  Since $\PR(s>\ell_n^{1/2+\eps})\leq\ell_n^{-1/2-\eps}\E[s]=O(\ell_n^{-\eps})$ by Markov's inequality and our assumptions on the expectations
of $s_1$ and $s_2$, we see that the resulting coupling is such that
$\bld{\cE}_n=\bld{\cE}$ with probability $1-o(1)$, showing that
$\lim_{n\to\infty}\dTV(\bld{\cE}_n,\bld{\cE}) = 0$,
as required.

The proof of Proposition~\ref{lem:poisson-approximation} is now complete.
\end{proof}

\section{Proof of results on preferential attachment model}\label{sec:proof_PA}
In this section, we prove Theorem~\ref{thm:PAM-limit}.
To avoid notational overhead, we re-cycle some notation, and denote the random point process
$\lbl_{\sqrt{2m_n}}(\PA)$ for the graph $\PA$ by $\xi_n$.
For a subset $A\in \sB(\R_+)$, let $V_n(A)$ denote the set of vertices obtained by labeling the vertices uniformly from $[0,\sqrt{2m_n}]$ independently and retaining the vertices with labels in $A$.
For any $V\subset [n]$, let $S_n(l;V)=\sum_{i\in V} (\delta_i+d_i(l))$.
We will also define a random measure $\bar S_{n,\delta}$ by
$\bar S_{n,\delta}(A)=\frac{\sqrt{2m_n}}{\elld}S_n(0,V_n(A))$.

The main ingredients of the proof can be decomposed into the following
lemma and propositions;  in all of them, we  assume that
$\min\{\elld,m_n\}=\omega(\log n)$ and
$\max_i\delta_i=o(\elld)$, stating any additional assumption explicitly.

\begin{lemma}\label{prop:PAM-non-loop}
If $\log m_n =o(\elld)$, then $e(\PA) = m_n (1+o(1))$ a.s.~$\PA$
.
\end{lemma}

\begin{proposition}\label{prop:PAM-TV}
Let $( V_i)_{i\in [k]}$  be a disjoint collection of vertex subsets such that, for all $i \in [k]$, $S_n(0,V_i) = O(\elld/\sqrt{m_n})$, and for any $\varepsilon>0$,
\begin{eq}\label{eq:condn-total-var}
\lim_{n\to\infty}\PR\bigg(\sup_{l\leq m_n} \bigg| \frac{\sqrt{2m_n}S_n(l,V_i)}{\elld+2l} - \frac{\sqrt{2m_n}S_n(0,V_i)}{\elld} \bigg|>\varepsilon\bigg) = 0.
\end{eq}
Let $\bld{\mathcal{E}}_n := (\mathcal{E}(V_i, V_j) )_{1 \leq i \leq j \leq k }$ , $\bld{\mathcal{E}} := (\mathcal{E}_{ij})_{1 \leq i \leq j \leq k }$, where $\cE(V_i,V_j)$ denotes the number of edges in $\PA$ with one end-point in $V_i$ and the other end-point in $V_j$, and $\bld{\cE}$ is a collection of independent random variables with $\mathcal{E}_{ij} \sim \mathrm{Poisson}(2m_n S_n(0,V_i) S_n(0,V_j)/{\elld^2})$, $i \neq j$, while \\
$\mathcal{E}_{ii} \sim \mathrm{Poisson}( m_n S_n(0,V_i)^2/ \elld^2)$. Then
\begin{eq}
\lim_{n\to\infty}\dTV ( \bld{\mathcal{E}}_n ,  \bld{\mathcal{E}}) \to 0.
\end{eq}
Further, suppose that $V_i$'s are random subsets
chosen independently of $\PA$ such that
 $\E[ S_n(0,V_i)] = O(\elld/\sqrt{m_n})$ and such
for any $\varepsilon > 0$, \eqref{eq:condn-total-var} holds, where the probabilities and expectations under consideration are to be taken over the joint distribution of $V_i$'s and $d_i(l)$'s.
Then, $\lim_{n\to\infty}\dTV ( \bld{\mathcal{E}}_n ,  \bld{\mathcal{E}}) = 0$.
\end{proposition}

We will want to apply this proposition to random sets of the form
$V_i=V_n(A_i)$, where  $A_i\in \sB(\R_+)$, $i=1,\dots,k$ are pairwise disjoint.
To do this, we use the following proposition, and the fact that
$\E[ S_n(0,V_n(A_i))]
=\Lambda(A_i)\elld/\sqrt{m_n}=O(\elld/\sqrt{m_n})$.

\begin{remark}\label{rem:PA=CM-bard}
\normalfont Before stating the next proposition, let
us apply the statements of the previous one to the sets $V_i=\{i\}$ to relate the proposition to
the heuristic arguments given for the preferential attachment model in the introduction.  Since our actual
proof will not use this argument, let us not worry about verifying the condition~\eqref{eq:condn-total-var}.
Then the statement of the proposition say that the number of edges between $i$ and $j$
is a Poisson random variable with parameter $2m_n\delta_i\delta_j/\elld^2=\bar d_i\bar d_j/2m_n$,
and the number of loops at $i$ is a Poisson random variable with parameter $\bar d_i^2/4m_n$,
showing at least heuristically that the preferential attachment model behaves like a configuration model with
degree sequence $(\bar d_i)_{i\in [n]}$.
\end{remark}

\begin{proposition}\label{prop:PAM-degree}
Fix  $A\in \sB(\R_+)$, and   $\varepsilon > 0$. If  $m_n=o(\elld^2)$ then
\begin{eq}
\lim_{n\to\infty}\PR\bigg(\sup_{l\leq m_n} \bigg| \frac{\sqrt{2m_n}S_n(l,V_n(A))}{\elld+2l} - \frac{\sqrt{2m_n}S_n(0,V_n(A))}{\elld} \bigg|>\varepsilon\bigg) = 0.
\end{eq}
\end{proposition}
\noindent Finally, we need the following analogue of Proposition~\ref{lem:concentration-PP}.
\begin{proposition}\label{prop:PAM-concentration}
For any $A,B\in \sB(\R_+)$, $l\in \N$, and $\delta>0$,
\begin{align*}
&\PR \Big( \big|\PR\big(\xi_{n}(A\times B) = l \vert \PA\big) - \PR\big(\xi_{n}(A\times B) = l\big)\big|>\delta \Big)  \\
&\leq 2\exp \Big(-\frac{\delta^2m_n}{8\Lambda(A)^2\Lambda(B)^2}\Big). \nonumber
\end{align*}
\end{proposition}
\noindent
First, we complete the proof of Theorem~\ref{thm:PAM-limit}, given Lemma~\ref{prop:PAM-non-loop}, and Propositions~\ref{prop:PAM-degree}, ~\ref{prop:PAM-TV} and ~\ref{prop:PAM-concentration}, and defer their proofs to the end of the section.

\begin{proof}[Proof of Theorem~\ref{thm:PAM-limit}]
The equivalence of (ii), (iii) and (iv) can be read of Theorem~\ref{thm:main-CM} applied to the sequence
$(\bar d_i)_{i\in [n]}$.  All we need to observe is that $\sum_i\bar d_i=2m_n$, and that our assumptions
on $\max_i\delta_i$, $\elld$ and $m_n$ imply that $\max_i\bar d_i=o(2m_n)$ and $m_n=\omega(\log n)$.

To show equivalence of (i) and (iv), we first note that the assumptions of Theorem~\ref{thm:PAM-limit}
implies those of Lemma~\ref{prop:PAM-non-loop} and Proposition~\ref{prop:PAM-degree}.  We then
apply Proposition~\ref{prop:PAM-TV}
to the sets $V_n(A_1),\dots$, $V_n(A_k)$, where  $A_i\in \sB(\R_+)$, $i=1,\dots,k$, and
rewrite the statement of the proposition for this case in terms of the random variables
$\bar S_{n,\delta}(A_i)$, $i=1,\dots,k$.  The proof is then identical to the proof of the equivalence
of (i) and (iv) in Theorem~\ref{thm:main-CM}, once we
replace
Lemma~\ref{CM:edge_count} by Lemma~\ref{prop:PAM-non-loop},  Proposition~\ref{lem:concentration-PP} by Proposition~\ref{prop:PAM-concentration}, and Proposition~\ref{lem:poisson-approximation}
by Proposition~\ref{prop:PAM-TV}
and Proposition~\ref{prop:PAM-degree}.
The proof that the sampling limit is given by the graphex $\cW_{\sss \mathrm{CM}}$ is again
the same, once we observe that Proposition~\ref{lem:conv-sum-degree} and \eqref{X_n-characteristic}
hold for $\bar S_{n,\delta}$, $\rho_{n,\delta}$ and $(X_{n,\delta}(t))_{t\geq 0}$ where $X_{n,\delta}(t) = \mu_{n,\delta}([0,t])$.

\end{proof}

\noindent Next we prove Propositions~\ref{prop:PAM-degree},
Lemma~\ref{prop:PAM-non-loop}, and
Propositions~\ref{prop:PAM-concentration} and~\ref{prop:PAM-TV}, in that order.
We let $(\mathscr{F}_l)_{l=1}^{m_n}$ denote the canonical filtration associated with the graph process $(\mathrm{PA}_n(\bld{\delta},l))_{l=1}^{m_n}$, and let $(\mathscr{F}_l')_{l=1}^{m_n}$ be a  filtration, where $\mathscr{F}_l'$ is the minimal sigma algebra containing the information about $V_n(A)$ and $\mathrm{PA}_n(\bld{\delta},l)$.
Note that $\mathscr{F}_0'$ is the sigma algebra containing the information about $V_n(A)$ only,
 while both $\PA$ and $V_n(A)$ are measurable with respect to the filtration $\mathscr{F}_{m_n}'$.
\begin{proof}[Proof of Proposition~\ref{prop:PAM-degree}]
To avoid cumbersome notation, we write $S_n(l)$ for $S_n(l,V_n(A))$ in this proof.
Clearly, $(S_n(l))_{l\in [0,m_n]}$ is a Markov chain with $S_n(0) = \sum_{i\in V_n(A)} \delta_i$ and conditionally on $\sF_l'$,
\begin{eq}\label{eq:growth}
S_n(l+1) =
\begin{cases}
S_n(l) & \text{ with probability }\frac{(\elld+2l-S_n(l))^2}{(\elld+2l)^2},\\
S_n(l) + 1 & \text{ with probability }\frac{2S_n(l)(\elld+2l-S_n(l))}{(\elld+2l)^2},\\
S_n(l)+2 & \text{ with probability }\frac{S_n(l)^2}{(\elld + 2l)^2}.
\end{cases}
\end{eq}
Note that for any $l\geq 0$,
\begin{align*}
\E[S_n(l+1)-S_n(l)\vert\mathscr{F}_l'] = \frac{2S_n(l)}{\elld+2l}  \implies   \E \bigg[ \frac{S_n(l+1)}{\elld+2(l+1)} \Big\vert \mathscr{F}'_l\bigg] = \frac{S_n(l)}{\elld+2l}.
\end{align*}
and therefore $(\frac{\sqrt{2m_n}S_n(l)}{\elld+2l})_{l=0}^{m_n}$ is a martingale with respect to $(\sF_l')_{l=0}^{m_n}$.
Let $\QV$ denote the quadratic variation of this martingale.
To compute $\QV$, note that
\begin{align}
&\E\bigg[\bigg(\frac{\sqrt{2m_n}S_n(l+1)}{\elld+2l+2}-\frac{\sqrt{2m_n}S_n(l)}{\elld+2l}\bigg)^2\bigg\vert\sF_l'\bigg] \nonumber\\
&\leq \frac{4m_n}{(\elld+2l)^2({\elld+2l+2})^2}\bigg((\elld+2l)^2\E[(S_n(l+1)-S_n(l))^2\vert \sF_l']+4S_n^2(\ell)\bigg) \nonumber \\
&\leq\frac{4m_n}{(\elld+2l)^4}\bigg(2(\elld+2l)^2\E[(S_n(l+1)-S_n(l))\vert \sF_l']+4S_n^2(\ell)\bigg) \nonumber \\
&=\frac{16m_nS_n(l)}{(\elld+2l)^3}+\frac{16m_nS_n^2(\ell)}{(\elld+2l)^4}.\label{eq:QV-perstep-ubound}
\end{align}
\noindent
We will need the following fact, whose proof is given immediately after completing the proof of this proposition.
\begin{fact}\label{fact:total-weight-pam}
$\E[S_n(l)]  =O\Big(\frac{\elld + 2 l}{\sqrt{m_n}}\Big)$
and $\E[S_n(l)^2] = o\Big(\frac{(\elld+2 l)^2}{\sqrt{m_n}}\Big)$.
\end{fact}

\noindent Using Fact~\ref{fact:total-weight-pam}, \eqref{eq:QV-perstep-ubound} now shows that for all $t\in [0,1]$, we have
\begin{align*}
\E[\QV(tm_n)] &= \sum_{l< tm_n}\E\bigg[\bigg(\frac{\sqrt{2m_n}S_n(l+1)}{\elld+2l+2}-\frac{\sqrt{2m_n}S_n(l)}{\elld+2l}\bigg)^2\bigg\vert\sF_l'\bigg]\\
&\leq 8m_n \sum_{l=0}^{\infty}\bigg(\frac{\E[S_n(l)]}{(\elld+2l)^3}+
\frac{\E[S_n(l)^2]}{(\elld+2l)^4}
\bigg) = O\Big(\frac{\sqrt{m_n}}{\elld}\Big)=o(1).
\end{align*}
An application of  Doob's inequality \cite[Chapter 1, \S 9, Theorem~1(3)]{LS89} now completes the proof.
\end{proof}

\begin{proof}[Proof of Fact~\ref{fact:total-weight-pam}]
Recall that $(\frac{\sqrt{2m_n}S_n(l)}{\elld+2l})_{l=1}^{m_n}$ is a martingale with respect to $(\sF_l')_{l=1}^{m_n}$. Thus,
\begin{eq}
\E[S_n(l)] =\frac{\elld + 2 l}{\elld} \E[S_n(0)]
=O\bigg(\frac{\elld + 2 l}{\sqrt{m_n}}\bigg).
\end{eq}
For bound on the second moment, note that \eqref{eq:growth} implies
\begin{align*}
\E[S_n(l+1)^2]
&=
\E[(S_n(l))^2] \Big(1 + \frac{4}{(\elld+2 l)}+\frac{2}{(\elld+2 l)^2} \Big)
+\frac {2\E[S_n(l)]}{\elld+2l}
\nonumber
\\
&=
\E[(S_n(l))^2] \Big(1 + \frac{2}{(\elld+2l)} \Big)^2
+O\Bigl(\frac 1{\sqrt{m_n}}\Bigr),
\nonumber
\end{align*}
which in turn shows that
\begin{align*}
&\frac {\E[(S_n(l))^2]}{(\elld+2 l)^2}\leq \frac {\E[S_n(0)^2]}{\elld^2} +O\Bigl(\frac 1{\sqrt{{m_n}}}\Bigr)\sum_{k=0}^{\ell}\frac 1{(\elld+2k+1)^2} \nonumber \\
&=\frac {\E[S_n(0)^2]}{\elld^2} +O\Bigl(\frac 1{\elld\sqrt{{m_n}}}\Bigr)
\end{align*}
Combined with the bound
\begin{align*}
\E[S_n(0)^2] = \Var(S_n(0)) + \E^2[S_n(0)] =
O\Big( \frac{\elld}{\sqrt{m_n}} \max_{i\in [n]} \delta_i  + \frac{\elld^2}{m_n}\Big)
=O\Big(\frac{\elld^{2}}{\sqrt{m_n}}\Big),
 \nonumber
\end{align*}
(where the last step follows using our assumption that $\max_{i\in [n]} \delta_i=o( \elld)$),
the claim now follows.
\end{proof}

\begin{proof}[Proof of Lemma~\ref{prop:PAM-non-loop}]
Recall that, according to the definition preceeding
Lemma~\ref{prop:PAM-non-loop},
we have $S_n(l,\{i\}):= d_i(l) +\delta_i$.
Let $L_l$ denote the number of loops in $\mathrm{PA}_n(\bld{\delta}, l)$,
and let
$P_l=\frac{\sum_{i\in [n]}(S_n(l,\{i\}))^2}{(\elld+2l)^2}$. Conditionally on $\sF_l$,
we then have that
\begin{eq}\label{defn:L-l}
L_{l+1} =
\begin{cases} L_l + 1 & \text{with probability } P_l,\\
L_l &\text{otherwise,}
\end{cases}
 \end{eq}
Setting $L_l^\prime=L_l-\sum_{k=0}^{l-1}P_k$, we see that $(L_l')_{l=0}^{m_n}$ is a Martingale with respect to the filtration
$(\sF_l)_{l=0}^{m_n}$, and that $L_l^\prime -P_l\leq L_{l+1}^\prime\leq L_l^\prime +1$. The Azuma-Hoeffding inequality and the fact that $L_0=0$ then implies that
\begin{eq}
\label{Azuma-Loops}
\PR\bigg(L_{m_n}\geq\sum_{l=0}^{m_n-1}P_l+\lambda m_n\bigg)
=\PR\Big(L_{m_n}'\geq\lambda m_n\Big)\leq e^{-\frac{\lambda^2m_n}{2}}.
\end{eq}
Next define $Q_l: =\sum_{i\in [n]}(S_n(l,\{i\}))^2$.
By a simple calculation using
the analogue of \eqref{eq:growth} for $S_n(l,\{i\})$, we get that
\begin{align*}
\E[Q_{l+1}|\sF_l]
&=
Q_l \Big(1 + \frac{4}{(\elld+2\ell)}+\frac{2}{(\elld+2\ell)^2} \Big)
+2
\leq Q_l \Big(\frac{\elld+2\ell+2}{\elld+2\ell} \Big)^2
+2
\end{align*}
which we rewrite as
\[
\E[P_{l+1}|\sF_l]\leq P_l+\frac 2{(\elld+2\ell+2)^2}
\]
to conclude that $P_l'=P_l - \sum_{k=1}^{l}\frac 2{(\elld+2k)^2}$ is a supermartingale.
Next we note that
upon the addition of an edge at step $l+1$,
either one of the terms in $Q_l$ change from $(S_n(l,\{i\}))^2$ to $(S_n(l,\{i\})+2)^2$
or two of them change from $(S_n(l,\{i\}))^2$ to $(S_n(l,\{i\})+1)^2$.  In either case,
the total change is at most
$4(\elld+2\ell)+4$.
A straightforward calculation using the fact that $Q_l\leq (\elld+2\ell)^2$ then shows
that
\[
-\frac{4}{\elld+2l+2}\leq P_{l+1}^\prime-P_{l}'\leq \frac{4}{\elld+2l+2}
\]
The
Azuma-Hoeffding inequality combined with the facts that $P_0'=P_0$ then implies that
\begin{eq} \label{Q-sample-path-deviation}
\PR\{\exists l\leq m_n: P_l' \geq P_0+\lambda\}
\leq m_ne^{-\frac{\lambda^2\elld}{32}}\leq e^{-\frac{\lambda^2\elld}{64}}
\end{eq}
provided $\log m_n =o(\elld)$ and $n$ large enough.
Combining \eqref{Azuma-Loops} and \eqref{Q-sample-path-deviation} then shows that with probability
at least $1-e^{-\lambda^2 m_n/2}-e^{-{\lambda^2\elld}/{64}}$, we have that
\[
L_{m_n}<\sum_{l=0}^{m_n-1}P_l+\lambda m_n\leq
\sum_{l=0}^{m_n-1}P_l' +\frac {m_n}{\elld^2}+\lambda m_n
\leq m_n P_0 +\frac {m_n}{\elld^2}+2\lambda m_n
=(2\lambda +o(1))m_n,
\]
where in the last step we used that
$P_0=\frac 1{\elld^2}\sum_{i\in [n]}\delta_i^2=o(\elld)\frac 1{\elld^2}\sum_{i\in [n]}\delta_i = o(1)$.
By our assumption that $\min\{m_n,\elld\}=\omega(\log n)$, the
the error probability is summable for all fixed $\lambda>0$.  Since $\lambda$ was arbitrary, this proves that
$L_{m_n}/m_n\to 0$ with probability $1$.
\end{proof}

\begin{proof}[Proof of Proposition~\ref{prop:PAM-concentration}]
 Let $X :=\PR\big(\xi_{n}(A\times B) = l \vert \PA\big)$ and let $( Y_{l} := \E[X | \mathscr{F}_l])_{l=0}^{m_n}$ denote the Doob martingale with respect to the filtration $(\mathscr{F}_l)_{l=0}^{m_n}$. Therefore, $ Y_0 = \E[X]$, while $Y_{m_n} = X$.
Moreover, $X$ can change by changing the status of the edge $(i,j)$ only if
$(i,j)\in V_n(A)\times V_n(B)$ or $(j,i)\in V_n(A)\times V_n(B)$.
Thus, the martingale difference $|Y_l - Y_{l-1} |$ can be bounded as
\begin{eq}
|Y_l - Y_{l-1} | \leq \frac{2\Lambda(A)\Lambda(B)}{m_n}.
\end{eq}
An application of the Azuma- Hoeffding inequality now completes the proof.
\end{proof}

\begin{proof}[Proof of Proposition~\ref{prop:PAM-TV}]
We prove the proposition for $k=2$, the general case follows similarly.
For $l \in [ m_n ]$, we denote the added edge at time $l$ by $\{v_1(l), v_2(l)\}$,
and let
\begin{align}
I_{l+1} = \begin{cases}
(1,0,0) & \quad \text{if }  v_1(l), v_2(l) \in V_1, \\
(0,0,1) & \quad \text{if }  v_1(l), v_2(l) \in V_2,\\
(0,1,0) & \quad \text{if }  v_1(l) \in V_1, \text{ and } v_2(l)\in V_2 \text{ or vice versa,}\\
(0,0,0)& \quad \text{otherwise}.
\end{cases}\nonumber
\end{align}
Note that $(\mathcal{E}(V_1,V_1),\mathcal{E}(V_1,V_2),\mathcal{E}(V_2,V_2) ) = \sum_{l =1}^{m_n} I_{l}$.
To approximate this sum in total variation distance, we first couple the vectors $I_{l}$ to independent random vectors, and then we couple the sum of independent indicators to independent Poisson random variables.

For the first step, we will use an explicit coupling.
Let $( (U_1(l), U_2 (l) ) )_{l=1}^{m_n}$ be i.i.d random variables, with $(U_1(1), U_2(1) )$ being uniformly distributed on the unit square $[0,1]^2$.
 For $i=1,2$, we denote the events
\begin{align*}
 G_{i1}(l) = \Big\{ U_i(l) \leq  \frac{S_n(l,V_1)}{\ell_{n, \delta} + 2l} \Big\} , \quad
&G_{i2} (l) =\Big \{ \frac{S_n(l,V_1)}{\ell_{n, \delta} + 2l}<  U_i(l) \leq  \frac{S_n(l,V_1)+S_n(l,V_2)}{\ell_{n, \delta} + 2l}\Big \},  \\
\hat{G}_{i1}(l) = \Big\{ U_i(l) \leq  \frac{S_n(0,V_1)}{\ell_{n, \delta} } \Big\} , \quad
&\hat{G}_{i2} (l) =\Big \{ \frac{S_n(0,V_1)}{\ell_{n, \delta} }<  U_i(l) \leq  \frac{S_n(0,V_1)+S_n(0,V_2)}{\ell_{n, \delta} }\Big \}.
\end{align*}
We then generate the processes $(S_n(l,V_i))_{l=0}^{m_n}$, $i=1,2$, $(I_l)_{l=0}^{m_n}$ and an independent collection $(\hat{I}_l)_{l=0}^{m_n}$ jointly as follows:
\begin{align}
S_n(l+1,V_i) =\begin{cases}
S_n(l,V_i)+2 &\text{on } G_{1i} (l ) \cap G_{2i}(l) \\
S_n(l,V_i) +1 & \text{on } ( G_{1i} (l ) \cap G_{2i}(l)^c) \cup (G_{1i} (l )^c \cap G_{2i}(l) )  \\
S_n(l,V_i) &\textrm{otherwise},
\end{cases} \label{eq:tot_deg_update}
\end{align}
\begin{align}
I_{l+1} = \begin{cases}
(1,0,0) & \quad \text{on }  G_{11}(l) \cap G_{21}(l), \\
(0,0,1) & \quad \text{on }  G_{12}(l) \cap G_{22}(l),\\
(0,1,0) & \quad \text{on }  (G_{11}(l) \cap G_{22}(l)) \cup ( G_{12}(l) \cap G_{21}(l))\\
(0,0,0)& \quad \text{otherwise}.
\end{cases}\label{eq:indicators_real}
\end{align}
The definition of $\hat{I}_{l+1}$ is identical to \eqref{eq:indicators_real} by replacing $G_{ij}$ by $\hat{G}_{ij}$, $i,j=1,2$.
It is easy to see that the above give the processes $(S_n(l,V_i))_{l=0}^{m_n}$, $i=1,2$, $(I_l)_{l=0}^{m_n}$ defined before and  $(\hat{I}_l)_{l=0}^{m_n}$ is an independent collection.
Next, fix $\varepsilon>0$ and let $$\mathcal{A}_l = \bigg\{\sup_{l'\leq l} \bigg| \frac{\sqrt{2m_n}S_n(l',V_i)}{\elld+2l'} - \frac{\sqrt{2m_n}S_n(0,V_i)}{\elld}\bigg| \leq \varepsilon \bigg\}, \quad \mathcal{A} = \bigcup_{l=1}^{m_n} \mathcal{A}_l.$$
By
the assumption \eqref{eq:condn-total-var},
$\PR(\mathcal{A}^c) = o(1)$ and
\begin{eq}
&\PR(I_{l+1} \neq I_{l+1}' \vert \mathscr{F}_l') \\
&\leq 2 \bigg[ \bigg|   \frac{S_n(l,V_1)}{\ell_{n, \delta} + 2l} - \frac{S_n(0,V_1)}{\ell_{n, \delta}}   \bigg| + \bigg|   \frac{S_n(l,V_1\cup V_2)}{\ell_{n, \delta} + 2l} - \frac{S_n(0,V_1\cup V_2)}{\ell_{n, \delta}}   \bigg| \bigg] \times \\
& \hspace{2cm}\bigg[\frac{S_n(l,V_1 \cup V_2)}{\elld + 2 l} + \frac{S_n(0, V_1 \cup V_2)}{\elld}\bigg],
\end{eq}
so that $\PR(I_{l+1} \neq I_{l+1}' \vert \mathscr{F}_l') \mathbbm{1}_{\mathcal{A}_l} \leq C\varepsilon/ m_n$, for some constant $C>0$.
Thus,
\begin{align*}
& \PR( \exists l : I_{l+1}\neq \hat{I}_{l+1}) = \PR( \exists l : I_{l+1}\neq \hat{I}_{l+1}, \mathcal{A}) +o(1)\leq \PR( \exists l : I_{l+1}\neq \hat{I}_{l+1}, \mathcal{A}_l) +o(1) \\
&\leq \sum_{l=1}^{m_n} \E[\PR(I_{l+1} \neq I_{l+1}' \vert \mathscr{F}_l')\mathbbm{1}_{\mathcal{A}_l}]+o(1) \leq C\varepsilon +o(1).
\end{align*}
Since $\varepsilon>0$ is arbitrary, $\lim_{n\to\infty}\PR( \exists l : I_{l+1}\neq \hat{I}_{l+1}) =0.$
Now, we can use multivariate Stein's method to approximate $\sum_{l=1}^{m_n}\hat{I}_l$ in an identical manner as at the end of Section~\ref{sec:edge-count-CM},
this time with
\begin{eq}
\lambda_{11} = \sum_{\alpha=1}^{m_n}\PR(\hat{I}_\alpha = (1,0,0)) &= \frac{m_n}{\elld^2}S_n(0,V_1)^2,
\nonumber
\\
 \lambda_{12} = \sum_{\alpha=1}^{m_n}\PR(\hat{I}_\alpha = (0,1,0)) &= \frac{2m_n}{\elld^2}S_n(0,V_1)S_n(0,V_2)
 ,
 \nonumber\\
 \lambda_{22} = \sum_{\alpha=1}^{m_n}\PR(\hat{I}_\alpha = (0,0,1)) &= \frac{m_n}{\elld^2}S_n(0,V_2)^2.
 \nonumber
\end{eq}

\end{proof}

\section{Proof of results on generalized random graph}\label{sec:grg_proof}

We first establish that the number of edges in $\GRG$ is concentrated around  a deterministic value. Recall that for any graph $G$, we use $e(G)$ to denote the number of non-loop edges in $G$,
and that $L_n$ denotes the $\ell_1$ norm of the weight vector, $L_n=\sum_{i\in [n]}w_i$.

\begin{lemma} \label{lemma:grg_edges}
For $0<\varepsilon\leq 1$,
\begin{align}
\PR\Big( | e(\GRG) - \E[e(\GRG)] | > \varepsilon L_n  \Big) \leq
2 \exp \Big(- \frac{\varepsilon^2 L_n}3 \Big), \label{eq:grg_edge_conc}\\
\frac{1}{L_n} \E[ e(\GRG) ]  =
\frac 12 \int_0^{\infty} \int_0^{\infty} \frac{xy}{1 +xy} \rho_n(\dif x) \rho_n (\dif y)
 +o(1).
\end{align}
\end{lemma}

 We denote the the random point process $\lbl_{\sqrt{L_n }}(\GRG)$ for the graph $\GRG$ by $\xi_n$. Let $\xi$ denote the random adjacency measure corresponding to the graphex $\cW_{\mathrm{GRG}}^{1}$.

\begin{proposition}\label{prop:annealed-law-GRG}
 Under {\rm Assumption~\ref{assumption-weight}}, as $n\to \infty$,
\begin{equation}
\PR(\xi_n(A) = 0) \to \PR(\xi(A) = 0),
\end{equation}
for any $A$ that is a union of disjoint rectangles in $\R_+^2$.
\end{proposition}
\begin{proposition} \label{prop:concentration-GRG}
Suppose that {\rm Assumption~\ref{assumption-weight}}  holds.
For  $0<\varepsilon\leq 1$ and  $A\subset \sB(\R_+^2)$ that is a union of disjoint rectangles, there exists a constant $C = C(\varepsilon,A)>0$ such that
\begin{equation}
\PR\big(\big|\PR(\xi_n(A) = 0\vert \GRG) - \PR(\xi_n(A) = 0)\big|>\varepsilon\big) \leq \e^{-C L_n}.
\end{equation}
\end{proposition}

\noindent
We first prove Theorem~\ref{thm:GRG-limit}, given these results.

\begin{proof}[Proof of Theorem~\ref{thm:GRG-limit}]
Lemma~\ref{lemma:grg_edges} implies that for any $\delta_n \to 0$,
\begin{align}
\PR\Big( | e(\GRG) - \E[e(\GRG)] | > \delta_n L_n  \Big) \leq
2 \exp \Big(- \frac{\delta_n^2 L_n}{3} \Big).  \nonumber
\end{align}
As $L_n = \omega(\log n)$, we can choose $\delta_n \to 0$ such that the above probabilities are summable.
As a result, Lemma~\ref{lemma:grg_edges} implies $\PR_{\sss \mathrm{GRG}}$ a.s.,
\begin{align}
\frac{1}{L_n} e(\GRG) - \frac{1}{L_n} \E[ e(\GRG) ] \to 0. \nonumber
\end{align}
By  Lemma~\ref{lemma:rescaling}, it is therefore
 enough to show that
\begin{align}
\cL(\xi_n\vert \GRG) \to \cL(\xi), \quad \PR_{\sss \mathrm{GRG}} \text{ a.s.}
\end{align}
To this end, we use \cite[Theorem A.1]{LLR83}.  To apply this theorem,
we need to show that
for every  union  $A$  is a union of disjoint rectangles in $\R_+^2$, we have
\begin{gather}
\PR(\xi_n(A) = 0\vert \GRG) \to \PR(\xi(A) = 0), \quad \PR_{\sss \mathrm{GRG}} \text{ a.s.} \label{eq:GRG-1-suffice}\\
\E [\xi_n(A)\vert \GRG] \to \E[\xi(A)], \quad \PR_{\sss \mathrm{GRG}} \text{ a.s.} \label{eq:GRG-2-suffice}
\end{gather}
Propositions~\ref{prop:annealed-law-GRG},~and~\ref{prop:concentration-GRG} together directly imply \eqref{eq:GRG-1-suffice}.
We will verify \eqref{eq:GRG-2-suffice} only for $A = [0,t)^2$, leaving the general case to the reader.
Note that
\begin{eq}
\E[\xi_n(A)\vert \GRG] = (1+o(1))\frac{t^2}{L_n}  e(\GRG) \to \frac{t^2}{2} c\quad \PR_{\sss \mathrm{GRG}} \text{ a.s.},
\end{eq}
where the last step follows from Lemma~\ref{lemma:grg_edges}.
 This concludes the proof.
\end{proof}

\begin{proof}[Proof of Corollary~\ref{cor:grg-sub}]
Consider a sampling convergent subsequence of $\GRG$.
Observing that
\begin{align}
0\leq \int_0^{\infty} \int_0^{\infty} \frac{xy}{1+ xy} \rho_n (\dif x) \rho_n (\dif y) \leq \Big(\int_0^{\infty} x \rho_n (\dif x) \Big) \Big( \int_0^{\infty} y \rho_n(\dif y)\Big) =1.  \nonumber
\end{align}
We may therefore choose a further subsequence such that \eqref{eq:GRG-edges-main} holds for some $c>0$.
By Theorem~\ref{thm:GRG-limit}, this subsequence is a.s. sampling convergent to $\cW_{\sss\mathrm{GRG}}^{{c}}$. Since a sequence can't converge to  $\cW_{\sss\mathrm{GRG}}^{{c}}$
and $\cW_{\sss\mathrm{GRG}}^{{c'}}$ for $c\neq c'$, this completes the proof.
\end{proof}

\begin{proof}[Proof of Corollary~\ref{lemma:grg_cut-metric}]
The proof is similar to that of Corollary~\ref{cor:pure-graphon} in Section~\ref{CM:edge_count}, and thus we only sketch the main ideas. First, by our assumption that $a=0$,
\begin{align}
\lim_{\varepsilon \to 0} \limsup_{n \to \infty} \int_0^{\varepsilon} x \rho_n (\dif x) =0.
\end{align}
As in the proof of Corollary~\ref{cor:pure-graphon}, we set $V_{>\eps} = \{ i : w_ i > \varepsilon\sqrt{L_n}\}$ and set $\mathcal{E}_{\sss \mathrm{GRG}}(\varepsilon)$ to denote the number of edges with at least one end in $V_{\leq\eps}$. This implies
\begin{align}
\E[\mathcal{E}_{\sss \mathrm{GRG}}(\varepsilon) ] =
\sum_{i,j \in V_{\leq\eps}: i<j}  p_{ij}+
 \sum_{i \in V_{\leq\eps}} \sum_{j \notin V_{\leq\eps}} p_{ij}
 \leq  \sum_{i \in V_{{<}\eps}} \sum_{j \in [n]} \frac{w_i w_j}{L_n
  }
 = L_n \int_0^{\varepsilon} x \rho_n (\dif x) \nonumber
\end{align}
where in the second to last step we used that $p_{ij}\leq w_iw_j/L_n$.
Next, concentration for sum of independent Bernoulli variables \cite[(2.5) and (2.6), Theorem 2.8]{JLR00}
and the fact that $\E[\mathcal{E}_{\sss \mathrm{GRG}}(\varepsilon) ]\leq L_n$
immediately implies that as long as $0 < \delta_n\leq 1$,
\begin{align}
\PR\Big( |\mathcal{E}_{\sss \mathrm{GRG}}(\varepsilon) - \E[\mathcal{E}_{\sss \mathrm{GRG}}(\varepsilon) ] | > \delta_nL_n  \Big) &\leq 2
\exp\bigg(-\frac{L_n^2\delta_n^2}{2(\E[\mathcal{E}_{\sss \mathrm{GRG}}(\varepsilon) ]+L_n\delta_n/3)}\bigg) \nonumber \\
&\leq
2e^{- \frac{ L_n\delta_n^2}3 }. \nonumber
\end{align}
Choosing $\delta_n=o(1)$ in such a way that the error bound is summable (which is possible by our assumption that $L_n=\omega(\log n)$), we conclude that
 $\PR_{\sss \mathrm{GRG}} \text{ a.s.}$, $\mathcal{E}_{\sss \mathrm{GRG}}(\varepsilon)  \leq L_n \Err(\varepsilon,n) + o(L_n)$, where we recall the notation $\Err(\varepsilon,n)$ from the proof of Corollary~\ref{cor:pure-graphon}. The rest of the proof follows exactly as Corollary~\ref{cor:pure-graphon}, upon setting $U_n$ to be the set of vertices corresponding to $V_{\leq\eps}$.
\end{proof}
\noindent
%
%
%
It remains to establish Propositions~\ref{prop:annealed-law-GRG}-~\ref{prop:concentration-GRG} and Lemma \ref{lemma:grg_edges}.  We prove Lemma~\ref{lemma:grg_edges} first and  defer the proof of Proposition~\ref{prop:annealed-law-GRG} to Section \ref{proof:annealed_grg} and that of Proposition~\ref{prop:concentration-GRG} to Section~\ref{proof:conc_grg}.
\begin{proof}[Proof of Lemma \ref{lemma:grg_edges}]
Note that
\begin{eq}
\E[e(\GRG)] = \frac{1}{2}\sum_{ i \neq j } \frac{w_iw_j}{L_n+w_iw_j} = \frac{1}{2} \sum_{ i, j \in [n] }\frac{w_iw_j}{L_n+w_iw_j} - \frac{1}{2}\sum_{i\in [n]} \frac{w_i^2}{L_n+w_i^2}.\nonumber
\end{eq}
Since
\begin{eq}
\sum_{i\in [n]} \frac{w_i^2}{L_n+w_i^2} \leq \sum_{i:w_i \leq \sqrt{L_n}} \frac{w_i^2}{L_n}+ \sum_{i:w_i>\sqrt{L_n}} 1 =O(\sqrt{L_n}), \nonumber
\end{eq} it follows that
\begin{equation}
\frac{1}{L_n}\E[e(\GRG)] = 
\frac{1}{2} \int_{0}^\infty \int_0^\infty \frac{xy}{1+xy} \rho_n(\dif x ) \rho_n( \dif y)
+o(1). \nonumber
\end{equation}
\noindent
\eqref{eq:grg_edge_conc} follows by standard application of martingale concentration inequalities for sums of independent Bernoulli random variables such as
\cite[(2.5) and (2.6), Theorem 2.8]{JLR00}, together with the observation that
$\E[e(\GRG)]\leq L_n \int \int \frac{xy}{1+ xy} \rho_n (\dif x) \rho_n (\dif y) \leq L_n$.
\end{proof}

\subsection{Proof of Proposition~\ref{prop:annealed-law-GRG}}\label{proof:annealed_grg}
Fix any $k\geq 1$,  let $(B_{i})_{i\in [k]}$ be a collection of disjoint intervals, and
let $\cE_n(B_i,B_j)$ denote the number of edges between vertices with labels in $B_i$ and $B_j$, respectively.
We will  want to prove that
\begin{equation}\label{eq:general-PP-conv}
\lim_{n\to\infty} \PR(\cE_n(B_i\times B_j) = 0, \forall\, 1\leq i \leq j \leq k) = \PR(\xi(B_i\times B_j) = 0, \forall\, 1\leq i \leq j \leq k).
\end{equation}
Let $N=(w_i,\theta_i)_{i\geq 1}$
denote the Poisson point process on $\R_+^2$ with intensity
$\rho(\dif w) \otimes \dif\theta$.
Further let $N_{i\varepsilon} = N([\varepsilon,\infty)\times B_i)$. Throughout the proof, $\Err(\varepsilon,n)$ is a generic notation for some function $f(\varepsilon,n)$ s.t. $\lim_{\varepsilon\to 0}\limsup_{n\to\infty}f(\varepsilon,n)=0$.
Similarly,
$$
\lim_{\eps\to 0}\limsup_{K\to\infty}\limsup_{n\to\infty}\Err(\varepsilon,K,n)=0
$$
 and
$$
\lim_{K\to\infty}
\limsup_{\varepsilon\to 0}\limsup_{M\to\infty}\limsup_{n\to \infty}\,
\Err(K,\varepsilon,M,n) = 0.
$$

\noindent
Fix  $\varepsilon>0$ and let $V_{>\eps} = \{i\in [n]: w_i>\varepsilon\sqrt{L_n}\}$ and $V^c_{\leq\eps}= V_{>\eps}$.
Recalling that we assigned a random label in $[0,\sqrt{L_n}]$ to each vertex in $[n]$,
 let $V_i$ be the set of vertices with labels in $B_i$.
We set $V_i^{>} = V_i\cap V_{>\eps}$, $V_i^{\leq}= V_i\cap V_{\leq\eps}$
and  $T_i = \sum_{u\in V_i^{\leq}}\bar{w}_u$, where $\bar{w_u}= w_u/\sqrt{L_n}$.

Also let $I_{ij}:= \ind{i \text{ and }j\text{ create an edge}}$.
Thus, $(I_{ij})_{1\leq i<j\leq n}$ is an independent collection of Bernoulli random variables with $\PR(I_{ij} = 1)=p_{ij} = \bar{w_i}\bar{w_j}/(1+\bar{w_i}\bar{w_j})$.
Defining
\begin{eq} \label{definition-FIj}
F_{ii}^n(\varepsilon) &= \prod_{\substack{u<v\\ u,v\in V_i^{>} }}(1-I_{uv})\prod_{\substack{ u\in V_i^{>}\\ v\in V_i^{\leq}}}(1-I_{uv})\prod_{\substack{u<v\\ u,v\in V_i^{\leq}} }(1-I_{uv}) = F_{ii}^n(\varepsilon,1)F_{ii}^n(\varepsilon,2)F_{ii}^n(\varepsilon,3),
\\
F_{ij}^n(\varepsilon)& =
\prod_{\substack{u\in V_i^{>}\\v\in V_j^{>} }}(1-I_{uv})
\prod_{\substack{u\in V_i^{\leq} \\ v\in V_j^{>}}}
               (1-I_{uv})
               \prod_{\substack{
               u\in V_i^{>} \\v\in V_j^{\leq}}}
               (1-I_{uv})
\prod_{\substack{u\in V_i^{\leq}\\v\in V_j^{\leq}} }(1-I_{uv})
\\
&= F_{ij}^n(\varepsilon,1)F_{ij}^n(\varepsilon,2)F_{ji}^n(\varepsilon,2)F_{ij}^n(\varepsilon,3),
\end{eq}
we note that
\begin{eq} \label{eq:general-PP-exp}
\PR(\cE_n(B_i\times B_j) = 0, \forall \, 1\leq i \leq j \leq k) =
\E\bigg[\prod_{1\leq i \leq j \leq k}F_{ij}^n (\varepsilon)\bigg]
\end{eq}

We first state a lemma which identifies a ``good" event.
\begin{lemma}\label{lem:bad-event-bounds-general} Define the events
\begin{align}
\cA_{1i} := \big\{|T_i-\E[T_i]|\leq \varepsilon^{1/4}, \sum_{u\in V_i^{\leq}}w_{u}^2 \leq \varepsilon^{1/2} L_n
\big\},\nonumber
\\
\cA_{2i} := \Big\{
\sum_{u\in V_i}\bar {w}_u\leq K\Big\},\ \cA_{3i}:= \{|V_i^{>}|\leq M\},\
 \nonumber
\end{align}
Then, for $\cA = \bigcap_{i\in [k]}(\cA_{1i}\cap \cA_{2i} \cap \cA_{3i})$, $\PR(\cA^c) = \Err(K,\varepsilon,M,n)$.
\end{lemma}

Let $\PR_{\sss [k]}$ (respectively $\E_{\sss [k]}$) denote the conditional probability measure (respectively expectation) conditional on the choices of the random sets $(V_{i}^{r})_{i\in [k],r=>,\leq}$. The next lemma characterizes the asymptotic behavior of $F_{ii}^{n}, F_{ij}^{n}$.
\begin{lemma}\label{lem:bounds-gen-case}
On the set $\cA$, for all $1\leq i \leq j \leq k$,
\begin{eq}\label{eq:low_edges}
\E_{\sss [k]}[F_{ij}^n(\varepsilon,3)] =
\begin{cases} \e^{-(\E[T_i])^2/2}+ \Err(\varepsilon,n), \text{ for }  i=j\\
\e^{-\E[T_i]\E[T_j]}+ \Err(\varepsilon,n), \text{ for }i\neq j,
\end{cases}
\end{eq}
\begin{eq}\label{eq:between-edges}
\E_{\sss [k]}[F_{ij}^n(\varepsilon,2)] = \e^{-\E[T_i]\sum_{v\in V_j^{>}}\bar{w}_v}\big(1+ \Err(K,\varepsilon,n)\big).
\end{eq}
Moreover, for all $i\in [k]$,
\begin{align}
\dTV\Big(\cL(|V_i^{>}|), \poi(\Lambda(B_i)\rho ([\varepsilon,\infty)) ) \Big) = \Err(\eps,M,n).\label{eq:poisson_approx}
\end{align}
\end{lemma}
\noindent

Next we prove Proposition~\ref{prop:annealed-law-GRG}, deferring the proof of the lemmas to the later part of this section.

\begin{proof}[Proof of Proposition~\ref{prop:annealed-law-GRG}]

To prove the proposition, we use the explicit expression for
$\PR(\xi(B_i\times B_j) = 0, \forall\, 1\leq i \leq j \leq k)$ given in Lemma~\ref{lem:expr:lim-no-point}, the expression
\eqref{eq:general-PP-exp} for $\PR(\cE_n(B_i\times B_j) = 0, \forall \, 1\leq i \leq j \leq k)$,
and Lemmas~\ref{lem:bad-event-bounds-general}~and~\ref{lem:bounds-gen-case}.
To avoid cumbersome notation, we prove this result for $k=1$, and $B_1 = [0,t]$.
The generalization to $k \geq 1$ and arbitrary $B_i$'s is identical except the notational overhead, and thus we will sketch the general proof after proving the $k=1$ case.
Let $ V_{t}$ denote the set of vertices with labels in $[0,t]$, and $V_t^{>} = V_t \cap V_{>\eps}$ and $V_t^{\leq} = V_t\cap V_{\leq\eps}$.
Also recall that $T_1 = \sum_{u\in V_t^{\leq}} \bar{w}_u$.
  Define the quantities
\begin{equation} \label{defn:F-123}
F_1 := \prod_{i<j,i,j\in V_t^{>}} (1-I_{ij}),\quad  F_2 := \prod_{i\in V_t^{>}, j\in V_t^{\leq}} (1-I_{ij}), \quad F_3 := \prod_{i<j,i,j\in V_t^{\leq}} (1-I_{ij}).
\end{equation}
Thus, $\PR(\xi_n([0,t]^2)=0)= \E[F_1F_2F_3]$. Notice that $F_1,F_2,F_3 \leq 1$ almost surely, which we will  use throughout the proof.

Next, by Assumption~\ref{assumption-weight},
\begin{eq}\label{eq:lim-expt-T}
\E[T_1] = \frac{t}{\sqrt{ L_n}}\sum_{i\in V_{\leq\eps}}\frac{w_i}{\sqrt{L_n}} = t\int_0^\varepsilon x\rho_n(\dif x) = a t +\Err(\varepsilon,n).
\end{eq}
Using  Lemma~\ref{lem:bounds-gen-case}, \eqref{eq:low_edges} and \eqref{eq:between-edges} together with
\eqref{eq:lim-expt-T} and
Lemma~\ref{lem:bad-event-bounds-general},
\begin{eq}\label{simplification-expt}
\E[F_1F_2F_3] &= \E[F_1F_2F_3 \1_{\sss \cA}] + \Err(K,\varepsilon,M,n) \\
&= \e^{-\frac{a^2t^2}{2}} \E\bigg[\e^{-at\sum_{i\in V_t^{>}}\bar{w}_i}\prod_{i<j, i,j\in V_t^{>}}\frac{1}{1+\bar{w}_i\bar{w}_j}\1_{\sss \cA}\bigg]+ \Err(K,\varepsilon,M,n)\\
&= \e^{-\frac{a^2t^2}{2}} \E[f(n, |V_t^{>}|)] + \Err(K,\varepsilon,M,n),
\end{eq}
where we set
\begin{align}
f(n,k) :=\int_{[\varepsilon,\infty)^{k}}\e^{-at\sum_{i=1}^kw_i} \prod_{1\leq i<j\leq k}\frac{1}{1+w_iw_j} \prod_{i=1}^k \frac{\rho_n(\dif w_i)}{\rho_n([\varepsilon, \infty))}. \nonumber
\end{align}
Indeed, if we take the expectation in the second line in \eqref{simplification-expt} and condition on $|V_t^{>}|=k$, the elements of $V_t^{>}$ are a sequence of $k$ numbers chosen without replacement from
$\{i\in [n]\colon \bar w_i>\eps\}$.  On the event $\cA$, where $k\leq K$, we can replace the without replacement sampling by sampling with replacement at the cost of an error $\Err(K,n)$, at which point we get a $k$ independent samples from $\rho_n(dw)$ conditioned on $w>\eps$.  This proves the last identity in \eqref{simplification-expt}.

Now, using the vague convergence of $\rho_n$ from Assumption~\ref{assumption-weight}, we get that for any $k\geq 1$ and any $\varepsilon>0$ such that $\rho$ has no atom at $\eps$
\begin{eq}
\lim_{n\to\infty} f(n,k) = \int_{[\varepsilon,\infty)^{k}}\e^{-at\sum_{i=1}^kw_i} \prod_{1\leq i<j\leq k}\frac{1}{1+w_iw_j} \prod_{i=1}^k \frac{\rho(\dif w_i)}{\rho([\varepsilon, \infty))}=:f(k).\nonumber
\end{eq}
Therefore, Lemma~\ref{lem:bounds-gen-case}, \eqref{eq:poisson_approx} implies that, for continuity point $\varepsilon>0$ of $\rho$, and any $R\geq 1$
\begin{eq}\label{lim-part-sum}
\lim_{n\to\infty}&\sum_{k=1}^{R}\big| f(n, |V_t^{>}|) \PR(|V_t^{>}| = k) -
f(k) \PR(\poi(t\rho([\varepsilon,\infty))) = k)\big| = 0.
\end{eq}
Now, notice that $\max\{f(n,k),f(k)\} \leq 1$. Thus,
\begin{eq}\label{lim-tail}
\sum_{k>R}\E[f(n,k)| |V_t^{>}| = k] \PR(|V_t^{>}| = k) \leq \sum_{k>R} {|V_{>\eps}|\choose k}\Big(\frac{t}{\sqrt{L_n}}\Big)^k\leq \sum_{k>R}\frac{1}{k!}\big(t \rho([\varepsilon,\infty))\big)^k,
\end{eq}which goes to zero as $R\to\infty$.
Thus \eqref{simplification-expt}, \eqref{lim-part-sum} and \eqref{lim-tail} together with Lemma \ref{lem:expr:lim-no-point} imply that
\begin{eq}
\E[F_1F_2F_3\mathbbm{1}_{\mathcal{A}}] = \PR(\xi([0,t]^2) = 0)+ \Err(K,\varepsilon,M,n), \nonumber
\end{eq}and the proof follows using Lemma~\ref{lem:bad-event-bounds-general}.

Let us now sketch the proof for the general $k$ case briefly.
For simplicity, let us consider $B_i = [t_{i-1},t_{i-1}+t_i]$, where $t_0 = 0$ and $t_i>0$ for $i\in [k]$.
Recall the notations in \eqref{definition-FIj}.
From the identity \eqref{eq:general-PP-exp}, we can use  identical computations as in \eqref{simplification-expt} that yields
\begin{align*}
&\PR(\cE_n(B_i\times B_j) = 0, \  \forall 1\leq i\leq j\leq k) \\
&= \E \bigg[\prod_{1\leq i\leq j\leq k}F_{ij}^n(\varepsilon) \mathbbm{1}_{\cA} \bigg] + \Err(K,\varepsilon,M,n)\\
&= \e^{-\frac{a^2}{2}\sum_{i=1}^kt_i^2 - a^2 \sum_{i<j}t_it_j} \E\big[\tilde{f}(n, (|V_i^{>}|)_{i=1}^k)\big] + \Err (K,\varepsilon,M,n),
\end{align*}
where we set
\begin{align}
\tilde{f}\big(n,(r_i)_{i=1}^k\big) &:=\int_{[\varepsilon,\infty)^{\sum_{i=1}^k r_i}}\e^{-a \sum_{i\neq j}t_i\sum_{l=1}^{r_j}w_{jl}} \prod_{1\leq i<j \leq k} \prod_{\substack{1\leq l_1\leq r_i \\ 1\leq l_2\leq r_j}}\frac{1}{1+w_{il_1}w_{jl_2}} \times \\
&\hspace{2cm}\prod_{i=1}^k \prod_{1\leq l_1 < l_2 \leq r_i}\frac{1}{1+w_{il_1}w_{il_2}}
 \prod_{i=1}^k\prod_{l=1}^{r_i} \frac{\rho_n(\dif w_{il})}{\rho_n([\varepsilon, \infty))}. \nonumber
\end{align}
The rest of the proof is identical to the case $k=1$.
\end{proof}

\noindent
Finally, we prove Lemma~\ref{lem:bad-event-bounds-general} and Lemma~\ref{lem:bounds-gen-case}. To avoid notational overhead, we again prove these for the special case $k=1$, and $B_1=[0,t]$ and the generalization to $k >1$ and general $B_i$'s follow using identical arguments. Recall the notations $V_t$, $V_t^{>}$, $V_t^{\leq}$ defined above \eqref{defn:F-123}, which will be used throughout the proof.

\begin{proof}[Proof of Lemma~\ref{lem:bad-event-bounds-general}]
First, note that
\begin{eq} \label{sum-square-GRG-UB}
\E\bigg[\sum_{i\in V_t^{\leq}}w_i^2\bigg]
\leq {\eps}{\sqrt{L_n}}\E\bigg[\sum_{i\in V_t}w_i\bigg]
=
{t\eps }\sum_{i\in [n]}w_i  = t\eps L_n,
\end{eq}
where the first inequality follows using
$\max_{i\in V_t^{\leq}}w_i\leq{\eps}\sqrt{L_n}$.
Further,
$$\var{T_1}  = \sum_{i\in V_{\leq\eps}}\frac{w_i^2}{L_n}\frac{t}{\sqrt{ L_n}}\Big(1-\frac{t}{\sqrt{ L_n}}\Big)\leq \frac{t\varepsilon}{L_n}\sum_{i\in V_{>\eps}^c}w_i\leq t \varepsilon.$$
Thus, by Chebyshev's inequality, together with \eqref{sum-square-GRG-UB} and Markov's inequality yields $\PR(\cA_1^c) = \Err(\varepsilon,n)$.
Next, again by Markov's inequality,
\begin{eq}
\PR(\cA_2^c)
\leq \frac 1K\E[\sum_{i\in V_t}\bar w_i]
=\frac 1K \frac t{L_n}\sum_{i\in [n]}w_i=\frac tK=\Err(1/K,n).
\end{eq}
Finally,
$|V_t^{>}| \sim
\mathrm{Bin}(|V_{>\eps}|, t/\sqrt{ L_n})$.
Thus, another application of Markov's inequality yields
\begin{equation}
\PR(\cA_3^c) \leq \frac{t|V_{>\eps}|}{\sqrt{ L_n}M} = \frac{t\rho_n([\varepsilon,\infty))}{ M} = \Err(\varepsilon,M,n).
\end{equation}
\end{proof}

\begin{proof}[Proof of Lemma \ref{lem:bounds-gen-case}]

First, observe that, on the event $\cA$ defined in  Lemma~\ref{lem:bad-event-bounds-general},
\begin{align}
\E[F_2 \vert V_t] &
= \prod_{i\in V_t^{>}, j\in V_t^{\leq}}\frac{1}{1+\bar{w}_i\bar{w}_j}
=\e^{-T_1\sum_{i\in V_t^{>}}\bar{w}_i+O(\sum_{i\in V_t^{>}}\bar{w}_i^2\sum_{j\in V_t^{\leq}}\bar{w}_j^2)}
\nonumber \\
&=\e^{-\E[T_1]\sum_{i\in V_t^{>}}\bar{w}_i} \big(1+\Err(K,\varepsilon,n)\big)
\label{eq:T-2-cond-expt}
\end{align}
where we used that $\sum_{i\in V_t^{>}}\bar{w}_i^2\leq
\Big(\sum_{i\in V_t}\bar{w}_i\Big)^2$ to obtain the final error bound, proving \eqref{eq:between-edges}.

Next, let $\lambda = \sum_{i,j\in V_t^{\leq}}p_{ij}$.
Using standard bounds for coupling sums of independent Bernoulli random variables to Poisson random variables \cite[Theorem 2.10]{RGCN1},
\begin{eq}
\dTV\bigg(\cL\bigg(\sum_{i<j, i,j\in V_t^{\leq}}I_{ij}\Big|V_t^{>},V_t^{\leq}\bigg), \poi(\lambda)\bigg) &\leq \sum_{i<j,i,j\in V_t^{\leq}}p_{ij}^2 \leq \frac{\big(\sum_{i\in V_t^{\leq}}w_i^2\big)^2}{L_n^2} \leq\eps
\end{eq}where the last inequality holds on $\cA_1$.
Moreover, on $\cA_1$
\begin{gather}\label{eq:pij-ub-lb}
\sum_{i<j,i,j\in V_t^{\leq}}p_{ij} \leq \frac{\big(\sum_{i\in V_t^{\leq}}w_i\big)^2}{2L_n^2} = \frac{T_1^2}{2} = \frac{(\E[T_1])^2}{2}+\frac{\sqrt[4]\eps}2\Big(\frac{\sqrt[4]\eps}2+\E[T_1]\Big)=O(\sqrt[4]\eps),\\
\sum_{i<j,i,j\in V_t^{\leq}}p_{ij} \geq \frac{\big(\sum_{i\in V_t^{\leq}}w_i\big)^2}{2L_n^2} - \frac{1}{L_n^2}\sum_{i\in V_t^{\leq}}w_i^2= \frac{(\E[T_1])^2}2+O(\sqrt[4]\eps),
\end{gather}and therefore
\begin{eq}
\sum_{i<j,i,j\in V_t^{\leq}}p_{ij} = \frac{(\E[T_1])^2}{2}+O(\sqrt[4]\eps).
\end{eq} Thus, on $\cA$,
\begin{eq}
\dTV\bigg(\cL\bigg(\sum_{i<j, i,j\in V_t^{\leq}}I_{ij}\Big|V_t^{>},V_t^{\leq}\bigg), \poi((\E[T_1])^2/2)\bigg) = \Err(\varepsilon,n),
\end{eq}
and \eqref{eq:low_edges} follows immediately.

Finally, we prove \eqref{eq:poisson_approx}.
Note that $|V_t^{>}|\sim \mathrm{Bin}(|V_{>\eps}|, t/\sqrt{L_n})$.
The proof follows from standard inequalities for distance between Binomial and Poisson random variables \cite[Theorem 2.10]{RGCN1}, which implies that the left hand side of \eqref{eq:poisson_approx}
is bounded by $|V_{>\eps}| t^2/{L_n}\leq Mt^2/L_n=\Err(M,n)$.
Since $|V_{>\eps}| t/\sqrt{L_n}=t\rho_n([\eps,\infty))=t\rho([\eps,\infty))+\Err(n)$,
provided $\rho$ does not have an atom at $\eps$; since the limit $\eps\to 0$ can be taken through
the continuity points,
the proof follows.
\end{proof}

\subsection{Concentration}\label{proof:conc_grg}
\begin{proof}[Proof of Proposition~\ref{prop:concentration-GRG}]
We only give a proof for $A = [0,t]\times [0,s]$ leaving  the general case to the reader.
Let $R= {n\choose 2}$ and let $(p_{i_rj_r})_{r\in [R]}$ denote a non-increasing ordering of the $p_{ij}$'s.
Let $I_r$ denote the indicator that an edge has been created between $i_r$ and $j_r$; thus $I_{r}\sim \mathrm{Ber}(p_{i_rj_r})$, independently over $r\in [R]$.
To simplify notation, let $X= \PR(\xi_n(A)=0\vert G_n)$.
Further, for $r=,\dots, R$, let $\mathscr{F}_r = \sigma(I_i:i\in [r])$
(where we used the notation $[0]=\emptyset\}$) and define $X_r = \E[X\vert\mathscr{F}_r]$.
Thus, $(X_r)_{r=0}^R$ is a martingale with respect to the filtration
$(\mathscr{F}_r)$ satisfying $X_0 = \E[X]$ and $X_R = X$.
We will apply a concentration inequality from \cite[Theorems 18, 22]{CL06}.
Thus, if we can show that
\begin{equation}\label{eq:concentration-to-show}
 \var{X_r\vert\mathscr{F}_{r-1}}\leq \sigma_r^2, \quad |X_r-X_{r-1}| \leq M,
\end{equation}
then
\begin{equation}\label{eq:concentration-chung-vu}
\PR(|X-\E[X]|>\varepsilon)\leq 2\exp\bigg(-\frac{\varepsilon^2}{2(\sum_{r=1}^R\sigma_r^2+M\varepsilon/3)}\bigg).
\end{equation}
Thus, we need to obtain the correct $M$ and $\sigma_r^2$ such that \eqref{eq:concentration-to-show} holds.
Note that
\begin{eq}\label{concen-simpl-1}
\E[X\vert \mathscr{F}_{r-1}] &= p_{i_rj_r}\Big(\E[X\vert \mathscr{F}_{r-1},I_{r}=1]-\E[X\vert \mathscr{F}_{r-1},I_{r}=0]\Big)+\E[X\vert (\mathscr{F}_{r-1},I_{r}=0]\\
\E[X\vert \mathscr{F}_{r}] &= I_{i_rj_r}\Big(\E[X\vert \mathscr{F}_{r-1},I_{r}=1]-\E[X\vert \mathscr{F}_{r-1},I_{r}=0]\Big)+\E[X\vert \mathscr{F}_{r-1},I_{r}=0].
\end{eq}
Moreover, $X$ can change by changing the status of the edge $(i_r,j_r)$ only if both $i_r,j_r\in V_s\cup V_t$ and thus
\begin{eq}\label{concen-simpl-2}
\big|\E[X\vert \mathscr{F}_{r-1},I_{r}=1]-\E[X\vert \mathscr{F}_{r-1},I_{r}=0]\big|\leq \frac{(s+t)^2}{L_n}.
\end{eq}
Combining \eqref{concen-simpl-1} and \eqref{concen-simpl-2},
\begin{equation}\label{UB-mart-diff}
|X_r-X_{r-1}|\leq \big|I_r-p_{i_rj_r}\big|\frac{(s+t)^2}{L_n}.
\end{equation}
Therefore,
\begin{eq}
\var{X_r\vert \mathscr{F}_{r-1}} = \E[(X_r-X_{r-1})^2\vert\mathscr{F}_{r-1}] \leq \frac{(s+t)^4}{L_n^2}\E[(I_r-p_{i_rj_r})^2] \leq \frac{(s+t)^4}{L_n^2}p_{i_rj_r},
\end{eq}where the second step follows from $\E[X_r\vert\mathscr{F}_{r-1}] = X_{r-1}$ and the third step follows from \eqref{UB-mart-diff}.
Thus, we can apply \eqref{eq:concentration-chung-vu} with $\sigma_r^2 = (s+t)^4p_{i_rj_r}/L_n^2$, and $M=(s+t)^2/L_n$.
Now, the proof of Proposition~\ref{prop:concentration-GRG} follows by using the fact that $\sum_{i\neq j}p_{ij} \leq \sum_{i\neq j}w_iw_j/L_n \leq L_n$.
\end{proof}

\section{Proofs of results on Bipartite Configuration Model}\label{sec:bcm_proof}
The proof of Theorem~\ref{thm:main-BCM} is very similar to that of
Theorem~\ref{thm:main-CM} for the configuration model and again relies on three key propositions, whose proofs are also similar to those of the corresponding key propositions from the proof of Theorem~\ref{thm:main-CM}.  We will  outline this
proof strategy by  stating the key propositions, but we will leave both the reduction of Theorem~\ref{thm:main-BCM} to these propositions, and the proofs of the propositions themselves to the reader.

\begin{proposition}\label{lem:poisson-approximation-BCM}
Let $\cE_n(S,S')$ denote the number of edges created between the set of half-edges $S$ and $S'$ in the construction of $\BCM$.
Consider $k$ disjoint subsets of half-edges $(S_i)_{i\in [k]}$ such that $|S_i| = O(\sqrt{\ell_n})$ for all $i\in [k]$.
Let $S_{ij} = S_i\cap V_j$, $i\in [k]$, $j=1,2$.
Let $\bld{\cE}_n = (\cE_n(S_i,S_j))_{1\leq i \leq j \leq k}$, $\bld{\cE} := (\cE_{ij})_{1\leq i\leq j\leq k}$, where $\cE_{ij}\sim \mathrm{Poisson} ((|S_{i1}||S_{j2}| + |S_{i2}||S_{j1}|)/\ell_n)$ for $i\neq j$, $\cE_{ii} \sim \mathrm{Poisson}(|S_{i1}||S_{i2}|/\ell_n)$, and $\bld{\cE}$ is an independent collection.
Then, as $n\to\infty$,
\begin{equation}\label{BCMeq:total-var-conv-edges}
\dTV(\bld{\cE}_n,\bld{\cE}) \to 0.
\end{equation}
Moreover, if $S_j$'s are random disjoint subsets
chosen independently of \\
$\BCM$ and
satisfying $\E[|S_j|] = O(\sqrt{\ell_n})$,
then  $\lim_{n\to\infty}\dTV(\bld{\cE}_n,\bld{\cE}) = 0$,
where both $\bld{\cE}_n$ and $\bld{\cE}$ refer to the joint distribution, including in particular the
randomness stemming from the random sets $S_j$'s.
\end{proposition}

\begin{proposition}\label{BCMlem:conv-sum-degree}
Let $V_n(A)$ denote the set of vertices obtained by labeling the vertices uniformly from $[0,\sqrt{\ell_n}]$ and then retaining the vertices with labels in $A$.
For a vertex set $V$, define $\bar{S}_n(V)=\frac{1}{\sqrt{\ell_n}} \sum_{i\in V} d_i$.
For any disjoint collection of sets $(A_i)_{i\in [k]}$ from $\sB(\R_+)$,
let $V_{ij}$ denote the set of vertices in $V_j$ with labels in $A_i$,
 and $\alpha = (\alpha_{ij})_{i\in [k],j=1,2}\in \R^{2k}$.
 Define $\Phi((V_{ij})_{i\in [k],j=1,2}):= \E[\e^{\ri\sum_{j=1}^2\sum_{i\in [k]}\alpha_{ij} \bar{S}_n(A_{ij})}] $.
Then,
\begin{equation}\label{BCMeq:char-funct-degree}
\Phi((V_{ij})_{i\in [k],j=1,2}) =\exp\Big((1+o(1))\sum_{j=1,2}\sum_{i\in [k]}\Lambda(A_i)\int (\e^{\ri \alpha_{ij} x}-1)\rho_{nj}(\dif x) \Big).
\end{equation}
\end{proposition}
\begin{proposition}\label{BCMlem:concentration-PP}
For any $A,B\in \sB(\R_+)$, $l\in \N^*$, and $\delta>0$,
\begin{align*}
&\PR \Big( \big|\PR\big(\xi_{n}(A\times B) = l \vert \BCM\big) - \PR\big(\xi_{n}(A\times B) = l\big)\big|>\delta \Big) \nonumber \\
& \leq 2\exp \Big(-\frac{\delta^2\ell_n}{12\Lambda(A)^2\Lambda(B)^2}\Big).
\end{align*}
\end{proposition}

\noindent
\section*{Acknowledgments}
We thank Samantha Petti for several suggestions on improving an earlier version of the draft, and in particular suggesting  a simplification in the coupling in Section~\ref{sec:edge-count-CM}.



\begin{appendix}

\section{Sampling convergence for multigraphs} \label{sec:multigraph-convergence-suppl}
In this section, we prove Proposition~\ref{prop:samp-embed-lbl-equiv}.
The corresponding result for simple graphs was established in \cite{BCCV17}. The extension to multigraphs is relatively straightforward, and thus we just sketch the proof.
%
%
%
\begin{proof}[Proof of Proposition~\ref{prop:samp-embed-lbl-equiv}]
Taking into account Remark~\ref{rem:adjacency-proof} which implies that any exchangeable adjacency measure
can be represented by a possibly random multigraphex, the proof of \cite[Lemma 3.2]{BCCV17} can be immediately adapted to the multigraph setting.
The only crucial point to note is that \cite{BCCV17} use \cite[Lemma 4.11]{VR16}, which, in turn, depends on  the  discreteness of the space of finite graphs.
In the case of multigraphs, that is again true because the sampled graph almost surely take values in the space of multigraphs with finite number of edges, on which the discrete topology can be similarly defined.
\end{proof}

\section{Properties of Completely Random measures and Lev\'y processes}
\label{sec:crm}

In the proof of Theorem~\ref{thm:main-CM}, we require the notion of completely random measure which we define here.
\begin{defn}[Completely random measure] \normalfont
A random measure $\mu$ on $\R_+$ is called a completely random measure if for all finite families of bounded disjoint sets $(A_i)_{i\leq k}$ in $\sB(\R_+)$, $(\mu(A_i))_{i\leq k}$ is an independent collection of random variables.
\end{defn}
Any completely random measure $\mu$ is
a random element of $\cM(R_+)$
 and admits a nice representation \cite{K67}, \cite[Theorem 10.1III]{DV08}.
In the special case where $\mu$ is stationary, i.e., the distribution of $\mu([t,t+s])$ depends only on $s$ for any $t,s\in \R_+$, the representation takes the form
\eqref{eq:CRM-form} where
the measure $\rho$ satisfies the condition
\begin{equation}
\int_0^\infty (x\wedge 1)\rho(\dif x)<\infty,
\end{equation}see \cite[Example 10.1 (a)]{DV08}).

\begin{lemma}\label{lem:CF-CRM}
Let $\mu$ be a completely random measure of the form \eqref{eq:CRM-form}, and let
$A\in \sB(\R_+)$with $\mu(A)<\infty$. Then the characteristic function of $\mu(A)$ is given by
\begin{equation}\label{eq:char-funct-mix-poisson}
\E\big[\e^{\ri \theta\mu(A)}\big] = \exp\bigg(\ri \theta a\lambda(A)+\lambda(A)\int (\e^{\ri\theta x}-1)\rho(\dif x)\bigg).
\end{equation}
\end{lemma}

\begin{proof}
This is a straightforward calculation, very similar to the one in Exercise 10.1.2 in \cite{DV08}.
\end{proof}

We will also need the notion of a Lev\'y process. It is defined as a real valued c\`{a}dl\`{a}g process
$X=(X(t))_{t\geq 0}$ such that $X(0)=0$, the increments $X(t_1)-X(0)$, $X({t_2})-X({t_1}), \dots, X({t_n})-X({t_{n-1}})$ are independent whenever $0<t_1<\dots<t_n$, and
such that  $X({t+s})-X(t)$ is equal in distribution to $X(s)$ for all $s,t>0$. It is well know that given any bounded function $h:\R\to\R$ such that $h(x)=x$ in a neighborhood of $0$,
the characteristic function, $\chi_t(\theta)=\E(e^{i\theta X(t)})$,
can be written as $e^{t\psi(\theta)}$ with
\[
\psi(\theta)=ia\theta-\frac 12
{\sigma^2}\theta^2+\int_{\R\setminus\{0\}}d\rho(x)(e^{i\theta x}-1-i\theta h(x))
\]
where $a\in\R$, $\sigma\geq 0$, and $\rho$ is a $\sigma$-finite measure on $\R$ such that
$\int (x^2\wedge 1)d\rho(x)<\infty$. Following \cite{JS03}, we call the triple $(a,\sigma,\rho)$ the characteristics associated with $h$, or simply the
characteristics of $X$ when $h$ is clear from the context.
 While $h$ is usually chosen as $h(x)=x1_{|x|\leq 1}$, here we follow the approach of
 \cite{JS03} insisting that $h$ is continuous (since this is more convenient when considering limits); specifically, we will choose $h(x)=(|x|\wedge 1)\text{sign} (x)$.
We will need the following lemma, which is a special case of   Corollary 3.6 in Chapter VII in \cite{JS03}.

\begin{lemma}[\cite{JS03}]\label{lem-js03}
Let $X_n=(X_n(t))_{t\geq 0}$ be a sequence of Lev\'y processes with characteristics
$(b_n,\sigma_n,\rho_n)$.  Then $X_n$ converges to a Lev\'y process $X$ with characteristics
$(b,\sigma,\rho)$ in law if and only if $b_n \to b$, $\sigma_n\to\sigma$ and
$\int f \dif \rho_n \to \int f \dif \rho$
for all bounded continuous functions $f$ vanishing in a neighborhood of zero.
\end{lemma}

\noindent We will apply the lemma in the special case where
$\rho_n$ has support on $\R_+$,  $\int x\rho_n(dx)$ is bounded uniformly in $n$,
and  $b_n$ is given in terms of $\rho_n$ as $b_n=\int (|x|\wedge 1)\rho_n(dx)$.
To facilitate the application in this case, we prove the following, auxiliary lemma.

\begin{lemma}\label{lem:bn}
Let $\rho_n$ be a sequence of measures on $\R_+$
such that \newline $\limsup_{n\to\infty}\int x\rho_n(dx)<\infty$,  let
 $b_n=\int (x\wedge 1)\rho_n(dx)$, and let
  \[
a_-=\lim_{\eps\to 0}\liminf_{n\to\infty} \int_0^\eps x\rho_n(x)dx
\qquad
a_+=\lim_{\eps\to 0}\limsup_{n\to\infty} \int_0^\eps x\rho_n(x)dx
\]
Then $\int f \dif \rho_n \to \int f \dif \rho$
for all bounded continuous functions $f$ vanishing in a neighborhood of zero if and only if $\rho_n$ converges vaguely to $\rho$.  Furthermore, if
 $\rho_n\to\rho$  vaguely then $b_n$  converges to some $b$ if and only if
$a_-=a_+$,
in which case
$
b=\int_0^\infty(x\wedge 1)\rho(dx)+a_+$.
\end{lemma}

\begin{proof}
Restricting ourself to large enough $n$, we may w.l.o.g assume that
$\int x\rho_n(\dif x)\leq 2c$.  Combined with the fact that $\int x\rho(\dif x)\leq\limsup_{n\to\infty}\int x\rho_n(\dif x)\leq c$, we conclude that  $\rho_n[[M,\infty))\leq c/M$ and $\rho([M,\infty))\leq 2c/M$.  Thus convergence for all bounded, continuous functions
is equivalent to vague convergence.

Let's now assume that $\rho_n$ is vaguely convergent to $\rho$, and let
$0<\eps\leq 1$ be such that $\rho$ has not atom at $\eps$.
Then
$\int_\eps^\infty(x\wedge 1)\rho_n(dx)$ converges to
$\int_\eps^\infty(x\wedge 1)\rho(dx)$,
showing that
\[
b_n=\int_\eps^\infty(x\wedge 1)\rho(dx)+\int_0^\eps (x\wedge 1)\rho_n(dx)+o(1).
\]
This implies that
\[
\liminf_{n\to\infty}b_n=\int_0^\infty(x\wedge 1)\rho(dx)+a_-\quad\text{and}\quad
\limsup_{n\to\infty} b_n=\int_0^\infty(x\wedge 1)\rho(dx)+a_+,
\]
which completes the proof.
\end{proof}

\section{Properties of limiting adjacency measures}\label{sec:properties}
In
this appendix, we calculate the finite dimensional distributions of  random adjacency measures corresponding to the graphexes in Theorem \ref{thm:main-CM}, Theorem \ref{thm:GRG-limit}, and Theorem~\ref{thm:main-BCM}. These are used extensively in the respective proofs.

\paragraph*{Configuration model.}
Let $\xi_{\sss \mathrm{CM}}$ denote the random adjacency measure associated to the multigraphex $\cW_{\sss \mathrm{CM}} = (W_{\sss \mathrm{CM}}, S_{\sss \mathrm{CM}} , I_{\sss \mathrm{CM}})$ and $\xi^*_{\sss \mathrm{CM}}:=\xi_{\sss \mathrm{CM}}\vert_{(x,y):y\leq x}$.

Then we have the following:
\begin{lemma}\label{lem:xi_WCM}
For any $A,B \in \sB(\R_+)$  with $A\cap B = \varnothing$, the conditional distribution of $\xiCM^*(A\times A)$, conditional on $\{(\theta_i,v_i)\}_{i\geq 1}$, is $\mathrm{Poisson}(\mu(A)^2/2)$ and that of $\xiCM(A\times B)$ is $\mathrm{Poisson}(\mu(A)\mu(B))$.
Moreover, for a disjoint collection $(B_i)_{i=1}^k$,  conditionally on $\{(\theta_i,v_i)\}_{i\geq 1}$, $(\xiCM^*(B_i\times B_i))_{i\in [k]},(\xiCM (B_i\times B_j))_{1\leq i \leq  j\leq k})$ is an independent collection.
\end{lemma}
\begin{proof}
Let $\{(\theta_i, v_i)\}$ be a unit rate Poisson process on $\mathbb{R}^2_+$ and set  $w_i:= \bar{\rho}^{-1}(v_i)$.
Now, conditionally on $\{(\theta_i,v_i)\}_{i\geq 1}$,
\begin{align*}
&\xiCM^*(A\times A) = \sum_{i>j} \poi(w_iw_j) \ind{\theta_i\in A, \theta_j\in A}+ \sum_{i} \poi(w_i^2/2) \ind {\theta_i\in A} \\
&\hspace{.2cm}+ \sum_{j,k} \ind{\chi_{jk}\leq a w_j} \ind{\theta_j\in A, \sigma_{jk}\in A} + \sum_k \ind{\eta_k''\leq a^2/2} \ind{\eta_k\in A, \eta_k'\in A}\\
& = \sum_{i>j} \poi(w_iw_j) \ind{\theta_i\in A, \theta_j\in A}+ \sum_{i} \poi(w_i^2/2) \ind {\theta_i\in A} \\
&\hspace{1cm}+ \sum_{j} \poi(a \Lambda(A) w_j) \ind{\theta_j\in A} + \poi(a^2\Lambda(A)^2/2),
\end{align*}where, by construction, all the $\poi(\cdot)$ random variables above are mutually independent.
Therefore,
\begin{align*}
\xiCM^*(A\times A) &= \poi \bigg(\frac{a^2\Lambda(A)^2}{2}+a\Lambda(A)\sum_{i\geq 1}w_i\ind{\theta_i\in A}+\frac{1}{2}\Big(\sum_{i\geq 1}w_i\ind{\theta_i\in A}\Big)^2\bigg) \\
&= \poi(\mu(A)^2/2).
\end{align*}
Similarly, conditionally on $(w_i,\theta_i)_{i\geq 1}$,
\begin{eq}
&\xiCM(A\times B) = \sum_{i\neq j} \poi(w_iw_j) \ind{\theta_i\in A, \theta_j\in B} \\
&\hspace{.25cm}+ \sum_{j,k} \ind{\chi_{jk}\leq a w_j} \ind{\theta_j\in A, \sigma_{jk}\in B}  + \sum_{j,k} \ind{\chi_{jk}\leq a w_j} \ind{\theta_j\in B, \sigma_{jk}\in A} \\
& \hspace{.5cm}+\sum_k \ind{\eta_k''\leq a^2/2} \ind{\eta_k\in A, \eta_k'\in B} +\sum_k \ind{\eta_k''\leq a^2/2} \ind{\eta_k\in B, \eta_k'\in A}\\
&= \sum_{i\neq j} \poi(w_iw_j) \ind{\theta_i\in A, \theta_j\in B}+ \sum_{j} \poi(a \Lambda(A) w_j) \ind{\theta_j\in B} \\
&\hspace{1cm} + \sum_{j} \poi(a \Lambda(B) w_j) \ind{\theta_j\in A} + \poi(a^2\Lambda(A)\Lambda(B)),
 \end{eq}
 and thus
 \begin{eq}
 \xiCM(A\times B) &= \poi \bigg(\Big(a\Lambda(A)+\sum_{i\geq 1}w_i\ind{\theta_i\in B}\Big)\times\Big(a\Lambda(B)+\sum_{i\geq 1}w_i\ind{\theta_i\in A}\Big)\bigg) \\
 &= \poi(\mu(A)\mu(B)).
\end{eq}
The stated conditional independence follows by construction.
\end{proof}

\paragraph*{Generalized Random Graphs.}
Let $\xiGRG$ denote the random adjacency measure associated to the graphex $\cW_{\sss \mathrm{GRG}}^1 $ in Theorem~\ref{thm:GRG-limit},  and $\xi^*_{\sss \mathrm{GRG}}:=\xi_{\sss \mathrm{GRG}}\vert_{(x,y):y\leq x}$.
We fix any $k\geq 1$ and let $(B_{i})_{i\in [k]}$ be a collection of disjoint intervals such that $B_{i+1}$ lies to the left of $B_i$ on $\R_+$.
Let $N$ denote the Poisson point process on $\R_+^2$ with intensity $\rho(\dif w) \otimes \dif \theta$.
Further let $N_{i\varepsilon} = N([\varepsilon,\infty)\times B_i)$.
\begin{lemma} \label{lem:expr:lim-no-point}
\begin{align}\label{eq:expression-limit-no-point}
&\PR(\xiGRG^*(B_i\times B_j)=0, 1 \leq i \leq j \leq k) = \lim_{\varepsilon \to 0} \E[ G(N_{1\varepsilon}, \cdots, N_{k \varepsilon})] \\
&G(l_1, \cdots, l_k) =\int \prod_{1 \leq i \leq j \leq k}f_{ij}(\bf w)\prod_{i\in [k],\ell\in [l_{i}]}\1\{w_\ell^{\sss (i)}\in [\varepsilon,\infty)\}\frac{\rho(\dif w_\ell^{\sss (i)})}{\rho([\varepsilon, \infty))},
\end{align}
where $\bf w$ is the collection of random variables $(w_\ell^{\sss (i)})_{i\in [k],\ell\in [l_{i}]}$ and
\begin{align}
f_{ii}(\bf w) &= \e^{-a^2\Lambda(B_i)^2/2} \prod_{1\leq u\leq v\leq l_i} \frac{1}{1+w_{u}^{\sss (i)}w_{v}^{\sss (i)}} \e^{-a\Lambda(B_i) \sum_{u\in [l_i]}w_{u}^{\sss (i)}}, \nonumber\\
f_{ij}(\bf w) &= \e^{-a^2\Lambda(B_i)\Lambda(B_j)} \prod_{u\in [l_i],v\in [l_j]}\frac{1}{1+w_{u}^{\sss (i)}w_{v}^{\sss (j)}} \e^{-a\Lambda(B_i) \sum_{u\in [l_j]}w_{u}^{\sss (j)}-a\Lambda(B_j) \sum_{u\in [l_i]}w_{u}^{\sss (i)}}. \nonumber
\end{align}
\end{lemma}
\begin{proof}
Fix $\varepsilon>0$ and note that, conditional on $N_{i\varepsilon} = k_i$ for all $i\in [k]$, the collection $(w_j:\theta_j \in \cup_{i\in [k]}B_i)$ can be considered as i.i.d. samples from the normalized measure $\rho\vert_{[\varepsilon,\infty)}$.
Fix $1 \leq i <j \leq k$. Given $(\theta_u^{(i)}, w_u^{(i)})$ and $(\theta_v^{(j)}, w_v^{(j)})$ with $\theta_u^{(i)} \in B_i$, $\theta_v^{(j)} \in B_j$, $w_u^{(i)} , w_v^{(j)} \geq \varepsilon$, the adjacency measure $\xi_{\sss \mathrm{GRG}}$ has a point at $(\theta_u^{(i)}, \theta_v^{(j)})$ with probability $\frac{w_u^{(i)} w_v^{(j)}}{ 1 + w_u^{(i)} w_v^{(j)}}$. Moreover, these points are independent given $\{(\theta_u^{(i)}, w_u^{(i)}) : 1 \leq u \leq l_i \}$ and  $\{(\theta_v^{(j)}, w_v^{(j)}) : 1 \leq v \leq l_j \}$ variables. Further, \eqref{eq:def-graphex} implies that the set $B_i \times B_j$ has an independent
$\textrm{Poi}(a^2 \Lambda(B_i) \Lambda(B_j) + a \Lambda(B_i) \sum_{v \in [l_j]} w_v^{(j)} + a \Lambda(B_j )\sum_{u \in [l_i] } w_u^{(i)})$ points from points $\{(\theta, w)\}\in N$ with $w \geq \varepsilon$. Using independence of the contributions, the probability of having zero points in the adjacency measure is precisely $f_{ij}(\bf w)$. Further, we note that given $N$, the edges are all independent. This directly motivates the RHS of \eqref{eq:expression-limit-no-point}. Finally, we let $\varepsilon \downarrow 0$ to get the desired equality. The argument for $f_{ii}(\bf w)$ is similar, and is therefore omitted.
\end{proof}

\paragraph*{Bipartite Configuration model}
The limiting adjacency measure in case of Bipartite configuration model is given by $\xi_{\sss \cW_{\sss \mathrm{BCM}}}$, where $\cW_{\sss \mathrm{BCM}}$ is defined in \eqref{multigraphon-BCM}.
Let $\{(\theta_i, v_i,c_i)\}_{i\geq 1}$ be a unit rate Poisson process on $\mathbb{R}^2_+\times \{0,1\}$ and set  $w_i:= \bar{\rho}_j^{-1}(v_i)$ if $c_i = j$.
For $r=1,2$, define the completely random measure $\mu_r := a_r\Lambda+ \sum_{i\geq 1} \bar{\rho}_r^{-1}(v_i) \delta_{\theta_i}\ind{c_i = r}$.
We will show the following:
\begin{lemma}
For any  Borel subsets $A,B$ of $\mathbb{R}$ with $A\cap B = \varnothing$, the conditional distribution of $\xi_{\sss \cW_{\sss \mathrm{BCM}}}^*(A\times A)$, conditional on $\{(\theta_i,v_i,c_i)\}_{i\geq 1}$, is $\mathrm{Poisson}(\mu_1(A)\mu_2(A))$ and that of $\xi_{\sss \cW_{\sss \mathrm{BCM}}}(A\times B)$ is $\mathrm{Poisson}(\mu_1(A)\mu_2(B)+ \mu_2(A)\mu_1(B))$.
Moreover, for a disjoint collection $(B_i)_{i=1}^k$,  conditionally on $\{(\theta_i,v_i)\}_{i\geq 1}$, $(\xi_{\sss \cW_{\sss \mathrm{BCM}}}^*(B_i\times B_i))_{i\in [k]},$ $ (\xi_{\sss \cW_{\sss \mathrm{BCM}}} (B_i\times B_j))_{1\leq i < j\leq k})$ is an independent collection.
\end{lemma}
\begin{proof}
Note that conditionally on $\{(\theta_i,v_i,c_i)\}_{i\geq 1}$,
\begin{eq}
&\xi_{\sss \cW_{\sss \mathrm{BCM}}}^*(A\times A) = \sum_{i>j} \poi(w_iw_j) \ind{c_i\neq c_j}\ind{\theta_i\in A, \theta_j\in A} \\
&\hspace{.2cm}+ \sum_{j,k} \sum_{r=0,1}\ind{\chi_{jk}\leq a_r w_j} \ind{c_j \neq r} \ind{\theta_j\in A, \sigma_{jk}\in A} \\
&\hspace{.4cm}+ \sum_k \ind{\eta_k''\leq a_1a_2} \ind{\eta_k\in A, \eta_k'\in A}\\
& = \sum_{i>j} \poi(w_iw_j) \ind{c_i\neq c_j}\ind{\theta_i\in A, \theta_j\in A} \\
&\hspace{1cm}+ \sum_{j}\sum_{r=0,1} \poi(a_r \Lambda(A) w_j) \ind{c_j\neq r}\ind{\theta_j\in A} + \poi(a_1a_2\Lambda(A)^2),
\end{eq}where, by construction, all the $\poi(\cdot)$ random variables above are mutually independent.
Therefore,
\begin{align*}
\xi_{\sss \cW_{\sss \mathrm{BCM}}}^*(A\times A) &= \poi \bigg(a_1a_2\Lambda(A)^2+\sum_{r=0,1}a_r\Lambda(A)\sum_{i\geq 1}\ind{c_i\neq r}w_i\ind{\theta_i\in A} \\
&\hspace{.2cm}+ \sum_{i>j}w_iw_j\ind{c_i\neq c_j}\ind{\theta_i\in A,\theta_j\in A}\bigg) \\
&= \poi(\mu_1(A)\mu_2(A)).
\end{align*}
Similar argument can be carried out for $A,B\in \sB(\R_+)$ with $A\cap B = \varnothing$ to conclude that
 \begin{eq}
 \xi_{\sss \cW_{\sss \mathrm{BCM}}}(A\times B) &= \poi(\mu_1(A)\mu_2(B)+ \mu_2(A)\mu_1(B)).
\end{eq}
The stated conditional independence follows by construction.
\end{proof}

\section{Rescaling of a graphon process}\label{appendix:rescaling}

\begin{lemma}[Rescaling lemma]\label{lemma:rescaling}
Given a sequence of multigraphs $(G_n)_{n\geq 1}$ and real numbers $(\ell_n)_{n\geq 1}$, suppose that $\lim_{n\to\infty}\frac{2e(G_n)}{\ell_n} = c>0$.
Further, let $\lbl_{\sqrt{\ell_n}}(G_n) \xrightarrow{\sss d} \xi_{\sss \cW}$ for some multigraphex $\cW = (W,S,I)$.
Then $\lbl(G_n) \xrightarrow{\sss d} \xi_{\sss \cW^{c}}$, $\cW^c = (W^c, S^c, I^c)$ with
$$W^c(x,y, \cdot) = W(\sqrt{c}x,\sqrt{c}y, \cdot), \quad S^c(x,\cdot) = \frac{1}{\sqrt{c}}S(\sqrt{c}x,\cdot), \quad \text{and} \quad I'(\cdot) = \frac{I(\cdot)}{c}.$$
\end{lemma}
\begin{proof}
Define the point process $\xi'$ by  $\xi'([0,s]\times [0,t]) = \xi_{\sss \cW}([0,c^{-1/2}s]\times [0,c^{-1/2}t])$, for any $0<s,t<\infty$.
 First, let us show that, as $n\to\infty$,
\begin{eq}\label{lim-rescaled}
\lbl(G_n) (B) \dto \xi' (B),
\end{eq} where $B$ is any finite union of rectangles.
For simplicity, let us take $B = [0,t]^2$; the general case follows similarly.
Let $U_1, \cdots , U_n \sim \mathrm{Uniform}([0, \sqrt{2 e(G_n)} ])$ be iid random variables.
Further, define $U_i' = (\frac{\ell_n}{2e(G_n)})^{1/2}U_i$, $1 \leq i \leq n$, so that, $U_1' ,\cdots, U_n'$ are iid samples from $\mathrm{Uniform}([0, \sqrt{\ell_n}])$. Thus, we have,
\begin{align}
\lbl (G_n)([0,t]^2) &= \sum_{\{i,j\} \in E_n} \ind{U_i \leq t, U_j \leq t} ,\nonumber\\
\lbl_{\sqrt{\ell_n}} (G_n)([0,t]^2) &= \sum_{\{i,j\} \in E_n} \ind{U_i' \leq t, U_j' \leq t} \nonumber \\
&= \sum_{\{ i,j\} \in E_n } \mathbbm{1}\bigg\{U_i \leq \sqrt{\frac{2e(G_n)}{\ell_n}} t, U_j \leq \sqrt{\frac{2e(G_n) }{\ell_n}} t\bigg\} \nonumber
\end{align}
Let $t' = c^{-1/2} t$.
It now follows that
\begin{align}
&\E\Big[\Big| \lbl (G_n)([0,t]^2) -\lbl_{\sqrt{\ell_n}} (G_n)([0,t']^2) \Big| \Big] \to 0. \nonumber
\end{align}
 The proof of \eqref{lim-rescaled} now follows using the the assumption that $\lbl_{\sqrt{\ell_n}}(G_n) \dto \xi_{\mathcal{W}}$.

Next, we need to show that $\xi' \stackrel{d}{=} \xi_{\mathcal{W}^c}$.
Recall Definition~\ref{defn:adj_multigraphex} with all related notations.
Thus the multigraphon part in $\xi'$ is given by
\begin{equation}
\sum_{i\neq j} \zeta_{ij} \delta_{(\sqrt{c}\theta_i,\sqrt{c}\theta_j)} + \sum_i\zeta_{ii} \delta_{(\sqrt{c}\theta_i,\sqrt{c}\theta_i)}.
\end{equation}
On the other hand, if $\{(\theta_i,v_i)\}_{i\geq 1}$ is a unit rate Poisson point process in $\R_+^2$, then in distribution, $(\sqrt{c}\theta_i,v_i)_{i\geq 1} $ is equal to $ (\theta_i,\sqrt{c}v_i)_{i\geq 1}$.
This gives the required rescaling of the multigraphon part.
The rescaling of the star and isolated parts can be dealt similarly using rescaling properties of Poisson point processes, and thus is omitted here.
\end{proof}

\end{appendix}


\bibliographystyle{apa}
\bibliography{BCDS17}

\end{document}